\documentclass[reqno, letterpaper]{amsart}

\usepackage{amssymb}

\usepackage[style=alphabetic, backend=bibtex]{biblatex}
\usepackage{hyperref, cleveref}
\hypersetup{
	colorlinks = true,
	linkcolor = blue,
	citecolor = magenta
}

\usepackage{tikz-cd}
\tikzcdset{
	diagrams={sep=small}			
}

\usepackage{multirow}				
\usepackage{mathtools}				

\usepackage{enumitem}
\setlist[1]{labelindent=\parindent}
\setlist[enumerate, 1]{label = \textnormal{(\arabic*)}, ref = \arabic*}
\setlist[enumerate, 2]{label = \textnormal{(\alph*)}, ref = \alph*}
\setlist[enumerate, 3]{label = \textnormal{(\roman*)}, ref = \roman*}

\makeatletter
\def\l@subsection{\@tocline{2}{0pt}{2.5pc}{5pc}{}}
\makeatother

\numberwithin{equation}{subsection}

\usepackage[OT2,T1]{fontenc}
\DeclareSymbolFont{cyrletters}{OT2}{wncyr}{m}{n}
\DeclareMathSymbol{\Sha}{\mathalpha}{cyrletters}{"58}

\usepackage{thmtools}

\declaretheorem[sibling=equation, style=definition]{definition}
\declaretheorem[sibling=equation]{theorem}
\declaretheorem[sibling=equation]{proposition}
\declaretheorem[sibling=equation]{lemma}
\declaretheorem[sibling=equation]{corollary}
\declaretheorem[sibling=equation, style=remark]{remark}
\declaretheorem[sibling=equation, style=remark]{example}

\declaretheorem[sibling=equation, style=definition, name=Proposition-Definition]{proposition-definition}

\let\emptyset\varnothing

\let\AA\undefined

\let\to\longrightarrow

\let\mapsto\longmapsto

\newcommand{\AA}{\mathbb{A}}

\newcommand{\GG}{\mathbb{G}}

\newcommand{\PP}{\mathbb{P}}
\newcommand{\QQ}{\mathbb{Q}}

\newcommand{\ZZ}{\mathbb{Z}}

\DeclareMathOperator{\Hom}{Hom}
\DeclareMathOperator{\Ext}{Ext}

\DeclareMathOperator{\Aut}{Aut}

\DeclareMathOperator{\id}{id}
\DeclareMathOperator{\im}{im}
\DeclareMathOperator{\coker}{coker}

\DeclareMathOperator{\ord}{ord}

\DeclareMathOperator{\PGL}{PGL}

\DeclareMathOperator{\Spec}{Spec}

\DeclareMathOperator{\Pic}{Pic}
\DeclareMathOperator{\NS}{NS}
\DeclareMathOperator{\Br}{Br}

\DeclareMathOperator{\Sing}{Sing}
\DeclareMathOperator{\pr}{pr}
\DeclareMathOperator{\red}{red}

\DeclareMathOperator{\Alb}{Alb}

\newcommand{\SheafHom}{\underline{\mathrm{Hom}}}

\newcommand{\SheafExt}{\underline{\mathrm{Ext}}}

\DeclareMathOperator{\St}{St}
\DeclareMathOperator{\alb}{alb}

\begin{document}
\title[Singular fiber of abelian fibrations]{Kodaira-type classification of singular fibers of some minimal abelian fibrations}

\author[Y.-J. Kim]{Yoon-Joo Kim}
\address{Department of Mathematics, Columbia University, New York, NY 10027, USA}
\address{HCMC, Korea Institute for Advanced Study, Seoul, South Korea}
\email{yk3029@columbia.edu}

\date{\today}

\begin{abstract}
	Let $X \to S$ be a minimal abelian fibration of relative dimension $n$ over a curve. We classify all possible singular fibers $X_s$ having $(n-1)$-dimensional ``abelian variety parts''. This generalizes Kodaira's work on elliptic fibrations, and Matsushita and Hwang--Oguiso's work on Lagrangian fibrations into a single framework. The classification is divided into three parts: semistable, unstable, and multiple. Multiple fibers are again divided into three types: semistable-like, mixed, and unstable-like.
\end{abstract}

\maketitle
\tableofcontents

\section{Introduction}
	Kodaira classified all possible fibers $X_s$ arising in a minimal elliptic surface $f : X \to S$. Besides smooth elliptic curves ($\mathrm I_0$), one encounters nodal rational curves ($\mathrm I_1$) and cycles of $r \ge 2$ copies of $\PP^1$'s ($\mathrm I_r$), called \emph{semistable fibers}. There are more fibers of types $\mathrm{II}$, $\mathrm{III}$, $\mathrm{IV}$, $\mathrm I_{r \ge 0}^*$, $\mathrm{II}^*$, $\mathrm{III}^*$, and $\mathrm{IV}^*$, called \emph{unstable fibers}. These are non-semistable but still manageable as they have a smooth point. Finally, nowhere-smooth fibers are called \emph{multiple fibers}. They are necessarily of the form $X_s = m \cdot X_{s, \red}$ and its reduction $X_{s, \red}$ is a semistable fiber of type $\mathrm I_r$ as above. Kodaira's classification played an important role in his study of compact complex surfaces \cite{kod60, kod63} and was later recovered in an algebraic setting by N\'eron \cite{neron:original_article}. See \cite[\S V.7]{bar-hulek-pet-van:surface} for nice pictorial descriptions and simplified proof.
	
	A higher-dimensional analog of a minimal elliptic fibration is a \emph{minimal abelian fibration}, a flat projective degeneration $f : X \to S$ of $n$-dimensional abelian varieties over a smooth curve $S$ that is relatively minimal in the sense of the minimal model program. A natural question is whether singular fibers of such fibrations admit a classification. In general this problem is intractable, and much of the existing literature instead focuses on their semistable reduction behavior via allowing ramified base change. This perspective is important in the compactification problems of the moduli stack $A_{n,d}$ of polarized abelian varieties.
	
	Note however that Kodaira's original intent was to understand the total space $X$ (surface in his case). In other words, if one's goal is to understand the total space $X$, it is preferable to analyze the fiber $X_s$ directly. Such attempts first arose in the study of Calabi--Yau varieties, but to our knowledge, the most successful results were obtained for algebraic symplectic varieties. In relative dimension $2$, Matsushita \cite{mat01} classified codimension $1$ singular fibers of Lagrangian fibrations on smooth symplectic fourfolds. Hwang--Oguiso \cite{hwang-ogu09, hwang-ogu11, hwang-ogu16} subsequently extended this to higher-dimensional Lagrangian fibrations. Remarkably, their results closely resemble Kodaira's classification, suggesting that a classification of singular fibers for higher-dimensional minimal abelian fibrations may be possible under suitable assumptions.
	
	It is later observed in \cite{kim24} (see also forthcoming work \cite{decat-kim-sch26}) that every Lagrangian fibration satisfies a condition called \emph{$\delta$-regularity}. Roughly speaking, this condition implies that every codimension $1$ fiber admits an \emph{abelian variety part} of dimension $\ge n-1$. In this paper, we develop a general framework for \emph{minimal and $\delta$-regular} abelian fibrations, and show that their singular fibers admit a complete classification, unifying the results of Kodaira, Matsushita, and Hwang--Oguiso. The following is the main theorem of this article.
	
	\begin{theorem} \label{thm:main}
		Let $S$ be a Dedekind scheme of characteristic $0$ with algebraically closed residue fields and $f : X \to S$ a minimal $\delta$-regular abelian fibration of relative dimension $n \ge 2$. Write $X_s = f^{-1}(s)$ for a closed fiber. Then
		\begin{enumerate}
			\item \textnormal{(Semistable: \Cref{thm:classification of semistable fibers})} Every semistable fiber $X_s$ is isomorphic to one of the following.
			\begin{enumerate}
				\item[\textnormal{($\mathrm I_0$)}] An abelian variety of dimension $n$.
				\item[\textnormal{($\mathrm I_1$)}] A $\PP^1$-bundle $E \to A$ over an abelian variety of dimension $n-1$ with two sections $0$ and $\infty$ identified.
				\item[\textnormal{($\mathrm I_r$)}] A cycle of $r \ge 2$ isomorphic copies of $\PP^1$-bundles $E_i \to A$ for $i = 1, \cdots, r$, whose sections $0_i$ and $\infty_i$ are successively identified by $\infty_i = 0_{i+1}$.
			\end{enumerate}
			
			\item \textnormal{(Unstable: \Cref{thm:classification of unstable fibers})} Every unstable fiber $X_s$ is isomorphic to $(C \times A)/G$, where $C$ is an unstable curve in Kodaira's classification (type $\mathrm {II}$, $\cdots$, $\mathrm{IV}^*$) and $G$ is a finite abelian group classified in \Cref{table:classification of G}.
			
			\item \textnormal{(Multiple: \Cref{thm:classification of multiple fiber i}, \ref{thm:classification of multiple fiber ii}, \ref{thm:classification of multiple fiber iii}, \ref{thm:classification of multiple fiber iv}, and \ref{thm:classification of multiple fiber v})} Every multiple fiber $X_s$ is of the form $X_s = m \cdot \frac{1}{m} X_s$, where $\frac{1}{m}X_s$ is isomorphic to one of the following.
			\begin{enumerate}
				\item[\textnormal{($\mathrm I_0, \, \mathrm I^1_1$)}] Type $\mathrm I_0$ and $\mathrm I_1 = \mathrm I_1^1$ fibers described in (1).
				\item[\textnormal{($\mathrm I_R^k$)}] A cycle of $\PP^1$-bundles $E_i \to A$ for $i = 1, \cdots, R \, (\ge 2)$ similar to type $\mathrm I_R$ in (1), but $E_i$ and $E_j$ are isomorphic when $i \equiv j \mod k$ for $k \mid R$.
				
				\item[\textnormal{($\mathrm I_0^+$)}] A $\ZZ/d$-quotient of an abelian variety of dimension $n$ by a free automorphism of order $d = 2$, $3$, $4$, or $6$.
				\item[\textnormal{($\mathrm I_R^\pm$)}] A $\ZZ/2$-quotient of a type $\mathrm I_{2R}^{m/2}$ or $\mathrm I_R^{m/2}$ fiber with $2 \mid m \mid 2R \ (R \ge 1)$ by a free involution sending $E_i$ to $E_{-i}$ or $E_{1-i}$, or its additional $G$-quotient for $G = \ZZ/2$, $\ZZ/4$, or $(\ZZ/2)^{\times 2}$.
				
				\item[\textnormal{($\mathrm{II}$, $\cdots$)}] Any unstable fiber in (2) except $\mathrm I_R^*$ and its quotients for $R \ge 1$.
				\item[\textnormal{($\mathrm I_R^*$)}] Similar to type $\mathrm I_R^*$ in (2) with $m \mid R$, but the $R+1$ non-reduced $\PP^1$-bundles $E_i$'s $(0 \le i \le R)$ are isomorphic only when $i \equiv \pm j \mod m$. Additionally, its $G$-quotient for $G = \ZZ/2$, $\ZZ/4$, or $(\ZZ/2)^{\times 2}$.
				\item[\textnormal{($\mathrm I_0^*$-a, $\cdots$)}] Six additional cases stated in \Cref{table:exceptional unstable-like fibers}.
			\end{enumerate}
			Moreover, possible multiplicities $m$ in each case are classified in \Cref{table:classification of multiple fibers}.
		\end{enumerate}
	\end{theorem}
	
	The readers can find the relevant definitions in \S \ref{sec:models and delta-regularity} (\Cref{def:delta-regular} for $\delta$-regularity and \ref{def:semistability} for semistability). Some details and subtleties are omitted in \Cref{thm:main} for the brevity of exposition. The more precise versions are referred to the body of the article. The interested readers can draw the pictures of the above possible singular fibers as in \cite[Figures~1--3]{mat01}.
	
	\begin{table}[h]
		\begin{tabular}{|c|c||c|c|} \hline
			Stability & Kodaira type & Restrictions & Properties \\ \hline\hline
			\multirow{3}{*}{semistable} & $\mathrm I_0$ & - & abelian variety \\ \cline{2-4}
			& $\mathrm I_1$ & - & integral \\ \cline{2-4}
			& $\mathrm I_r$ & $r \ge 2$ & reduced, \ reducible \\ \hline
			\multirow{4}{*}{unstable} & $\mathrm {II}, \ \mathrm{III}/2, \ \mathrm{IV}/3$ & - & integral \\ \cline{2-4}
			& $\mathrm{III}, \ \mathrm{IV}$ & - & reduced, \ reducible \\ \cline{2-4}
			& \begin{tabular}[x]{@{}c@{}} $\mathrm I_r^*, \ \mathrm I_r^*/2, \ \mathrm I_r^*/4, \ \mathrm{II}^*,$ \\ $\mathrm {III}^*, \, \mathrm {III}^*/2, \ \mathrm {IV}^*, \, \mathrm {IV}^*/3$ \end{tabular} & $r \ge 0$ & non-reduced, \ reducible \\ \hline
		\end{tabular}
		\caption{Classification of non-multiple central fibers $X_s$.}
		\label{table:classification of non-multiple fibers}
	\end{table}
	
	\begin{table}[h]
		\begin{tabular}{|c|c||c|c|} \hline
			\begin{tabular}[x]{@{}c@{}} Stability \\ (multiple) \end{tabular} & \begin{tabular}[x]{@{}c@{}} Kodaira type of $X_s$ \\ ($\times \, 1/m$) \end{tabular} & Restrictions & Kodaira type of $P_s$ \\ \hline\hline
			\multirow{3}{*}{semistable-like} & $\mathrm I_0$ & - & $\mathrm I_0$ \\ \cline{2-4}
			& $\mathrm I^k_R$ & \begin{tabular}[x]{@{}c@{}} $k \mid m, \ R = kr$ \\ $k, r \ge 1$ \end{tabular} & $\mathrm I_r$ \\ \hline
			
			\multirow{3}{*}{mixed behavior} & $\mathrm I_0^+$ & \begin{tabular}[x]{@{}c@{}} $d \mid m$ \\ $d=2,3,4,6$ \end{tabular} & \begin{tabular}[x]{@{}c@{}}Any unstable type but\\$\mathrm I_r^*, \, \mathrm I_r^*/2, \, \mathrm I_r^*/4$ for $r \ge 1$ \end{tabular} \\ \cline{2-4}
			& \begin{tabular}[x]{@{}c@{}} $\mathrm I_R^+, \, \mathrm I_R^+/2 , \, \mathrm I_R^+/4,$ \\ $\mathrm I_R^-, \, \mathrm I_R^-/2$ \end{tabular} & \begin{tabular}[x]{@{}c@{}} $m = 2k, \ R = kr$ \\ $k, r \ge 1$ \end{tabular} & $\mathrm I_r^*, \ \mathrm I_r^*/2, \ \mathrm I_r^*/4$ \\ \hline
			
			\multirow{7}{*}{unstable-like}  & $\mathrm{II}, \, \mathrm{II}^*$ & $m \equiv \pm 1 \mod 6$ & $\mathrm{II}, \, \mathrm{II}^*$ \\ \cline{2-4}
			& $\mathrm{III}, \, \mathrm{III}/2, \ \mathrm{III}^*, \, \mathrm{III}^*/2$ & $2 \nmid m$ & $\mathrm{III}, \, \mathrm{III}/2, \ \mathrm{III}^*, \, \mathrm{III}^*/2$ \\ \cline{2-4}
			& $\mathrm{IV}, \, \mathrm{IV}/3, \ \mathrm{IV}^*, \, \mathrm{IV}^*/3$ & $3 \nmid m$ & $\mathrm{IV}, \, \mathrm{IV}/3, \ \mathrm{IV}^*, \, \mathrm{IV}^*/3$ \\ \cline{2-4}
			& $\mathrm I_0^*, \ \mathrm I_0^*/2, \ \mathrm I_0^*/4$ & $2 \nmid m$ & $\mathrm I_0^*, \ \mathrm I_0^*/2, \ \mathrm I_0^*/4$ \\ \cline{2-4}
			& $\mathrm I_R^*, \ \mathrm I_R^*/2, \ \mathrm I_R^*/4$ & \begin{tabular}[x]{@{}c@{}} $2 \nmid m ,\ R = mr$ \\ $r \ge 1$ \end{tabular} & $\mathrm I_r^*, \ \mathrm I_r^*/2, \ \mathrm I_r^*/4$ \\ \cline{2-4}
			& Additional types & see \Cref{table:exceptional unstable-like fibers} & - \\ \hline
		\end{tabular}
		
		\caption{Classification of multiple central fibers $X_s$ (the multiplicity $m > 1$ is suppressed for notational simplicity). Here $P_s$ is the \emph{translation-automorphism group}, see \Cref{def:t-automorphism scheme} and \ref{def:Kodaira type of t-automorphism group}.}
		\label{table:classification of multiple fibers}
	\end{table}
	
	\begin{definition}
		In \Cref{thm:main}, the \emph{Kodaira type} of the central fiber $X_s$ in each case is as follows.
		\begin{enumerate}
			\item The \emph{Kodaira type} of a semistable fiber $X_s$ is $\mathrm I_r$ for $r \ge 0$.
			\item The \emph{Kodaira type} of an unstable fiber $X_s$ is denoted by (Kodaira type of $C$)$/$(order of $G$). See \Cref{table:classification of non-multiple fibers} for all possible unstable Kodaira types.
			\item The \emph{Kodaira type} of a multiple fiber $X_s$ is $m$ (the multiplicity) multiplied by the indicated type in \Cref{thm:main}(3). We say that type $m \cdot \mathrm I^k_R$ has a \emph{semistable-like} stability, type $m \cdot \mathrm I_R^{\pm}$ has a \emph{mixed} stability, and type $m \cdot \mathrm{II}$, $\cdots$, together with the additional cases have an \emph{unstable-like} stability.
		\end{enumerate}
	\end{definition}
	
	Our study of minimal $\delta$-regular abelian fibrations is intimately related to the study of abstract N\'eron models of abelian varieties. Consequently, \Cref{thm:main} yields some new results about N\'eron models of abelian varieties. Their $1$-dimensional versions are well-known since Kodaira and N\'eron. We highlight the items (2a) and (3), which reveal new higher-dimensional properties of N\'eron models.
	
	\begin{theorem} [\Cref{thm:splitting} and \ref{thm:classification of component group}] \label{thm:main 2}
		Let $P \to S$ be the N\'eron model of an abelian variety over its function field, and write $P_s$ for its closed fiber.
		\begin{enumerate}
			\item If the linear part of $P_s$ is $\GG_m$, then the N\'eron component group $\pi_0(P_s)$ is a cyclic group.
			\item If the linear part of $P_s$ is $\GG_a$, then the following holds.
			\begin{enumerate}
				\item The Chevalley exact sequence of its neutral component splits: $P^\circ_s = \GG_a \times A$ for an abelian variety $A$ of dimension $n-1$.
				\item The N\'eron component group $\pi_0(P_s)$ is either a cyclic group of order $\le 4$ or $(\ZZ/2)^{\times 2}$ (see \Cref{table:classification of Neron component group}).
			\end{enumerate}
			\item Let $\check P \to S$ be the N\'eron model of the dual abelian variety over the function field. If the linear part of $P_s$ has dimension $\le 1$, then $\pi_0(P_s)$ and $\pi_0(\check P_s)$ are isomorphic.
		\end{enumerate}
	\end{theorem}
	
	Compared to the results of Matsushita and Hwang--Oguiso, our result has an advantage that it applies to wider classes of abelian fibrations and has more complete descriptions (even for Lagrangian fibrations). Our work, however, was highly inspired by theirs. We also note that \Cref{thm:main} does not fully recover the results in \cite[Theorem~1.1]{hwang-ogu11}: in the paper, the symplectic form puts more constraints to $f$ and eliminates many cases in \Cref{thm:main}(3). Symplectic forms play no role in this article.

	\subsection{Strategy of the proof}
		Here is an overview of our strategy and the structure of this article. First, \S \ref{sec:models and delta-regularity} is devoted to some reviews of various different objects arising in the study of abelian fibrations: minimal abelian fibrations $f : X \to S$, N\'eron models $P \to S$ of abelian schemes (t-automorphism scheme), etc.
		
		\medskip
		
		Our proof proceeds in two big stages: non-multiple and multiple fibers. When $X_s$ is non-multiple, the core of the proof boils down to an intersection number computation, similar to the work of Kodaira and Hwang--Oguiso. To realize this idea, however, we will first need to understand the geometry of each irreducible component $E_i$ of $X_s = \sum_{i=1}^r a_i E_i$. Roughly speaking, $E_i$ tends to be a $\PP^1$-bundle over an $(n-1)$-dimensional abelian variety with some minor exceptions. This will be studied more precisely in \S \ref{sec:irreducible component}. In the next \S \ref{sec:shears and untangles}, we address two new phenomena in higher-dimensional abelian fibrations that did not arise in elliptic fibrations: we call these a \emph{shear} and \emph{tangle}. Again roughly, they measure the failure of $X_s$ being isomorphic to the product $C \times A$ of a Kodaira singular curve $C$ and an abelian variety $A$ of dimension $n-1$ (shear in the semistable case and tangle in the unstable case). Along the way, we prove the splitting of the t-automorphism group $P_s^\circ$ in the unstable case (\Cref{thm:splitting}) and introduce a technique that \emph{untangles} the unstable fibrations (\Cref{prop:untangle for unstable}). The main ingredient here is the result of Conrad and Laurent--Schr\"oer \cite{con:albanese, lau-sch24} about the existence of Albanese morphisms. In \S \ref{sec:classification of non-multiple fibers}, we collect the results in \S \S \ref{sec:irreducible component}--\ref{sec:shears and untangles} to imitate Kodaira's intersection number computation and classify non-multiple fibers $X_s$. As a consequence, the N\'eron component group $\pi_0(P_s)$ is classified as well.
		
		\medskip
		
		It remains to classify multiple fibers. The intersection number computation is insufficient to capture the more intricate geometry of multiple fibers. To overcome this, we will start utilizing ramified coverings $T \to S$ and a new minimal abelian fibration $g : Y \to T$ obtained by the base change of $f$. This is done in \S \ref{sec:general base change}. To study how $f$ and $g$ are related, we first recall the relation between their t-automorphism schemes $P$ and $Q$ following Edixhoven \cite{edi92}. A byproduct of these discussions is an alternative interpretation of the classification of unstable fibers through their semistable base change (\Cref{prop:semistable reduction theorem untangled ver}).
		
		In \S \ref{sec:classification of multiple fibers}, we classify multiple fibers. Let us only highlight the two starting points. The first observation well-known: it is easier to study $\frac{1}{m}X_s$ than $X_s$. More precisely, there exists a preferred base change $g : Y \to T$ of degree $m$ inducing an \'etale $\mu_m$-covering $Y_t \to \frac{1}{m}X_s$ \eqref{eq:q and i}. This reduces the study of $\frac{1}{m}X_s$ to that of a non-multiple fiber $Y_t$ with a free $\mu_m$-action. The second is to divide the proof by the type of $P$ and $Q$ (\Cref{lem:classification of pullback}). When $P$ and $Q$ are both semiabelian, we show that $Y_t$ is semistable and its $\mu_m$-action cyclically permutes its components. When $Q$ is semiabelian but $P$ is not, $Y_t$ is still semistable but the $\mu_m$-action reverses the orientation of its components. When $P$ and $Q$ are both non-semiabelian, we need to take a further semistable base change $h : Z \to U$ and study the interactions of three fibrations $X$, $Y$, and $Z$. All these cases behave slightly differently, so we need to analyze them one-by-one. In \S \ref{sec:examples}, we give examples of all types of fibers stated in \Cref{thm:main}. \S \ref{sec:appendix} collects some facts in scheme and algebraic group theory that are used in this article.
		
		\medskip
		
		Finally, let us briefly explain why unstable-like multiple fibers admit additional possibilities (in fact, type $\mathrm I^-_R$ is a variation of $\mathrm I^+_R$ in the same fashion). There are two technical reasons: the first is a non-splitting of a certain Galois group (\Cref{lem:multiple fiber iv:untangle of X} and \ref{lem:multiple fiber v:untangle of X}). The second is a discrepancy between the degree of the semistable base change in \Cref{prop:semistable reduction theorem untangled ver} and the order of the automorphism group of a certain elliptic curve $E$ (\Cref{lem:multiple fiber v:varphi-action smooth}).

	\subsection*{Acknowledgment}
		I would like to thank Johan de Jong, Daniel Huybrechts, Jun-Muk Hwang, J\'anos Koll\'ar, Yuchen Liu, Mirko Mauri, and Takumi Murayama for helpful discussions and answering my questions. I would especially like to thank Johan de Jong for explaining the notion of a conductor subscheme and clearing up my confusions on the Picard space, and Takumi Murayama for detailed explanations of his results.

\section{Models and delta-regularity} \label{sec:models and delta-regularity}
	Throughout this article, $S$ is always a connected Dedekind scheme whose residue field $k$ is algebraically closed of characteristic $0$.
	
	\begin{definition}
		An \emph{abelian fibration} over $S$ is a flat projective morphism $f : X \to S$ whose generic fiber is a torsor under an abelian variety.
	\end{definition}
	
	Fix a closed point $s \in S$ and denote its complement by
	\[ S_0 = S \setminus \{ s \} .\]
	We will use the subscript $-_0$ to denote the base change $- \times_S S_0$. For simplicity, we assume that every abelian fibration in this article is smooth over $S_0$ and has one possible singular fiber over $s$. The restriction $f_0 : X_0 \to S_0$ is an \emph{abelian torsor}, a smooth projective morphism whose fibers are torsors under abelian varieties. Equivalently, we can say that $X_0$ is a torsor under a (unique) abelian scheme
	\begin{equation} \label{eq:abelian scheme}
		P_0 = \Aut^{\circ}_{f_0} \to S_0 .
	\end{equation}
	Let us review some definitions and basic properties in this section.

	\subsection{Minimal models and t-automorphism schemes}
		Following \cite[(16)]{kol25:neron}, we drop the $\QQ$-factoriality axiom in the definition of minimal models.
		
		\begin{definition} \label{def:minimal abelian fibration}
			An abelian fibration $f : X \to S$ is \emph{minimal} (resp. \emph{klt-minimal}) if $X$ has terminal singularities (resp. log terminal singularities) and $K_X$ is numerically $f$-trivial.
		\end{definition}
		
		By the theorem of N\'eron--Raynaud \cite{neron:original_article, raynaud:neron}, every abelian scheme $P_0 \to S_0$ admits a unique \emph{N\'eron model} $P \to S$, a quasi-projective, smooth, and commutative group scheme over $S$ whose restriction over $S_0$ is isomorphic to $P_0$. For the definition of N\'eron model, see \cite{neron}. Its central fiber $P_s$ is a possibly disconnected commutative algebraic group over $k$, sitting in a split exact sequence (\Cref{lem:pi_0 sequence splits})
		\[ 0 \to P^\circ_s \to P_s \to \pi_0(P_s) \to 0 ,\]
		where $P^\circ_s$ is its \emph{neutral component} and $\pi_0 (P_s)$ is its \emph{(N\'eron) component group}. By \cite[$\mathrm{IV_B}$ Theorem~3.10]{SGA3I}, $P$ admits a unique open subgroup scheme $P^\circ \subset P$ whose central fiber $P^\circ \times_S \{ s \}$ is the neutral component $P^\circ_s$ of $P_s$ as above. We call $P^\circ$ the \emph{neutral component} of $P$.
		
		\begin{definition} \label{def:t-automorphism scheme}
			Let $f : X \to S$ be a klt-minimal abelian fibration. The \emph{translation-automorphism group scheme} (\emph{t-automorphism scheme} for short) of $f$ is the N\'eron model $P \to S$ of the abelian scheme $P_0 \to S_0$ in \eqref{eq:abelian scheme}.
		\end{definition}
		
		\begin{remark}
			We remark a slight notational discrepancy with \cite[Definition~24]{kol25:TS}: a translation-automorphism sheaf $Aut^\circ_{X/S}$ in loc. cit. corresponds to our $P^\circ$, which is the neutral component of the translation-automorphism scheme in our terminology. In other words, our notion of the translation-automorphism $P$ contains more data of $\pi_0(P)$. When $X$ is only klt-minimal, this additional datum may not be useful because the $P^\circ$-action in \Cref{thm:Po-action} below may not be extendable to $P$. However, when $X$ is minimal and $\delta$-regular, $P$ will act on $X$ by \Cref{prop:P-action existence}. Our definition follows \cite[Definition~4.1]{kim24}.
		\end{remark}
		
		We start with the following special cases of two general theorems.
		
		\begin{theorem}[{\cite[Theorem~16.5]{kol25:neron}}] \label{thm:minimal model}
			Every abelian fibration has a $\QQ$-factorial minimal model.
		\end{theorem}
		
		\begin{theorem} [{\cite[Theorem~47]{kol25:neron}}] \label{thm:Po-action}
			Let $f : X \to S$ be a klt-minimal abelian fibration and $P \to S$ its t-automorphism scheme. Then there exists a unique $P^\circ$-action on $X$ over $S$ extending the $P_0$-torsor structure on $X_0$.
		\end{theorem}
		
		\begin{proposition} \label{prop:P-action property}
			The $P^\circ$-action in \Cref{thm:Po-action} satisfies the following.
			\begin{enumerate}
				\item The $P^\circ$-action on $X$ is faithful.
				\item The $P^\circ$-action on $X' = X \setminus \Sing(X_s)$ is free.
				\item The $P^\circ_s$-action on $\Sing(X_s)$ has nontrivial linear stabilizer groups.
			\end{enumerate}
		\end{proposition}
		\begin{proof}
			(1) Consider the relative automorphism scheme $\Aut_f \to S$ of $f$. Let $M \subset \Aut_f$ be its \emph{main component}, which is a smooth, commutative, and closed subgroup scheme over $S$ (see \cite[Definition~2.36, Proposition~2.42]{kim24}). \Cref{thm:Po-action} constructs a group scheme homomorphism $P^\circ \to M \subset \Aut_f$. Since $M$ is smooth over $S$, the N\'eron mapping property constructs a homomorphism $M \to P$. The composition
			\[ P^\circ \to M \to P \]
			must be the open immersion $P^\circ \subset P$ again by the N\'eron mapping property, so this shows $P^\circ \to M$ is an open immersion.
			
			\medskip
			
			(2) There is nothing to prove when $X'_s = \emptyset$, so let us assume that $X' \to S$ is surjective. Since the statement is local on $S$, we may shrink $S$ \'etale locally and fix a section $S \to X'$. Consider the orbit map $P^\circ \to X' \subset X$ of this section. Similar as above, the N\'eron mapping property constructs a morphism $X' \to P$, whose composition $P^\circ \to X' \to P$ must be the open immersion $P^\circ \subset P$. This forces $P^\circ \to X'$ to be an open immersion, meaning that the action is free along the section. Vary the section to conclude $P^\circ$ acts on $X'$ freely.
			
			\medskip
			
			(3) Since the $P^\circ_s$-action on $X_s$ is faithful by (1), every stabilizer is linear by \Cref{lem:linear stabilizer}. To prove the remaining claim, assume that $x_s \in X_s$ is a closed point with trivial stabilizer. Take a quasi-finite flat morphism $T \to S$ with a section $x : T \to X$ through $x_s$ (e.g., take a general complete intersection of $f$-ample divisors or use \cite[$\mathrm {VI_B}$~Lemma~5.6.1]{SGA3I}). The section defines an orbit map $P_T \to X_T$ by $a \mapsto a \cdot x$, which is an open immersion because there are no stabilizers along the section $x$. In particular, the central fiber $X_s$ is regular at $x_s$.
		\end{proof}
		
		\begin{remark}
			The quasi-projectiveness of $P$, \Cref{prop:P-action property}, and \cite[Theorem~1.2]{anc-fra24} make the $P^\circ$-action on $X$ a \emph{weak abelian fibration}. Though important in the study of $Rf_*\QQ_\ell$, we will not discuss this topic in this article.
		\end{remark}

	\subsection{Delta-regularity}
		The notion of $\delta$-regularity can be defined over arbitrary bases following \cite[(7.1.5)]{ngo10}, but we restrict ourselves to Dedekind schemes. Given a smooth commutative group scheme $P \to S$ over a Dedekind scheme $S$, its fiber $P_s$ admits a \emph{Chevalley exact sequence} \cite[Theorem~4.2]{brion15}
		\[ 0 \to L_s \to P_s \to \tilde A_s \to 0 , \quad\mbox{or}\quad 0 \to L_s \to P^\circ_s \to A_s \to 0 ,\]
		where $L_s$ is a connected linear algebraic group, $\tilde A_s$ is a proper algebraic group, and $A_s = \tilde A_s^\circ$ is an abelian variety. We call $L_s$ the \emph{linear part} and $A_s$ the \emph{abelian variety part} of $P_s$.
		
		\begin{definition} \label{def:delta-regular}
			\begin{enumerate}
				\item A smooth commutative group scheme $P \to S$ is \emph{$\delta$-regular} if its generic fiber is an abelian variety, and its closed fiber $P_s$ has at most $1$-dimensional linear part $L_s$.
				\item A klt-minimal abelian fibration $f : X \to S$ is \emph{$\delta$-regular} if its t-automorphism scheme $P \to S$ is $\delta$-regular.
			\end{enumerate}
		\end{definition}
		
		In other words, the $\delta$-regularity condition says either $P_s$ is an abelian variety (in which case $L_s = 0$), or $L_s = \GG_m$ or $\GG_a$. The main interest of this article is an abelian fibration that is both minimal and $\delta$-regular. In this case, \Cref{thm:Po-action} has the following improvement. The rest of this subsection proves this.
		
		\begin{proposition} \label{prop:P-action existence}
			Let $f : X \to S$ be a minimal $\delta$-regular abelian fibration and $P \to S$ its t-automorphism scheme. Then
			\begin{enumerate}
				\item There exists a unique faithful $P$-action on $X$ extending the $P_0$-torsor structure on $X_0$. Its restriction to $P^\circ$ is \Cref{thm:Po-action}.
				\item If $X_s$ has a regular point, then $X' = X \setminus \Sing(X_s)$ is a $P$-torsor.
			\end{enumerate}
		\end{proposition}
		
		\begin{lemma} \label{lem:P-orbit has dimension n-1}
			Let $f : X \to S$ be a klt-minimal $\delta$-regular abelian fibration and $P$ its t-automorphism scheme. Then
			\begin{enumerate}
				\item Every $P^\circ_s$-orbit in $X_s$ has dimension $\ge n-1$.
				\item Every $P^\circ_s$-orbit in $X_s$ of dimension $n-1$ is an abelian torsor.
			\end{enumerate}
		\end{lemma}
		\begin{proof}
			Since the linear part of $P^\circ_s$ has dimension $\le 1$, every point $x \in X$ has a stabilizer with dimension $\le 1$ by \Cref{prop:P-action property}. If an orbit has dimension $n-1$, then it is isomorphic to the quotient of $P^\circ_s$ by the linear stabilizer, which is an abelian variety by the Chevalley sequence.
		\end{proof}
		
		\begin{proposition} \label{prop:minimal delta-regular abelian fibration is unique}
			\begin{enumerate}
				\item Every minimal $\delta$-regular abelian fibration $f : X \to S$ has a regular total space $X$.
				\item Every birational map between two minimal $\delta$-regular abelian fibrations uniquely extends to an isomorphism.
			\end{enumerate}
		\end{proposition}
		\begin{proof}
			(1) The singular locus $\Sing(X)$ is over $s$ and has dimension $\le n-2$, since a terminal scheme is regular in codimension $2$. Since $\Sing(X)$ is invariant under $P^\circ_s$-action and every orbit has dimension $\ge n-1$ by \Cref{lem:P-orbit has dimension n-1}, we have $\Sing(X) = \emptyset$.
			
			\medskip
			
			(2) Let $X$ and $Y$ be two minimal $\delta$-regular fibrations and $\phi : X \dashrightarrow Y$ a birational map over $S$. Then $\phi$ is $P^\circ$-equivariant and hence its undefined locus is closed under the $P^\circ_s$-action. Since $X$ and $Y$ are minimal, $\phi$ is an isomorphism in codimension $1$ points \cite[Corollary~3.54]{kol-mori}. Hence the undefined locus $Z \subset X_s$ of $\phi$ (resp. undefined locus $W \subset Y_s$ of $\phi^{-1}$) has dimension $\le n-1$ and is a disjoint union of abelian torsors by \Cref{lem:P-orbit has dimension n-1}. Moreover, $\phi$ induces an isomorphism $X \setminus Z \to Y \setminus W$. Take the graph $\Gamma \subset X \times_S Y$ of the birational map $\phi$. Then for every $x \in Z$, the fiber $\Gamma_x = \Gamma \cap \pr_1^{-1}(x) \subset Y_s$ is contained in $W$, an abelian torsor of dimension $n-1$. But $\Gamma_x$ is positive dimensional and rationally chain connected \cite[Theorem~VI.1.2]{kol:rational} \cite[Corollary~1.5]{hacon-mcker07}, producing rational curves in $W$. Abelian torsors contain no rational curves, so this shows $Z = \emptyset$. Thus $\phi$ is a morphism, or an isomorphism since $\phi^{-1}$ is a morphism, too.
		\end{proof}
		
		\begin{proof} [Proof of \Cref{prop:P-action existence}]
			(1) The following strategy is taken from \cite[\S 5]{kim24}. Take the relative automorphism scheme and its main component $M \subset \Aut_f$ as in \Cref{prop:P-action property}. The claim is equivalent to saying that $M$ is a N\'eron model. Thanks to \cite[Theorem~7.1.1]{neron}, it is enough to show that it is a \emph{weak N\'eron model} in the sense of \cite[Definition~3.5.1 and Proposition~3.5.6]{neron} (or \cite[\S 3.6]{kim24}). Shrinking $S$ \'etale locally, we need to show that every rational map $u : S \dashrightarrow M$ defined on $S_0$ is regular. Consider $u$ as a birational automorphism $\phi_u : X \dashrightarrow X$ over $S$. \Cref{prop:minimal delta-regular abelian fibration is unique} shows that $\phi_u$ is necessarily an automorphism, so $u$ uniquely extends to a section $S \to M$ and shows the weak N\'eron mapping property of $M$. Hence $M$ is the N\'eron model of $P_0$, and it is isomorphic to $P$.
			
			\medskip
			
			(2) Note that $X' = X \setminus \Sing(X_s)$ is a N\'eron model of $X_0 \to S_0$: this is a consequence of \cite[\S 1.3]{neron} (see \cite[Example~3.13]{kim24} for a summary) or its generalization in \cite[Corollary~2]{kol25:neron}. When $X_s$ has a regular point, $X'$ is a $P$-torsor by \cite[Corollary~6.5.3]{neron} or \cite[Proposition~3.21]{kim24}.
		\end{proof}

	\subsection{Multiplicity and semistability}
		The central fiber $X_s$ is an effective Cartier divisor
		\begin{equation} \label{eq:central fiber}
			X_s = \sum_{i=1}^r a_i E_i \qquad\mbox{for}\quad E_i : \mbox{prime divisor and } a_i \in \ZZ_{>0} .
		\end{equation}
		The prime Weil divisors $E_i$ will be considered as integral subschemes of $X$.
		
		\begin{definition}
			Let $f : X \to S$ be an abelian fibration from a normal scheme $X$ and $X_s$ its central fiber.
			\begin{enumerate}
				\item The \emph{multiplicity} of the component $E_i \subset X_s$ is the positive integer $a_i$ in \eqref{eq:central fiber}. We say $E_i$ is a \emph{reduced component} of $X_s$ if $a_i = 1$ and \emph{non-reduced component} of $X_s$ if $a_i > 1$.
				\item The \emph{multiplicity} of $X_s$ is the positive integer
				\[ m = \gcd \{ a_1, \cdots, a_r \} \ \in \ZZ_{>0} .\]
			\end{enumerate}
		\end{definition}
		
		\begin{lemma} [{\cite[Remark~1.3]{cam05}}] \label{lem:non-multiple fiber has reduced component}
			Let $f : X \to S$ be an abelian fibration from a regular scheme $X$. If $X_s$ has multiplicity $1$, then it has a reduced component.
		\end{lemma}
		\begin{proof}
			The following copies Campana's argument but using techniques in \cite[\S \S 25--26]{kol25:TS}. We may assume $S = \Spec R$ and $S_0 = \Spec K$ are the Spec of a strictly henselian dvr and its function field. Write $z$ for a uniformizing parameter of $R$ and $X_s = \sum_{i=1}^r a_i E_i$. For each $i$, take a point $x_i \in E_i$ such that it is a regular point of $(X_s)_{\red}$. Then the point $x_i \in X_s$ extends to a \emph{multisection} $x_i : S \to X$ of $f$ such that the composition $f \circ x_i : S \to X \to S$ is $z \mapsto z^{a_i}$. 
			
			\medskip
			
			Recall that $X_0 \to S_0$ is a $P_0$-torsor. Set $X^{(0)}_0 = P_0$ and $X^{(m)}_0 = X_0 \times_{p_m} S_0$ for each $m \in \ZZ_{\neq 0}$ where $p_m : S_0 \to S_0$ is a base change $z \mapsto z^m$. (Equivalently, if $X_0$ is a $P_0$-torsor defined by $\alpha \in H^1_{\acute et}(S_0, P_0)$, then $X^{(m)}_0$ is a torsor defined by $m \alpha$.) Their disjoint union
			\[ \mathfrak X_0 \coloneq \bigsqcup_{m \in \ZZ} X^{(m)}_0 \to S_0 \]
			form a commutative group scheme. Use the assumption $\gcd \{ a_i \} = 1$ to take $b_i \in \ZZ$ such that $\sum_i a_i b_i = 1$. The multisection $x_i$ defines a section of $X_0^{(a_i)} \subset \mathfrak X_0$, so we may take take their sum using the commutative group structure:
			\[ x = \sum_{i=1}^r b_i \cdot x_i : S_0 \to \mathfrak X_0 .\]
			This section lands into $X^{(1)}_0 = X_0$, because the degree of the multisection $x$ is $\sum_i a_i b_i = 1$. The valuative criterion extends it to a section $x : S \to X$. Such a section necessarily factors through $X' = X \setminus \Sing(X_s)$ by the regularity of $X$ \cite[Proposition~3.1.2]{neron} and hence $X'_s \neq \emptyset$. This means $X_s$ has a reduced component.
		\end{proof}
		
		To talk about semistability, recall that $\delta$-regularity implies one of the three possibilities: (1) $P_s$ has a linear part $0$, i.e., $P_s$ is an abelian variety, (2) $P_s$ has a linear part $\GG_m$, or (3) it has a linear part $\GG_a$. The first case is easy to analyze. In fact, we can immediately classify $X_s$ in this case.
		
		\begin{lemma} \label{lem:classification when P is smooth}
			Let $f : X \to S$ be a klt-minimal abelian fibration. If $P_s$ is an abelian variety, then $X_s = m \cdot X_{s, \red}$ for an abelian torsor $X_{s, \red}$ of dimension $n$ and $m \ge 1$.
		\end{lemma}
		\begin{proof}
			The $P_s$-action on $X_{s, \red}$ can have only finite stabilizers by \Cref{prop:P-action property}(2--3). Since $X_{s, \red}$ is proper connected, the only possibility is $X_{s, \red} \cong P_s / K$ for a finite subgroup $K \subset P_s$. Hence $X_{s, \red}$ is an abelian torsor.
		\end{proof}
		
		\begin{definition} \label{def:semistability}
			Let $f : X \to S$ be a minimal $\delta$-regular abelian fibration and $m$ the multiplicity of its central fiber $X_s$. Then $f$ (or $X_s$) is called
			\begin{itemize}
				\item \emph{semistable} (resp. \emph{strictly semistable}) if $m = 1$, and $P_s$ has a linear part either $0$ or $\GG_m$ (resp. only $\GG_m$);
				\item \emph{unstable} if $m = 1$ and $P_s$ has a linear part $\GG_a$; and
				\item \emph{multiple} if $m > 1$.
			\end{itemize}
		\end{definition}
		
		\begin{remark}
			We deviate from the usual convention and divide the non-semistable fibers into two categories: unstable and multiple. In other words, \emph{unstable fibers are always assumed to be non-multiple} in this article. To justify, we will see in \S \ref{sec:classification of multiple fibers} that a multiple fiber can behave similar to semistable fibers, unstable fibers, or even a mixture of them.
			
			We also note that the multiplicity of the central fiber $X_s$ is \emph{not} a birational invariant of $f$. See the following simple example. For this reason, we call the central fiber $X_s$ multiple (or non-multiple) only when the abelian fibration is \emph{minimal} $\delta$-regular, which is unique in its birational class by \Cref{prop:minimal delta-regular abelian fibration is unique}.
		\end{remark}
		
		\begin{example} \label{ex:counterexample for reduced component}
			Take a minimal elliptic surface $f : X \to S$ with a Kodaira type $\mathrm {IV}^*$ central fiber
			\[ X_s = (E_1 + 2F_1) + (E_2 + 2F_2) + (E_3 + 2F_3) + 3G .\]
			Contracting three $E_i$'s yields an elliptic fibration $\bar f : \bar X \to S$ from a canonical surface $\bar X$ with three $\mathrm A_1$ singularities. Its central fiber $\bar X_s = 2(\bar F_1 + \bar F_2 + \bar F_3) + 3\bar G$ has four components with multiplicities $2$, $2$, $2$, and $3$, respectively. Hence \Cref{lem:non-multiple fiber has reduced component} fails without the regularity assumption on $X$.
			
			This time, contract three chains $E_i \cup F_i$ to yield another elliptic fibration $\bar {\bar f} : \bar {\bar X} \to S$ from a canonical surface $\bar {\bar X}$ with three $\mathrm A_2$ singularities. Its central fiber $\bar {\bar X}_s = 3\bar {\bar G}$ is irreducible with multiplicity $3$. We will \emph{not} call it a multiple fiber because $\bar {\bar f}$ is not minimal. Indeed, the original minimal model $f$ had a non-multiple fiber $X_s$.
		\end{example}
		
		Semistable fibers are generally better understood than unstable/multiple fibers. The following is a consequence of Koll\'ar's characterization of semistable fibers.
		
		\begin{theorem} [{\cite[Theorem~42]{kol25:neron}}] \label{thm:Kollar semistable}
			Let $f : X \to S$ be a klt-minimal abelian fibration and $P$ its t-automorphism scheme. If $X_s$ has a reduced component and $P_s$ is semiabelian, then $X_s$ has slc singularities. In particular, it is reduced and demi-normal.
		\end{theorem}

	\subsection{Dual of the t-automorphism scheme} \label{sec:dual Neron model}
		Recall that the t-automorphism scheme $P$ was nothing but the N\'eron model of the abelian scheme $P_0$ in \eqref{eq:abelian scheme}.
		
		\begin{definition}
			Let $f : X \to S$ be a klt-minimal abelian fibration and $P$ its t-automorphism scheme. The \emph{dual} of $P$ is the N\'eron model $\check P \to S$ of the dual abelian scheme $\check P_0 \to S_0$.
		\end{definition}
		
		\begin{proposition} \label{prop:dual Neron model}
			Let $P_0$ be an abelian scheme over $S_0$ and $\check P_0$ its dual. Letting $P$ and $\check P$ be their N\'eron models over $S$, their central fibers $P_s$ and $\check P_s$ satisfy the following.
			\begin{enumerate}
				\item Their linear parts $L_s$ and $\check L_s$ are isomorphic. In particular, the semiabelian property and $\delta$-regularity of $P$ and $\check P$ are the same.
				\item If $P_s$ (or equivalently $\check P_s$) is semiabelian, then their abelian variety parts $A_s$ and $\check A_s$ are dual to each other.
			\end{enumerate}
		\end{proposition}
		\begin{proof}
			(1) The abelian scheme $P_0 \to S_0$ is projective by \cite[Theorem~XI.1.4]{ray:group_schemes}. Take any polarizations $\lambda_0 : P_0 \to \check P_0$ and $\check \lambda_0 : \check P_0 \to P_0$ that satisfy
			\[ \lambda_0 \circ \check \lambda_0 = [m]_{\check P_0} ,\qquad \check \lambda_0 \circ \lambda_0 = [m]_{P_0} .\]
			These homomorphisms and equalities extend over $S$ by the N\'eron mapping property. Hence $\lambda$ and $\check \lambda$ have quasi-finite kernels over $S$. As a result, $\lambda_s : P^\circ_s \to \check P^\circ_s$ is an isogeny and hence the claim follows from \Cref{lem:isogeny}.
			
			\medskip
			
			(2) This is \cite[IX~Theorem~5.4]{SGA7I} and \cite[\S II.2]{fal-chai:abelian}. We note that this statement fails when $P$ is not semiabelian. See \Cref{rmk:abelian varieties are not dual}.
		\end{proof}
		
		By \cite[Theorem~1]{kol-kov25} (or \Cref{lem:Kollar-Kovacs} in our case), $f_* \mathcal O_X = \mathcal O_S$ and $R^1f_*\mathcal O_X$ are locally free and their formation commutes with arbitrary base change. This implies the existence of a locally separated and locally finite type group algebraic space $\Pic_f \to S$ that is smooth along the identity section (\cite[Theorem~8.3.1 and 8.4.1]{neron} and \cite[Tag~0D2D]{stacks-project}). Its \emph{strict neutral component} $\Pic^\circ_f$ is a unique open subgroup space of $\Pic_f$ with connected fibers. Its \emph{weak neutral component} $\Pic^\bullet_f$ is the largest open subgroup space of $\Pic_f$ with a connected total space $\Pic^\bullet_f$. There exists a sequence of open immersions
		\[ \Pic^\circ_f \ \subset \ \Pic^\bullet_f \ \subsetneq \ \Pic_f \qquad\mbox{(the first inclusion is generally not an equality, too)} ,\]
		and where $\Pic^\circ_f$ and $\Pic^\bullet_f$ are smooth group algebraic spaces. We note that $\Pic_f$ is not smooth in general. See \cite[\S 2.2]{kim24} for more detailed discussions, where the notation $\Pic^{\circ\circ}_f \subset \Pic^\circ_f\subset \Pic_f$ is used instead. The following is a higher-dimensional generalization of \cite[Theorem~9.5.4]{neron}.
		
		\begin{proposition} \label{prop:dual Neron model is Picard}
			Let $f : X \to S$ be a minimal $\delta$-regular abelian fibration. Then
			\begin{enumerate}
				\item $\check P$ is isomorphic to the separated quotient of $\Pic^\bullet_f$.
				\item $\check P^\circ$ is isomorphic to the separated quotient of $\Pic^\circ_f$.
				\item There exists an isogeny $\kappa : (\Pic^\circ_f)_s = \Pic^\circ_{X_s} \to \check P^\circ_s$. If $f$ is non-multiple, then $\kappa$ is an isomorphism.
			\end{enumerate}
		\end{proposition}
		
		See \cite[\S 1.1]{hol19} and \cite[\S 2.3]{kim24} for the notion of a separated quotient.
		
		\begin{proof}
			(1) Since $X$ is regular by \Cref{prop:minimal delta-regular abelian fibration is unique}, the existence part of the N\'eron mapping property holds for $\Pic_f$ (and hence $\Pic^\bullet_f$). The uniqueness property fails due to the its non-separatedness, so once we take its separated quotient $\Pic^\bullet_f / E$, where $E$ is the closure of its identity section, the uniqueness holds as well. See \cite[Theorem~6.2]{hol19} and \cite[Lemma~6.12]{kim24}. This proves $\Pic^\bullet_f/E$ is the N\'eron model $\check P$. (2) The N\'eron mapping property constructs a homomorphism $\Pic^\circ_f \to \check P$, whose image is $\check P^\circ$ and kernel is the closure of the identity section of $\Pic^\circ_f$ because $\check P$ is separated.
			
			\medskip
			
			(3) Recall that $\Pic_f$, as an fppf abelian sheaf, is the sheafification of the fppf presheaf $T \mapsto \Pic X_T / \Pic T$ (e.g., \cite[\S 9.2]{kle:picard}). The fiber $(\Pic_f)_s$ over $T = \{ s \}$ is the sheafification of the same fppf presheaf applied to $X_s \to \{ s\}$, so $(\Pic_f)_s = \Pic_{X_s}$. Therefore, the fiber of (2) over $s$ yields an isogeny $\Pic^\circ_{X_s} \to \check P^\circ_s$ in general.
			
			\medskip
			
			Suppose now that $f$ is non-multiple. It is enough to show that $\Pic^\circ_f$ is separated, or its identity section is closed. Shrinking $S$, we may assume $f$ has a section by \Cref{lem:non-multiple fiber has reduced component}, and hence every $S$-section of $\Pic^\circ_f \subset \Pic_f$ is representable by a line bundle on $X$ by \cite[Theorem~9.2.5]{kle:picard}. To show the identity section is closed, let $L$ be a line bundle on $X$ that is $f$-trivial over $S_0$. Write $X_s = \sum_{i=1}^r a_iE_i$ for the central fiber and $L = \mathcal O_X(D)$ where $D = \sum_{i=1}^r b_i E_i$ is a divisor supported in $X_s$. Since $L$ is a section of $\Pic^\circ_f \subset \Pic^\tau_f$, it is numerically $f$-trivial over $S$.
			
			Let $Y \subset X$ be a complete intersection of $n-1$ very ample divisors of $X$, a regular surface $Y$ with a fibration over $S$ whose central fiber is $Y_s = \sum_{i=1}^r a_i F_i$ for a prime divisor $F_i$ in $Y$. The restriction of the line bundle is $\mathcal O_Y(D_Y)$ for a divisor $D_Y = \sum_i b_i F_i$. Since it is numerically trivial over $S$, Zariski's lemma for fibrations of surfaces (\cite[Lemma~III.8.2]{bar-hulek-pet-van:surface} or \cite[Theorem~9.1.23]{liu:ag}) applies and shows $D_Y = c \cdot Y_s$ for $c \in \QQ$. Hence $D = c \cdot X_s$. But note that $f$ was non-multiple and hence we may assume $a_1 = 1$ by \Cref{lem:non-multiple fiber has reduced component}. This means $c = b_1 \in \ZZ$ and hence the divisor $D$ is an integral multiple of $X_s$. Hence $L$ is $f$-trivial over $S$.
		\end{proof}
		
		Combining \Cref{prop:dual Neron model}--\ref{prop:dual Neron model is Picard}, the semistability of a minimal $\delta$-regular abelian fibration $f$ can be checked from the Picard scheme $\Pic_{X_s}$ of the central fiber.
		
		\begin{remark} \label{rmk:dual Neron model is Picard}
			By N\'eron--Raynaud, $\check P^\circ \to S$ is always a quasi-projective group scheme. By Artin, $\Pic^\circ_f \to S$ is always a locally separated group algebraic space of finite type. Let us point out a subtle fact that $\Pic^\circ_f$ may fail to be a scheme. More precisely, consider the following four conditions:
			\begin{enumerate}[label = \textnormal{(\roman*)}]
				\item $f : X \to S$ is non-multiple.
				\item $\Pic^\circ_f \to S$ is separated.
				\item $\Pic^\circ_f = \check P^\circ$.
				\item $\Pic^\circ_f$ is a scheme.
			\end{enumerate}
			We claim (i) $\Longrightarrow$ (ii) $\Longleftrightarrow$ (iii) $\Longleftrightarrow$ (iv). The implications (i) $\Longrightarrow$ (ii) $\Longleftrightarrow$ (iii) are noting but \Cref{prop:dual Neron model is Picard}(2--3). (iii) $\Longrightarrow$ (iv) is clear, and (iv) $\Longrightarrow$ (ii) is \cite[$\mathrm{VI_B}$~Corollary~5.5]{SGA3I}. We thank Johan de Jong for clearing up our original confusions. We will later prove in \Cref{lem:kernel of Pic to P check} that if $f$ has multiplicity $m$, then the isogeny $\kappa : \Pic^\circ_{X_s} \to \check P^\circ_s$ has $\ker \kappa \subset \ZZ/m$.
		\end{remark}
		
		\begin{example} \label{ex:non-separated Pic}
			In \Cref{ex:mI0}, a minimal $\delta$-regular abelian fibration $f : X \to S$ with a multiple central fiber $X_s = m \cdot (A/a)$ will be constructed. Its dual t-automorphism scheme is the constant abelian scheme $\check P = \check A \times S \to S$ and $(\Pic^\circ_f)_s = \Pic^\circ_{X_s} = (A/a)^\vee$. The homomorphism $(A/a)^\vee \to \check A$ in \Cref{prop:dual Neron model is Picard}(3) is only an isogeny but not an isomorphism. Indeed, its kernel is generated by the restriction of a line bundle $L = \mathcal O_X(\frac{1}{m} X_s)$, which is trivial over $S_0$ but only numerically trivial over $S$. As a result, $\Pic^\circ_f \to S$ is a non-separated group space and hence not a scheme.
		\end{example}

	\subsection{Weierstrass models} \label{sec:Weierstrass models}
		A minimal elliptic surface $f : X \to S$ with non-multiple central fiber admits a unique \emph{Weierstrass model} $\bar f : \bar X \to S$.\footnote{A more standard terminology is a minimal Weierstrass model. We omitted the word minimal to prevent potential confusions with minimal models.} It is a birational model of $f$ where the total space $\bar X$ has canonical singularities and the central fiber $\bar X_s$ is integral (e.g., \cite[\S 1]{artin-swi73}). One cheap way of constructing Weierstrass model is to use the Kodaira classification: $X_s$ is always a union of a reduced component and a chain of smooth rational curves with negative ADE intersection matrix. Contracting the latter, we obtain an elliptic fibration $\bar X \to S$ from a canonical surface $\bar X$ with an integral central fiber, the Weierstrass model of $f$.
		
		One can still contract a different set of rational curves and yield a different model, resulting a total space with possibly klt singularities. This motivates the following definitions.
		
		\begin{definition}
			A $\delta$-regular abelian fibration $f : X \to S$ is \emph{Weierstrass} (resp. \emph{klt-Weierstrass}) if $X$ is $\QQ$-factorial and has canonical singularities (resp. log terminal singularities), and the central fiber $X_s$ is integral (resp. irreducible).
		\end{definition}
		
		A Weierstrass (resp. klt-Weierstrass) model automatically satisfies $K_X \sim_f 0$ (resp. $K_X \sim_{f, \QQ} 0$) because the central fiber is integral (resp. irreducible). In particular, both are klt-minimal. The following are the main results of this subsection.
		
		\begin{proposition} \label{prop:klt Weierstrass model}
			Every $\delta$-regular abelian fibration has a klt-Weierstrass model.
		\end{proposition}
		
		\begin{proposition} \label{prop:Weierstrass model}
			Every minimal $\delta$-regular abelian fibration $f : X \to S$ with a non-multiple central fiber has a Weierstrass model $\bar f : \bar X \to S$. Moreover, $\bar X$ is a contraction of $X$ and is unique up to isomorphism (possibly after shrinking $S$).
		\end{proposition}
		
		\begin{proof}[Proof of \Cref{prop:klt Weierstrass model}]
			Start with a $\QQ$-factorial klt-minimal model $f : X \to S$ (e.g., take the minimal model in \Cref{thm:minimal model}).
			
			\medskip
			
			If the central fiber $X_s$ is irreducible, then $f$ is a klt-Weierstrass model. Suppose that $X_s = \sum_{i=1}^r a_i E_i$ has $r \ge 2$ components. Take any irreducible component $E = E_1$. Then there exists another component $E'$ intersecting with $E$. Take a curve $C$ that is contained in $E$, passing through a point in $E \cap E'$, but not contained in any other components $E_i$ for $i \ge 2$. This makes the intersection $(E \cdot C)$ negative, because we have $(X_s \cdot C) = 0$ and $(E' \cdot C) > 0$, so
			\begin{equation} \label{eq:intersection negative}
				(E \cdot C) = -\frac{1}{a_1} \sum_{i=2}^r a_i (E_i \cdot C) < 0 .
			\end{equation}
			
			Fix a rational number $0 < \varepsilon \ll 1$ such that the pair $(X, \varepsilon E)$ is klt \cite[Corollary~2.35]{kol-mori}. Since the $\QQ$-divisor $K_X + \varepsilon E$ is not $f$-nef by \eqref{eq:intersection negative}, we can apply the relative cone/contraction theorem (Theorem~3.25 in op. cit.). First, the cone theorem guarantees the existence of a $(K_X + \varepsilon E)$-negative rational curve $R$ defining an extremal ray in the relative Mori cone. Such a rational curve $R$ must be contained in $E$. We will soon prove in \Cref{cor:rational curve in E} that in this case, the normalization $\tilde E$ admits a $\PP^1$-bundle $\tilde p : \tilde E \to A$ over an abelian torsor $A$ of dimension $n-1$ and moreover, every rational curve in $E$ is the image of a fiber of $\tilde p$. Therefore, there is only one $(K_X + \varepsilon E)$-extremal ray (defined by $R$) and its contraction morphism
			\[ \phi : X \to \bar X \]
			contracts each rational curve in $E$ to a point. This means $\phi$ is a divisorial contraction of $E$ and $\phi(E) \cong A$. This yields a $\QQ$-factorial klt pair $(\bar X, \phi_*E = 0)$ by Proposition~3.36 and 3.43 in op. cit. Finally, take the identity
			\[ K_X \sim_{f, \QQ} \phi^* K_{\bar X} + a E \qquad\mbox{for}\quad a \in \QQ .\]
			Since $(K_X \cdot R) = 0$, we have $a = 0$ and hence $K_{\bar X}$ is numerically $\bar f$-trivial as well. This means $\bar f : \bar X \to S$ is again klt-minimal, so we may run the same process until we reach an irreducible central fiber $X_s$.
		\end{proof}
		
		\begin{proof}[Proof of \Cref{prop:Weierstrass model}]
			For the existence, use \Cref{lem:non-multiple fiber has reduced component} to fix a reduced component of $X_s$. Contract all other components using argument of \Cref{prop:klt Weierstrass model}: this yields a contraction to a klt-Weierstrass model $X \to \bar X$ with an integral central fiber $\bar X_s = E$. Take a canonical divisor $K_{\bar X}$ of the form $aE$ for $a \in \ZZ$. Since $E = X_s$ is Cartier, $K_{\bar X}$ is Cartier and hence $\bar X$ is canonical.
			
			\medskip
			
			Let $\bar X$ be a Weierstrass $\delta$-regular abelian fibration over $S$ and $X$ its minimal model. We claim the birational map $\phi : X \dashrightarrow \bar X$ is a morphism. First, both $X$ and $\bar X$ admit $P^\circ$-action by \Cref{thm:Po-action} and $\phi$ is equivariant. Take a section $S \to \bar X$ through a regular point of $\bar X_s$ and extend it to a section $S \to X$ via valuative criterion. This picks up a reduced component $E_i \subset X_s$ corresponding to $\bar X_s$. Contracting all other components yields another Weierstrass model $X \to \bar X'$ and a birational map $\bar \phi : \bar X' \dashrightarrow \bar X$ that generically identifies the central fibers. The undefined loci of $\bar \phi$ and $\bar \phi^{-1}$ are disjoint unions of abelian torsors by \Cref{lem:P-orbit has dimension n-1}. Since they do not contain rational curves, both $\bar \phi$ and $\bar \phi^{-1}$ are morphisms.
			
			\medskip
			
			Finally, let $\bar \phi : \bar X \dashrightarrow \bar Y$ be a birational map between two Weierstrass $\delta$-regular abelian fibrations. Their minimal models $X \to \bar X$ and $Y \to \bar Y$ are necessarily isomorphic via $\phi : X \to Y$ by \Cref{prop:minimal delta-regular abelian fibration is unique}. Say $E \subset X_s$ and $F \subset Y_s$ are the reduced components corresponding to $\bar X_s$ and $\bar Y_s$. Then there exists a t-automorphism of $Y$ sending $\phi(E)$ to $F$ by \Cref{prop:P-action existence} (possibly after shrinking $S$). Replacing $\phi$ by its composition with this t-automorphism, we may assume $\phi(E) = F$, or $\bar \phi : \bar X \dashrightarrow \bar Y$ is isomorphic in codimension $1$. Again this extends to an isomorphism.
		\end{proof}

\section{Irreducible components of the central fiber} \label{sec:irreducible component}
	Let $f : X \to S$ be a minimal $\delta$-regular abelian fibration and $P$ its t-automorphism scheme. In this section, we focus on studying each irreducible component $E$ of the central fiber $X_s$. The main theme is to construct a morphism $p : E \to A$ to an abelian torsor $A$ of dimension $n-1$, such that the $P^\circ_s$-action on $E$ descends to a transitive $P^\circ_s$-action on $A$ (so $p$ will be a locally trivial morphism). This is almost always possible, as we will see in this section.
	
	We highlight four main results. In \Cref{prop:E normalization}, we show that every irreducible component $E \subset X_s$ after normalization is a $\PP^1$-bundle $\tilde E \to A$, with only few exceptions. In \Cref{prop:E}, we show that $E$ is either a $\PP^1$- or $\PP^1 \cup \PP^1$-bundle $E \to A$ \emph{without} a normalization, but assuming $X_s$ is reducible. When $X_s$ is integral, this may be false, but we will treat this case separately in \Cref{prop:Weierstrass fiber unstable} and \ref{prop:Weierstrass fiber semistable}.

	\subsection{Irreducible component after normalization}
		Let us first study the normalization $\tilde E$ of an irreducible component $E \subset X_s$. The results in this subsection were first proved in \cite[Theorem~1.3]{hwang-ogu09}. Our argument generalizes \cite[Proposition~5.14]{kim24}.
		
		\begin{proposition} \label{prop:E normalization-preparation}
			Let $f : X \to S$ be a klt-minimal $\delta$-regular abelian fibration. Then every irreducible component $E \subset X_s$ has a smooth normalization $\tilde E$.
		\end{proposition}
		\begin{proof}
			Take the $P^\circ$-action on $X$ in \Cref{thm:Po-action}. The $P^\circ_s$-action on $X_s$ lifts to its normalization $\tilde X_s$ and restricts to its connected component $\tilde E \subset \tilde X_s$. Note that $\Sing(\tilde E)$ is a $P^\circ_s$-orbit of dimension $\le n-2$ by normality. On the other hand, every orbit of $\tilde E$ has dimension $\ge n-1$ by \Cref{lem:P-orbit has dimension n-1}. This forces $\Sing(\tilde E) = \emptyset$ and hence $\tilde E$ is smooth.
		\end{proof}
		
		\begin{proposition} \label{prop:E normalization}
			In \Cref{prop:E normalization-preparation}, assume either of the following.
			\begin{enumerate}[label = \textnormal{(\roman*)}, ref = \roman*]
				\item $P_s$ is not an abelian variety and the $P^\circ_s$-action on $E$ has a finite kernel (e.g., when $E \subset X_s$ is a reduced component by \Cref{prop:P-action property}).
				
				\item $E$ has a rational curve.
			\end{enumerate}
			Then there exists an $P^\circ_s$-equivariant $\PP^1$-bundle morphism $\tilde p : \tilde E \to A$ to an abelian torsor $A$ of dimension $n-1$. Moreover, $P^\circ_s$ acts on $A$ transitively.
		\end{proposition}
		\begin{proof}
			(i) Let $K$ be the kernel of the action and denote $\bar P_s = P^\circ_s / K$. Then the $\bar P_s$-action on $\tilde E$ is generically free with a unique open orbit $\bar P_s \subset \tilde E$. Let $0 \to \bar L_s \to \bar P_s \to A \to 0$ be the Chevalley exact sequence, where $\bar L_s = \GG_m$ or $\GG_a$ and $A$ is an abelian variety of dimension $n-1$. The equivariant rational map $\tilde E \supset \bar P_s \to A$ is defined everywhere because $\tilde E$ is smooth and $A$ is an abelian torsor. Since the $\bar P_s$-action on $A$ is transitive, all its fibers are isomorphic to each other and contains the linear part $\bar L_s$ as an open subset. Generic smoothness shows it is a $\PP^1$-bundle.
			
			\medskip
			
			(ii) Note that $P_s$ cannot be an abelian variety by \Cref{lem:classification when P is smooth}. Let us further assume that (i) does not hold, i.e., $E$ has a $1$-dimensional kernel containing the linear part $L_s \subset P^\circ_s$ (\Cref{prop:P-action property}). We may thus consider an $P_s^\circ/L_s = A_s$-action instead. Take a sufficiently large finite subgroup $K \subset A_s$ that contains all stabilizers of the $A_s$-action on $\tilde E$. Defining the quotients $F = \tilde E / K$ and $\bar A = A_s / K$, the variety $F$ admits an $\bar A$-action by \Cref{lem:quotient action}, which is free because we killed every stabilizer. The variety $F$ is smooth by arguing $\Sing(F) = \emptyset$ as in \Cref{prop:E normalization-preparation}. The free quotient (e,g, \cite[Tag~06PH]{stacks-project})
			\[ F \to C = F/\bar A \]
			defines an $\bar A$-torsor over a smooth projective curve $C$. At this point, use the rational curve assumption to see that $C$ is dominated by a rational curve. This proves $C = \PP^1$, and hence forces $F = \PP^1 \times \bar A$ because $H^1_{\acute et} (\PP^1, \bar A) = 0$.
			
			\medskip
			
			We have an equivariant morphism $\tilde E \to F \to \bar A$ to an abelian torsor $\bar A$ of dimension $n-1$ with a transitive $A_s$-action. Take its Stein factorization $\tilde E \to A \to \bar A$. Then $A$ admits a transitive $A_s$-action as well, so it is an abelian torsor of dimension $n-1$ (and $A \to \bar A$ is an isogeny). Moreover, $\tilde p : \tilde E \to A$ is smooth projective and is locally trivial. Since $\tilde E$ has a rational curve by assumption, it is a $\PP^1$-bundle.
		\end{proof}
		
		\begin{corollary} \label{cor:rational curve in E}
			Let $f : X \to S$ be a klt-minimal $\delta$-regular abelian fibration. Then every rational curve in $E \subset X_s$ is the image of a fiber of the $\PP^1$-bundle $\tilde p : \tilde E \to A$ in \Cref{prop:E normalization}. \qed
		\end{corollary}
		
		\begin{remark}
			\begin{enumerate}
				\item Propositions~\ref{prop:E normalization-preparation}--\ref{cor:rational curve in E} are proved in \cite[\S 2--3]{hwang-ogu09} in the context of Lagrangian fibrations. A rational curve in $X_s$ is called a characteristic leaf there, since it is obtained as a leaf of a certain rank $1$ foliation induced form the symplectic form.
				
				\item A posteriori, (i) implies (ii) in \Cref{prop:E normalization}. According to \Cref{lem:P-orbit has dimension n-1}, the kernel of the $P^\circ_s$-action on $X_s$ is either finite or $1$-dimensional. It is these two cases that behaved differently in the proof of \Cref{prop:E normalization}.
				
				\item We will see in \Cref{prop:E} that once $X_s$ is reducible, every irreducible component $E \subset X_s$ is covered by smooth rational curves.
			\end{enumerate}
		\end{remark}
		
		We conclude with some properties of the morphism $\tilde p : \tilde E \to A$, which will be used in future arguments.
		
		\begin{lemma} \label{lem:E normalization structure-semistable}
			In \Cref{prop:E normalization}, assume the condition (i) holds.
			\begin{enumerate}
				\item If $P_s$ has a linear part $\GG_m$, then $\tilde E$ has precisely two $P^\circ_s$-orbits with nontrivial stabilizers. Both are sections of $\tilde p$, which we call $0$ and $\infty$.
				\item If $P_s$ has a linear part $\GG_a$, then $\tilde E$ has precisely one $P^\circ_s$-orbit with nontrivial stabilizers. It is a section of $\tilde p$.
			\end{enumerate}
		\end{lemma}
		\begin{proof}
			The $\PP^1$-bundle $\tilde p$ was constructed as a compactification of the Chevalley morphism $\bar P_s \to A$ where $\bar P_s \subset \tilde E$ is the open orbit. Its complement $Z = \tilde E \setminus \bar P_s \to A$ is a $\bar P_s$-equivariant morphism.
			
			\medskip
			
			(1) $Z \to A$ is finite \'etale of degree $2$ because $\bar L_s = \GG_m \subset \PP^1$ is a two-point compactification. The stabilizers of $Z$ contain $\bar L_s$ because $\dim Z = n-1$. But the stabilizer of $A$ is precisely $\bar L_s$, so this concludes $Z$ has the stabilizers $\bar L_s$ as well. This means $Z$ a disjoint union of two sections. (2) $Z \to A$ is an isomorphism because $\bar L_s = \GG_a \subset \PP^1$ is a one-point compactification.
		\end{proof}
		
		\begin{lemma} \label{lem:E normalization structure-unstable}
			In \Cref{prop:E normalization}, assume (ii) holds but (i) fails. Then $\tilde p : \tilde E \to A$ admits a finite \'etale Galois base change as follows with the following properties:
			\[\begin{tikzcd}
				\PP^1 \times A_s \arrow[r] \arrow[d, "\pr_2"] & \tilde E \arrow[d, "\tilde p"] \\
				A_s \arrow[r] & A
			\end{tikzcd}.\]
			\begin{enumerate}
				\item The Galois group $G$ acts diagonally on $\PP^1 \times A_s$.
				\item The $A_s$-action on $\tilde E$ has the stabilizer group $\ker (G \to \PGL_2) \subset A_s$ at every point $x \in \tilde E$.
			\end{enumerate}
		\end{lemma}
		\begin{proof}
			Recall from the proof of \Cref{prop:E normalization} that $L_s \subset P^\circ_s$ acts trivially on $E$ and hence $\tilde p : \tilde E \to A$ is an $A_s$-equivariant morphism. Since the $A_s$-action on $A$ is transitive (say $G \subset A_s$ is its stabilizer group), there exists an isogeny $A_s \to A$ with Galois group $G$. Let us temporarily write $\tilde p' : \tilde E' \to A_s$ for the base change of $\tilde p$. Then the $A_s$-action on $\tilde E'$ is free because $\tilde p'$ is $A_s$-equivariant and the action is free on $A_s$. Its free quotient $\tilde E' \to \tilde E' / A_s \cong \PP^1$ is an $A_s$-torsor. It is trivial because any fiber of $\tilde p$ gives its section, showing $\tilde E' = \PP^1 \times A_s$.
			
			\medskip
			
			(1) Because $\pr_2$ is equivariant, $g \in G \subset A_s$ acts on $\PP^1 \times A_s$ as
			\[ g \cdot (x, y) = (g_y(x), y + g) \qquad\mbox{for}\quad (x,y) \in \PP^1 \times A_s , \ g_y : \PP^1 \to \PP^1 .\]
			This defines a morphism $A_s \to \PGL_2$ sending $y \mapsto g_y$, which is necessarily constant since $\PGL_2$ is affine and $A_s$ is proper connected. This shows the $G$-action on $\PP^1 \times A_s$ is diagonal. (2) directly follows from (1). The same idea will be used again in \Cref{prop:untangle for unstable}.
		\end{proof}

	\subsection{Irreducible component without normalization}
		To analyze the structure of an irreducible component $E \subset X_s$ without normalization, we need two additional ingredients. The first is the existence of a log terminal contraction in \S \ref{sec:Weierstrass models}, and the second is the structural results of non-normal varieties (\S \ref{sec:conductor}).
		
		The following is the main result of this subsection.
		
		\begin{proposition} \label{prop:E}
			Let $f : X \to S$ be a minimal $\delta$-regular abelian fibration. Assume the central fiber $X_s$ is reducible. Then every irreducible component $E \subset X_s$, reduced or not, admits an \'etale locally trivial morphism $p : E \to \bar A$ to an abelian torsor $\bar A$ with a transitive $P^\circ_s$-action. Moreover, the $\PP^1$-bundle $\tilde p : \tilde E \to A$ in \Cref{prop:E normalization} exists and is related to $p$ by one of the following:
			\begin{enumerate}[label = \textnormal{(\roman*)}, ref = \roman*]
				\item $E = \tilde E$ is smooth projective and $p = \tilde p$ is a $\PP^1$-bundle.
				\item $E$ is singular along a section of $p$, and there exists a commutative diagram
				\[\begin{tikzcd}
					\tilde E \arrow[r, "\nu"] \arrow[d, "\tilde p"] & E \arrow[d, "p"] \\
					A \arrow[r] & \bar A
				\end{tikzcd}.\]
				Here $\nu$ is the normalization, the fiber of $p$ is isomorphic to $(xy = 0) \subset \PP^2_{[x,y,z]}$, and $A \to \bar A$ is an isogeny of degree $2$.
			\end{enumerate}
		\end{proposition}
		
		That is, when $X_s$ is reducible, every irreducible component $E$ admits an \'etale locally trivial fiber bundle
		\[ p : E \to \bar A \qquad\mbox{with a fiber } \PP^1 \mbox{ or } \PP^1 \cup \PP^1 \ (= (xy = 0) \subset \PP^2) .\]
		The second case did not arise when the relative dimension $n$ is $1$, as $\PP^1 \cup \PP^1$ is not irreducible. For $n \ge 2$, the second case does arise and has already been observed in \cite[Figure~1]{mat01}. We thank Mirko Mauri for pointing this out. The rest of this section is devoted to the proof of \Cref{prop:E}.
		
		\medskip
		
		A tree of $\PP^1$'s is a proper, connected, and reduced curve with nodal singularities, whose irreducible components are $\PP^1$ and dual graph is a tree.
		
		\begin{lemma} \label{lem:prop:E step 1}
			In the setting of \Cref{prop:E}, every irreducible component $E \subset X_s$ admits an equivariant isotrivial fibration $p : E \to \bar A$, where $\bar A$ is an abelian torsor isogenous to $A_s$ and its fiber $C$ is a tree of $\PP^1$'s. In particular, $E$ is covered by smooth rational curves, and every rational curve in $E$ is smooth.
		\end{lemma}
		\begin{proof}
			Since $X_s$ is reducible, we may apply the argument in \Cref{prop:klt Weierstrass model} to a single irreducible component $E \subset X_s$ and yield a contraction
			\[ \phi : X \to \bar X .\]
			The contracted locus $\bar A$ is a connected $P^\circ_s$-invariant subvariety of $X_s$ of dimension $\le n-1$. Hence it is a single $A_s$-orbit and is an $(n-1)$-dimensional abelian torsor (\Cref{lem:P-orbit has dimension n-1}). Hence the equivariant morphism $p = \phi_{|E} : E \to \bar A$ is locally trivial, say, with fiber $C$.
			
			Let $\bar Z \subset \bar X$ be a complete intersection of $n-1$ general very ample divisors and $Z = \phi^{-1}(\bar Z)$ its preimage. Then $Z$ is regular by Bertini's theorem (the divisors are still base point free), $\bar Z$ is log terminal by \cite[Lemma~5.17]{kol-mori}, and $\bar Z \cap \bar A$ is $0$-dimensional. Hence $\phi : Z \to \bar Z$ is a resolution of log terminal surface singularities, so its fiber $C$ is necessarily a tree of $\PP^1$'s as claimed (see case 4--5 in \cite[Theorem~4.7]{kol-mori}). Every rational curve in $E$ is contained in a fiber $C$, so it is isomorphic to $\PP^1$.
		\end{proof}
		
		Take the normalization $\nu : \tilde E \to E$ and conductor subscheme diagram (\S \ref{sec:conductor})
		\begin{equation} \label{diag:conductor of E}
		\begin{tikzcd}
			\tilde Z \arrow[r, hook] \arrow[d] & \tilde E \arrow[d, "\nu"] \\
			Z \arrow[r, hook] & E
		\end{tikzcd}.
		\end{equation}
		The diagram is $P^\circ_s$-equivariant because $E$ came equipped with a $P^\circ_s$-action. \Cref{prop:E} will eventually compute the full structures of $\tilde Z$ and $Z$, but the following weaker claims are sufficient at this point.
		
		\begin{lemma} \label{lem:prop:E step 2}
			The following holds in the diagram \eqref{diag:conductor of E}.
			\begin{enumerate}
				\item $\tilde Z_{\red}$ (resp. $Z_{\red}$) is a disjoint union of $(n-1)$-dimensional abelian torsors, each of which is an $A_s$-orbit in $\tilde E$ (resp. $E$).
				
				\item Every connected component of $\tilde Z_{\red}$ is isomorphic to each other.
				
				\item The surjective morphism $\tilde Z_{\red} \twoheadrightarrow Z_{\red}$ is an isogeny of abelian torsors on each connected component.
			\end{enumerate}
		\end{lemma}
		\begin{proof}
			(1, 3) Both $\tilde Z \subset \tilde E$ and $Z \subset E$ are $P^\circ_s$-subschemes of dimension $\le n-1$, so by \Cref{lem:P-orbit has dimension n-1}, $\tilde Z_{\red}$ and $Z_{\red}$ are disjoint unions of abelian torsors of dimension $n-1$, each of which is a single $P^\circ_s$-orbit. Moreover, $Z$ and $\tilde Z$ are homogeneous and thus have no embedded points. The equivariant morphism $\tilde Z_{\red} \to Z_{\red}$ is a finite \'etale morphism between disjoint union of abelian torsors.
			
			\medskip
			
			(2) Suppose that $E$ satisfies the condition (i) in \Cref{prop:E normalization}. Then \Cref{lem:E normalization structure-semistable} shows that $\tilde Z_{\red} \subset \tilde E$ is either one section or a disjoint union of two sections of $\tilde p$. Either case, $\tilde Z_{\red}$ has isomorphic connected components. Suppose now that $E$ satisfies the condition (ii) in \Cref{prop:E normalization} (by \Cref{lem:prop:E step 1}) but not (i). Then \Cref{lem:E normalization structure-unstable} shows that $A_s$ acts on $E$ with constant stabilizers. That is, every $A_s$-orbit is isomorphic and hence $\tilde Z_{\red}$ has isomorphic components.
		\end{proof}
		
		\begin{lemma} \label{lem:prop:E step 3}
			Let $p : E \to \bar A$ be a morphism constructed in \Cref{lem:prop:E step 1} and $\tilde p : \tilde E \to A$ a $\PP^1$-bundle in \Cref{prop:E normalization}. Then the two are related by an equivariant commutative diagram
			\[\begin{tikzcd}
				\tilde E \arrow[r, "\nu"] \arrow[d, "\tilde p"] & E \arrow[d, "p"] \\
				A \arrow[r] & \bar A
			\end{tikzcd},\]
			where $A \to \bar A$ is an isogeny whose degree is the number of $\PP^1$'s in $C$.
		\end{lemma}
		\begin{proof}
			\Cref{cor:rational curve in E} showed every rational curve $R \subset E$ must be of the form $\nu(\tilde R)$ where $\tilde R$ is a fiber of the equivariant $\PP^1$-bundle $\tilde p : \tilde E \to A$. We have already proved that the rational curve $R$ is smooth in \Cref{lem:prop:E step 1} (and $\nu$ collapses only the abelian variety directions by \Cref{lem:prop:E step 2}(3)), so the morphism $\nu_{|\tilde R} : \tilde R \to R$ is an isomorphism for every rational curve $R \subset E$. In summary, there is a bijective correspondence between the fibers $\tilde R$ of $\tilde p$ and (smooth) rational curves $R$ in $E$, and $\nu$ restricts to an isomorphism between them.
			
			Every fiber $\tilde R$ is sent to a rational curve $R = \nu(\tilde R)$ in $E$, and hence to a point in $\bar A$. Rigidity lemma as in \cite[Lemma~14]{kaw81} constructs the desired commutative diagram. The diagram is $P^\circ_s$-equivariant and both $A$ and $\bar A$ are a single $A_s$-orbit. This shows $A \to \bar A$ is an isogeny. To compute its degree $d$, take a point $y \in \bar A$ with fiber $C = p^{-1}(y)$. Its preimage in $\tilde E$ is $\nu^{-1}(C)$ on the one hand, and on the other hand is a disjoint union of $d$ number of $\PP^1$-fibers of $\tilde p$. Since $\nu$ sends the fibers $\tilde R$ isomorphically to a rational curve in $C$, this shows the degree $d$ coincides with the number of rational curves in $C$.
		\end{proof}
		
		\begin{proof}[Proof of \Cref{prop:E}]
			If $C = \PP^1$, then $p : E \to A$ is a $\PP^1$-bundle and hence we fall into the first case. Assume from now on that $C$ is a tree of at least two $\PP^1$'s.
			
			\medskip
			
			Take two smooth rational curves $R_1, R_2 \subset C$ intersecting at a single node $x_{12} \in C$. Say $R_i = \nu(\tilde R_i)$ for $i = 1, 2$, where $\tilde R_i = \tilde p^{-1}(t_i)$ for $t_i \in A$ are fibers (\Cref{cor:rational curve in E}). We have shown in \Cref{lem:prop:E step 3} that $\nu$ induces isomorphisms $\tilde R_i \to R_i$. Therefore, the intersection $x_{12} = R_1 \cap R_2$ lifts to two distinct points $\tilde x_1 \in \tilde R_1$ and $\tilde y_2 \in \tilde R_2$.
			
			Let us first prove that $\tilde x_1$ and $\tilde y_2$ are in the same $P^\circ_s$-orbit of $\tilde E$, or $A_s$-orbit because they are in $\tilde Z_{\red}$ in \eqref{diag:conductor of E}. Assume on the contrary that they are not. Then \Cref{lem:prop:E step 2}(1--3) says the two different orbits $A_s \cdot \tilde x_1$ and $A_s \cdot \tilde y_2$ are connected components of $\tilde Z$ and are mapped into a single connected component of $Z$, i.e., the two orbits have the same image in $E$. Take any point $\tilde x_2 \in (A_s \cdot \tilde x_1) \cap \tilde R_2$ and write it as
			\[ \tilde x_2 = a \cdot \tilde x_1 \qquad\mbox{for}\quad a \in A_s .\]
			Since $\nu$ is equivariant, it sends both $\tilde x_2 = a \cdot \tilde x_1$ and $\tilde y_3 \coloneq a \cdot \tilde y_2$ to a single point $x_{23} \coloneq a \cdot x_{12} \in E$. Take a fiber $\tilde R_3$ of $\tilde p$ containing $\tilde y_3$ (potentially identical to $\tilde R_1$) and its image $R_3 = \nu(\tilde R_3)$, which is a rational curve in $C$ intersecting with $R_2$ at $x_{23}$. This inductive process stops within a finite step, i.e., we have $\tilde R_i = \tilde R_j$ for some $i < j$. But this constructs a cycle of rational curves in $C$, contradicting the fact that $C$ is a tree of $\PP^1$'s. This shows $\tilde x_1$ and $\tilde y_2$ are in the same $A_s$-orbit.
			
			\medskip
			
			At this point, we have two smooth rational curves $R_1, R_2 \subset C$ intersecting at $x_{12} \in Z$, whose lifts $\tilde x_1$ and $\tilde x_2 \coloneq \tilde y_2$ lie in a single $A_s$-orbit, or the same connected component of $\tilde Z$. We now claim that the Galois group of $A \to \bar A$ is $2$-torsion. Say $\tilde x_2 = a \cdot \tilde x_1$ for $a \in A_s$. Take a point $\tilde x_3 \coloneq a \cdot \tilde x_2$ in $\tilde E$ and a fiber $\tilde R_3$ passing through it. By \Cref{lem:prop:E step 3}, the claim is equivalent to $\tilde R_3 = \tilde R_1$. Assume on the contrary that $\tilde R_3 \neq \tilde R_1$. Then there are three rational curves $R_1$, $R_2$, and $R_3 = \nu(\tilde R_3)$ that pass through a single point $x_{12} = \nu(\tilde x_1) = \nu(\tilde x_2) = \nu(\tilde x_3)$. This violates that $C$ has only nodes. Therefore, the Galois group of $A \to \bar A$ is $2$-torsion.
			
			\medskip
			
			It only remains to show that the Galois group of $A \to \bar A$ is $\ZZ/2$. Assume on the contrary that the Galois group contains $(\ZZ/2)^{\times 2}$. Then there exist four fibers $\tilde R_{ij}$ for $i,j = 0,1$ corresponding to the four points in $A$ with the same image in $\bar A$. Using \Cref{lem:prop:E step 2}--\ref{lem:prop:E step 3} and similar arguments as above, the only possible way to make this sense is that there are eight points $\tilde x_{ij}, \tilde y_{ij} \in \tilde R_{ij}$ for $i,j=0,1$ such that
			\[ \nu(\tilde x_{ij}) = \nu(\tilde x_{i+1, j}) \quad\mbox{and}\quad \nu(\tilde y_{ij}) = \nu(\tilde y_{i, j+1}) \qquad\mbox{for}\quad i,j = 0,1 \]
			where the additions are taken mod $2$. But this produces four rational curves $R_{ij}$ consisting of a cycle in $C$, which is again impossible. This completes the proof of the proposition.
		\end{proof}
		
		\begin{remark}
			Applying the argument of \Cref{lem:prop:E step 1} to a full klt-Weierstrass contraction $\phi : X \to \bar X$ in \Cref{prop:klt Weierstrass model}, its reduced fiber $C = F_{\red}$ is again a tree of $\PP^1$'s. This already shows that if $X_s$ is reducible, then for every irreducible component $E_j \subset X_s$, its complement $\bigcup_{i \neq j} E_i$ has no loops and has only nodal singularities.
		\end{remark}

	\subsection{Integral fibers: unstable case}
		\Cref{prop:E} considered the case when $X_s$ is reducible. Its starting point was the existence of a log terminal contraction in \Cref{lem:prop:E step 1}, which was from the negativity of the intersection number in \eqref{eq:intersection negative}. Such an intersection number is $0$ when $X_s$ is irreducible and the argument fails.
		
		The goal of the remaining two subsections is to treat the case when $X_s$ is integral separately. In this first subsection, we consider the unstable case.
		
		\begin{proposition} \label{prop:Weierstrass fiber unstable}
			Let $f : X \to S$ be a Weierstrass $\delta$-regular abelian fibration with a central fiber $X_s = E$. Assume that $P_s$ has a linear part $\GG_a$. Then there exists an equivariant morphism
			\[ p : E \to A ,\]
			where $A$ is an abelian torsor of dimension $n-1$ with a transitive $P^\circ_s$-action.
		\end{proposition}
		
		This is a consequence of the following statement about a single variety instead of degenerations of abelian varieties.
		
		\begin{proposition} \label{prop:prop:Weierstrass fiber unstable}
			Let $P$ be a connected commutative algebraic group acting on a proper integral variety $E$, both with dimension $n$ over an algebraically closed field $k$ of characteristic $0$. Assume
			\begin{itemize}
				\item $P$ admits a Chevalley decomposition $0 \to \GG_a \to P \to A \to 0$ for an abelian variety $A$ of dimension $n-1$.
				\item The action is generically free, and every nontrivial stabilizer is $\GG_a$.
			\end{itemize}
			Then there exists a Zariski locally trivial and $P$-equivariant morphism
			\[ p : E \to A ,\]
			whose fiber $R$ is a projective rational curve with one potentially singular point $x \in R$ with $\hat {\mathcal O}_{R, x} = k[[t^{\delta+1}, t^{\delta+2}, \cdots]]$ for $\delta \ge 0$.
		\end{proposition}
		
		\begin{proof}[Proof of \Cref{prop:Weierstrass fiber unstable} from \Cref{prop:prop:Weierstrass fiber unstable}]
			Since $X_s = E$ is reduced, $P^\circ_s$ acts on it generically freely by \Cref{prop:P-action property}. Moreover, every nontrivial stabilizer at $x \in X_s$ is of the form $\GG_a \times K_x$ for a finite group $K_x \subset P^\circ_s$ by \Cref{lem:pi_0 sequence splits}. Take a sufficiently large finite group $K \subset P^\circ_s$ containing all $K_x$, and take the quotient $\bar P_s = P^\circ_s / K$ and $\bar X_s = X_s / K$. Then $\bar P_s$ acts on $\bar X_s$ (\Cref{lem:quotient action}) generically freely with only nontrivial stabilizers $\GG_a$. \Cref{prop:prop:Weierstrass fiber unstable} shows that $\bar X_s$ is an $R$-fiber bundle over $\bar A = A_s / K$. In fact, $R$ is a cuspidal cubic because $\delta$ coincides with the dimension of the linear part of $\Pic^\circ_E = \check P^\circ_s$, which is $1$ in our case. The composition $X_s \to \bar X_s \to \bar A$ is the desired morphism.
			
			We note a subtle point that the fiber of $X_s \to \bar A$ may \emph{not} be a cuspidal cubic curve. For example, type $\mathrm{III}/2$ and $\mathrm{IV}/3$ unstable fibrations will produce counterexamples (they will both have integral central fibers, see \Cref{table:classification of non-multiple fibers}).
		\end{proof}
		
		The rest of this subsection is devoted to the proof of \Cref{prop:prop:Weierstrass fiber unstable}.

		\subsubsection{Descent along normalization}
			Keep the notation and assumptions in \Cref{prop:prop:Weierstrass fiber unstable}. Take the normalization $\nu : \tilde E \to E$. The $P$-action lifts to $\tilde E$ and $\nu$ is equivariant. The $P$-action on $\tilde E$ is generically free, so the argument in \Cref{prop:E normalization}(i) shows that the Chevalley morphism $P \to A$ compactifies to an equivariant $\PP^1$-bundle morphism
			\[ \tilde p : \tilde E \to A .\]
			Our goal is to prove $\tilde p$ descends along $\nu$ to a morphism $p : E \to A$.
			
			\begin{lemma} \label{lem:prop:integral fiber unstable step 1}
				The normalization $\nu$ is a homeomorphism.
			\end{lemma}
			\begin{proof}
				Both $E$ and $\tilde E$ have open orbits $j : P \hookrightarrow E$ and $\tilde j : P \hookrightarrow \tilde E$ that are identified by $\nu$. Since $P$ is a $\GG_a$-torsor, the complement of $\tilde j$ is a section $\tilde e : A \to \tilde E$ of $\tilde p$.
				\begin{equation} \label{diag:prop:integral fiber unstable}
				\begin{tikzcd}
					& E \\
					P \arrow[r, hook, "\tilde j"] \arrow[ru, hook', bend left=30, "j"] \arrow[rd, bend right=30, "\GG_a"'] & \tilde E \arrow[u, "\nu"'] \arrow[d, "\tilde p"] & A \arrow[l, hook', "\tilde e"'] \arrow[lu, hook, bend right=30, "e"'] \arrow[ld, bend left=30, "\id"]\\
					& A
				\end{tikzcd}
				\end{equation}
				This section is a single $P$-orbit with stabilizer $\GG_a$. Hence its image $e = \nu \circ \tilde e : A \to E$ is again a single $P$-orbit whose stabilizer contains $\GG_a$. Every nontrivial stabilizer is $\GG_a$ by assumption, so this shows that $\tilde e(A)$ and $e(A)$ are isomorphic by $\nu$. In particular, the normalization $\nu$ is bijective and hence (universally) homeomorphic by \cite[Tag~0BRD]{stacks-project}.
			\end{proof}
			
			Let us keep using the notations in \eqref{diag:prop:integral fiber unstable} but also consider the conductor subscheme diagram (\S \ref{sec:conductor})
			\begin{equation} \label{diag:prop:integral fiber unstable 2}
				\begin{tikzcd}
					\tilde Z \arrow[r, hook] \arrow[d, "\mu"] & \tilde E \arrow[d, "\nu"] \arrow[rd, bend left=20, "\tilde p"] \\
					Z \arrow[r, hook] & E & A
				\end{tikzcd}.
			\end{equation}
			\Cref{lem:prop:integral fiber unstable step 1} showed that $\mu$ induces an isomorphism of $\tilde Z_{\red} = \tilde e(A)$ and $Z_{\red} = e(A)$. The schemes $Z$ and $\tilde Z$ have no embedded points by $P$-equivariance, so $\tilde Z \subset \tilde E$ is a Cartier divisor.
			
			\begin{lemma} \label{lem:prop:integral fiber unstable step 2}
				$Z$ is reduced, i.e., $Z = e(A)$ scheme-theoretically.
			\end{lemma}
			\begin{proof}
				The question is local on $E$. Since $P \to A$ is a Zariski $\GG_a$-torsor (e.g., \cite[Proposition~16.55]{milne:group}), there exists an open subset $0 \in U \subset A$ with a splitting $P_U = \GG_a \times U$. Though $P_U$ is no longer a group, we may still think of its orbits in $\tilde E$ and $E$. Fix a reference point $t \in A$, a fiber $\tilde E_t = \tilde p^{-1}(t) = \PP^1$ and its $P_U$-orbit $\tilde E_{U \cdot t} = \tilde p^{-1}(U \cdot t)$. By \Cref{lem:prop:integral fiber unstable step 1}, its image $E_{U \cdot t} \coloneq \nu(\tilde E_{U \cdot t})$ is an open subscheme of $Y$. Let us from now on restrict everything over $E_{U \cdot t}$ and drop the subscripts for brevity.
				
				Fix a section $U \hookrightarrow P$ using the splitting $P = \GG_a \times U$. The $U$-translations on $\tilde E$ and $E$ have no stabilizers because every nontrivial stabilizer was $\GG_a \subset P$, whose intersection with $U \subset P$ is trivial. Hence the $U$-translations define smooth equivalence relations on both $E$ and $\tilde E$. Taking quotients, we have a cartesian diagram by \cite[Tag~06PH and 04S6]{stacks-project}:
				\[\begin{tikzcd}[column sep=normal]
					\tilde E \arrow[r, "/U"] \arrow[d, "\nu"] & \PP^1 \arrow[d, "\nu"] \\
					E \arrow[r, "/U"] & R
				\end{tikzcd}.\]
				The horizontal arrows are smooth, so $\nu_* \mathcal O_{\tilde E}/\mathcal O_E$ is the pullback of $\nu_* \mathcal O_{\PP^1} / \mathcal O_R$ and we may show $\nu_* \mathcal O_{\PP^1} / \mathcal O_R$ is supported on a reduced point of $R$. This reduces the statement to the case when $E = R$ is a $\GG_a$-curve. For curves, we will prove this with explicit computations in \Cref{prop:curve}.
			\end{proof}
			
			\begin{proof}[Proof of \Cref{prop:prop:Weierstrass fiber unstable}]
				Consider the diagram \eqref{diag:prop:integral fiber unstable 2}. \Cref{lem:prop:integral fiber unstable step 2} shows that the composition $\tilde Z \hookrightarrow \tilde E \to A$ descends to an isomorphism $Z \to A$ along $\mu$. By the universal property of pushout, $\tilde p$ descends to $p : E \to A$ along $\nu$. The remaining claims about the rational curve $R$ will follow again from \Cref{prop:curve}.
			\end{proof}

		\subsubsection{Explicit calculations for singular curves}
			It remains to consider the case when $E = R$ is a curve:
			
			\begin{proposition} \label{prop:curve}
				Let $R$ be a projective integral curve with a generically free $\GG_a$-action and let $x \in R$ be the only potentially singular point. Then
				\begin{enumerate}
					\item $\hat {\mathcal O}_{R, x} = k[[t^{\delta+1}, t^{\delta+2}, \cdots]] = k \oplus (t^{\delta+1}) \ \subset k[[t]]$ for $\delta \ge 0$.
					\item The conductor subscheme $Z \subset R$ is a reduced point $x$.
				\end{enumerate}
			\end{proposition}
			
			The generically free $\GG_a$-action has an open orbit $\GG_a \subset R$, so $R$ is a rational curve. Writing $\nu : \PP^1 \to R$ for the normalization, the $\GG_a$-action lifts to $\PP^1$ and $\nu$ is equivariant. Fixing an appropriate coordinate of $\PP^1$, we can assume
			\[ a \cdot [x,y] = [x, \, ax + y] \qquad\mbox{for}\quad [x,y] \in \PP^1 ,\ a \in \GG_a .\]
			With this coordinate, the lift of the point $x \in R$ is $[0,1] \in \PP^1$. Locally around $[0,1]$ with coordinate $t = x/y$, the action is
			\begin{equation} \label{eq:Ga-action}
				a \cdot t = \frac{x}{ax + y} = \frac{t}{1 + at} = t (1 - at + a^2t^2 - \cdots) \qquad\mbox{for}\quad a \in \GG_a ,
			\end{equation}
			describing the $\GG_a$-action on every thickening of $[0,1]$.
			
			\begin{lemma} \label{lem:Ga-subalgebra}
				Let $A_i = k[t] / t^{i+1}$ be a $\GG_a$-equivariant algebra defined by \eqref{eq:Ga-action}. Then every nonzero $\GG_a$-equivariant subalgebra of $A_i$ is of the form
				\[ B = k \oplus (t^{\delta+1}) \ \ \subset A_i \qquad\mbox{for}\quad \delta = 0, 1, \cdots, i .\]
			\end{lemma}
			\begin{proof}
				Assume $B \neq k$ and set (we will get $d = \delta + 1$)
				\[ d \coloneq \max \{ d \le i \, : \, t^d \mid f \mbox{ for every } f \in B \mbox{ without constant terms} \} .\]
				Start with $f_d = c_d t^d + \cdots + c_i t^i \in B$ for $c_d \neq 0$. Use $\operatorname{char} k = 0$ to find $a \in \GG_a$ such that
				\[ f_{d+1} \coloneq a \cdot f_d - f_d = c_{d+1}' t^{d+1} + \cdots + c_i' t^i \in B \qquad\mbox{with}\quad c_{d+1}' \neq 0 .\]
				Find polynomials $f_j$ with leading degrees $j$ inductively for all $d \le j \le i$. Their linear combinations span $(t^d)$ and thus completes the proof.
			\end{proof}
			
			\begin{proof}[Proof of \Cref{prop:curve}]
				(1) Start again with the conductor subscheme diagram
				\[\begin{tikzcd}
					\tilde Z \arrow[r, hook] \arrow[d] & \PP^1 \arrow[d, "\nu"] \\
					Z \arrow[r, hook] & R
				\end{tikzcd}.\]
				A Cartier divisor $\tilde Z$ is of the form $\tilde Z = \Spec A_i$ for some $i$ with an induced $\GG_a$-action in \eqref{eq:Ga-action}. Then $Z = \Spec B$ for a subalgebra $B \subset A_i$ is classified in \Cref{lem:Ga-subalgebra}. This pushout diagram around $x \in R$ is the pullback diagram of commutative rings
				\[\begin{tikzcd}
					k[[t]]/t^{i+1} & k[[t]] \arrow[l, twoheadrightarrow] \\
					B \arrow[u, hook] & \hat {\mathcal O}_{R, x} \arrow[l, twoheadrightarrow] \arrow[u, hook]
				\end{tikzcd},\]
				computing the ring $\hat {\mathcal O}_{R, x}$ as claimed.
				
				\medskip
				
				(2) By definition, the conductor ideal is
				\[ \mathfrak c = \big\{ f \in k[[t^{\delta+1}, \cdots]] : f \cdot k[[t]] \subset k[[t^{\delta+1}, \cdots]] \big\} = (t^{\delta+1}) ,\]
				which is maximal in $k[[t^{\delta+1}, \cdots]]$. Hence $Z = \Spec k$ as claimed. We can also compute $\tilde Z = \Spec k[[t]] / (t^{\delta+1}) = \Spec A_{\delta}$, so in fact we have $i = \delta$.
			\end{proof}

	\subsection{Integral fibers: semistable case}
		The goal of this subsection is to formulate a semistable version of \Cref{prop:Weierstrass fiber unstable}.
		
		\begin{proposition} \label{prop:Weierstrass fiber semistable}
			Let $f : X \to S$ be a Weierstrass $\delta$-regular abelian fibration with a central fiber $X_s = E$. Assume that $P_s$ has a linear part $\GG_m$ and take the $\PP^1$-bundle $\tilde p : \tilde E \to A$ in \Cref{prop:E normalization}. Then $E$ is obtained by identifying two sections $0$ and $\infty$ of $\tilde p$ in \Cref{lem:E normalization structure-semistable}(1).
		\end{proposition}
		
		The strategy is similar but simpler than \Cref{prop:prop:Weierstrass fiber unstable}. First, the normalization of $E$ admits an equivariant $\PP^1$-bundle $\tilde p : \tilde E \to A$ as usual. Take the conductor scheme diagram \eqref{diag:prop:integral fiber unstable 2}: understanding the conductor subschemes $\tilde Z \subset \tilde E$ and $Z \subset E$ will determine the variety $E$. The complement of the open orbit $P \subset \tilde E$ is a union of two sections in this case (\Cref{lem:E normalization structure-semistable}), so the conductor subscheme $\tilde Z \subset \tilde E$ is either
		\[ \tilde Z = \tilde Z_1 \quad\mbox{or}\quad \tilde Z_1 \sqcup \tilde Z_2 ,\qquad \tilde Z_{i, \red} \cong A \mbox{ are sections of } \tilde p .\]
		Similar to \Cref{lem:prop:E step 2}, $Z$ is a union of either one or two abelian torsors, and $\tilde Z \to Z$ is a surjective equivariant morphism. The following is a counterpart of \Cref{lem:prop:integral fiber unstable step 2}.
		
		\begin{lemma} \label{lem:prop:Weierstrass fiber semistable}
			$\tilde Z = \tilde Z_1 \sqcup \tilde Z_2$ is a union of two abelian torsors, $Z$ is an abelian torsor, and the morphism $\mu : \tilde Z \to Z$ is a disjoint union of two isomorphisms $\tilde Z_i \to Z$.
		\end{lemma}
		\begin{proof}
			We will use two additional inputs about $X_s = E$: it is demi-normal (\Cref{thm:Kollar semistable}) and has a trivial canonical bundle. By demi-normality, $E$ has nodal singularities in codimension $1$ and hence nodal singularities everywhere by equivariance. Therefore, the conductor subschemes $Z$ and $\tilde Z$ are reduced and the morphism $\mu : \tilde Z \to Z$ in diagram \eqref{diag:prop:integral fiber unstable 2} is finite \'etale of degree $2$. This leaves three possibilities:
			\begin{itemize}
				\item $\tilde Z \to Z$ is a degree $2$ isogeny between connected abelian torsors.
				\item $\tilde Z = \tilde Z_1 \sqcup \tilde Z_2$ and $Z = Z_1 \sqcup Z_2$ are both disconnected, and $\tilde Z_i \to Z_i$ are degree $2$ isogenies of abelian torsors.
				\item $\tilde Z = \tilde Z_1 \sqcup \tilde Z_2$ is disconnected, $Z$ is connected, and $\tilde Z_1 \cong \tilde Z_2 \cong Z$ are abelian torsors.
			\end{itemize}
			
			To rule out the first two cases, we use the triviality of the canonical divisor. Take a single smooth rational curve $\tilde R \subset \tilde E$, a fiber of $\tilde p$, and its image $R = \nu(\tilde R)$ in $E$. In all cases, take a sufficiently small \'etale neighborhood $U$ of $A$ and work on a $\GG_m \times U$-orbit of $R$. The smooth quotient argument in \Cref{lem:prop:integral fiber unstable step 2} reduces the claim to the case when $E = R$ is a curve with a generically free $\GG_m$-action and trivial canonical divisor. From the computation of \cite[\S IV.3]{serre:class_field}, the arithmetic genus of $E$ in the first, second, and third cases are $0$, $0$, and $1$, respectively. Hence only the third case is possible.
		\end{proof}
		
		\begin{proof} [Proof of \Cref{prop:Weierstrass fiber semistable}]
			\Cref{lem:prop:Weierstrass fiber semistable} implies that $X_s = E$ is obtained by identifying the two sections $\tilde Z_1$ and $\tilde Z_2$.
		\end{proof}

\section{Shears and untangles} \label{sec:shears and untangles}
	From now on, we will always assume that $f : X \to S$ is a minimal $\delta$-regular abelian fibration of relative dimension $\ge 2$. The goal of this section is to bring forward two different behaviors of $f$. Both of them are only present when the relative dimension $n$ is at lest $2$.
	
	\medskip
	
	The first phenomenon is well-known in the study of semistable degenerations. In the previous section, we have constructed a morphism $p_i : E_i \to \bar A_i$ for each irreducible component $E_i \subset X_s$ in various situations. One hopes to patch them all to a single morphism $X_s \to A$ to an abelian torsor $A$ of dimension $n-1$. This is in general impossible: there is an obstruction class $a \in A_s$ preventing this, which we call a \emph{shear} (\Cref{def:shear}). See the introduction of \cite{sawon22} for a nice picture of this.
	
	When $f$ is unstable, we will prove that the central fiber $X_s$ always admits a surjective morphism $X_s \to A$ to an abelian torsor of dimension $n-1$, or equivalently, its Albanese torsor $\Alb_s$ has dimension $n-1$. The Albanese morphism $\alb : X_s \to \Alb_s$ will be a locally trivial morphism, but it may fail to be trivial. The obstruction for its triviality can be measured by the stabilizer $G$ of the $A_s$-action on $\Alb_s$, which we call the \emph{Albanese stabilizer} (\Cref{def:Albanese stabilizer}). We will also explain how to clear up this obstruction with a process called the \emph{untangle} of $f$ (\Cref{prop:untangle for unstable}), which is the key technique in our study of unstable fibrations. This is similar to (but not the same as) the trivialization of the Albanese morphism of a K-trivial variety in \cite[Theorem~15]{kaw81}.

	\subsection{Albanese morphism and shears}
		The central fiber $X_s$ is a projective scheme with $h^0(X_s, \mathcal O_{X_s}) = 1$ by \Cref{lem:Kollar-Kovacs}. The existence of the Albanese morphism
		\[ \alb : X_s \to \Alb_s \coloneq \Alb_{X_s} \]
		in this situation is proved in \cite{con:albanese, lau-sch24}. See \Cref{thm:Albanese} for a summary of their results. The pullback $\alb^*$ induces an isomorphism between $\Alb_s^\vee$ and the maximal abelian subvariety of $\Pic^\circ_{X_s}$. Since $\Pic^\circ_{X_s}$ has dimension $n$ in this case, the dimension of $\Alb_s$ is bounded above by $n$.
		
		The universal property of $\alb : X_s \to \Alb_s$ descends the $P_s$-action to $\Alb_s$, making the Albanese morphism equivariant \cite[Corollary~10.3]{lau-sch24}. The neutral component $P^\circ_s$ acts on $\Alb_s$ by translations.
		
		\begin{lemma} \label{lem:transitive action on Albanese}
			The $P^\circ_s$-action on $\Alb_s$ is transitive. In particular, $\alb$ is an \'etale locally trivial morphism.
		\end{lemma}
		\begin{proof}
			By \Cref{cor:Albanese}, $\Alb_{X_{s, \red}} \to \Alb_s$ is surjective. Hence it is enough to show the $P^\circ_s$-action on $\Alb_{X_{s, \red}}$ is transitive. Suppose first that $X_s$ is reducible. Each irreducible component $E_i \subset X_s$ admits a locally trivial bundle $p_i : E_i \to A_i$ by \Cref{prop:E normalization}. Therefore, we have an $P^\circ_s$-equivariant commutative diagram
			\[\begin{tikzcd}
				E_i \arrow[r, phantom, "\subset"] \arrow[d, twoheadrightarrow, "p_i"] & X_{s, \red} \arrow[r, phantom, "\subset"] \arrow[d, "\alb"] &  X_s \arrow[d, "\alb"] \\
				A_i \arrow[r] & \Alb_{X_{s, \red}} \arrow[r, twoheadrightarrow] & \Alb_s
			\end{tikzcd}.\]
			Since each $A_i$ admits a transitive $P^\circ_s$-action, the image of $A_i \to A \coloneq \Alb_{X_{s, \red}}$ is a single $P^\circ_s$-orbit, an abelian subtorsor of $A$. Hence $\alb(X_{s, \red}) = \bigcup_i \alb(E_i) \subset A$ is a union of orbits but also connected, so it has to be a single abelian subtorsor of $A$. By the universal property of the Albanese morphism, this proves both $\alb : X_{s, \red} \to A$ and $A_i \to A$ are surjective. Hence the $P_s^\circ$-action on $A$ is transitive.
			
			Suppose now $X_s = m \cdot E$ is irreducible. If $P_s$ is an abelian variety, then $E$ is an abelian torsor by \Cref{lem:classification when P is smooth}. Hence $P_s$ acts on $X_{s, \red} = \Alb_{X_{s, \red}}$ transitively and we are done. Assume that $P_s$ is not an abelian variety. In \eqref{lem:non-multiple base change}, we will construct a new $\delta$-regular minimal abelian fibration $Y \to T$ such that its central fiber $Y_t$ is integral and $E = Y_t / \varphi$ is its free quotient. By \Cref{prop:E normalization}(i), $\tilde Y_t$ is a $\PP^1$-bundle. By \Cref{prop:E normalization}(ii), $\tilde E = \tilde Y_t / \varphi$ is a $\PP^1$-bundle, too. We may thus use a similar diagram above but with a normalization $\tilde E \to X_{s, \red}$ to conclude.
		\end{proof}
		
		Suppose that $P_s$ has a linear part $\GG_m$ and abelian variety part $A_s$. Then $\check P_s$ has a linear part $\GG_m$ and abelian variety part $\check A_s$ by \Cref{prop:dual Neron model}. The Chevalley sequence $0 \to \GG_m \to \check P^\circ_s \to \check A_s \to 0$ yields an extension class
		\begin{equation} \label{eq:shear}
			[\check P^\circ_s] \in \Ext^1 (\check A_s, \GG_m) = A_s(k) ,
		\end{equation}
		where the computation of the Ext-group is from \cite[\S VII.3.16]{serre:class_field}.
		
		\begin{definition} \label{def:shear}
			Assume that $P_s$ has a linear part $\GG_m$.
			\begin{enumerate}
				\item The \emph{shear} of $f$ (or $X_s$) is the extension class $a \in A_s$ of the dual t-automorphism group $\check P^\circ_s$ in \eqref{eq:shear}.
				\item We say $f$ (or $X_s$) has a \emph{torsion shear} if the shear $a \in A_s$ is torsion.
			\end{enumerate}
		\end{definition}
		
		\begin{remark}
			The Chevalley sequence $0 \to \GG_m \to P^\circ_s \to A_s \to 0$ of the original t-automorphism group measures the non-triviality of the $\PP^1$-bundle $\tilde E \to A_s$ in \Cref{prop:E normalization}. We thus call its extension class $\check a \in \Ext^1 (A_s, \GG_m) = \check A_s$ the \emph{slope} of $f$ (or $X_s$). Notice that shears and slopes play the dual role in the following sense: assume that $f : X \to S$ and $\check f : \check X \to S$ are semistable minimal compactifications of $P$ and $\check P$, respectively (use \Cref{thm:minimal model} and \Cref{prop:P-action existence}). Then the shear of $\check f$ is the slope of $f$, and the slope of $\check f$ is the shear of $f$ by \Cref{prop:dual Neron model}. The shear $a \in A_s$ is torsion if and only if the slope $\check a \in \check A_s$ is torsion because $P_s^\circ$ and $\check P_s^\circ$ are isogenous.
		\end{remark}
		
		The terminology will be justified in \Cref{thm:classification of semistable fibers}, or more precisely in \Cref{prop:shear}. One can consider a similar definition when $P_s$ has a linear part $\GG_a$, but it turns out this is redundant. Let us formulate this phenomenon in the following more general context.
		
		\begin{theorem} \label{thm:splitting}
			Let $P \to S$ be the N\'eron model of an abelian scheme $P_0$ over $S_0$. If its central fiber $P_s$ has a linear part $\GG_a$, then the Chevalley exact sequence of its neutral component uniquely splits. In other words,
			\[ P^\circ_s = \GG_a \times A_s .\]
		\end{theorem}
		\begin{proof}
			Start with the dual abelian scheme $\check P_0 \to S_0$ and its N\'eron model $\check P \to S$. Use \Cref{thm:minimal model} and \Cref{prop:P-action existence} to compactify $\check P$ into a minimal abelian fibration $\check f : \check X \to S$. Since $P$ was $\delta$-regular, so is $\check P$ by \Cref{prop:dual Neron model} and hence is $\check f$. The central fiber $\check X_s$ is non-multiple because it contains $\check P_s$, so $\check f$ is an unstable minimal $\delta$-regular abelian fibration. Thanks to \Cref{prop:dual Neron model is Picard}, the original group $P^\circ_s$ is isomorphic to the neutral Picard group $\Pic^\circ_{\check X_s}$.
			
			\medskip
			
			Take the Weierstrass model $\check X \to \bar X \to S$ using \Cref{prop:Weierstrass model}. \Cref{prop:Weierstrass fiber unstable} constructs a surjective morphism $\bar X_s \to A$ to an abelian torsor of dimension $n-1$, so we have a surjective morphism
			\[ \check X_s \to \bar X_s \to A .\]
			This shows the Albanese torsor of $\check X_s$ has dimension $\ge n-1$. Recall that $\Alb_{\check X_s}^\vee$ is the maximal abelian subvariety in $\Pic^\circ_{\check X_s} = P^\circ_s$, a $\GG_a$-extension of $A_s$. Hence $\Alb_{\check X_s}$ has dimension $n-1$ and $P^\circ_s$ has an abelian subvariety of dimension $n-1$. This forces the extension class $[P^\circ_s] \in \Ext^1 (A_s, \GG_a)$ to be torsion by \cite[Lemma~1(b)]{mat-miy66}. But $\Ext^1 (A_s, \GG_a) = H^1 (A_s, \mathcal O_{A_s}) = k^n$ is torsion-free by \cite[\S VII.3.17]{serre:class_field}, so $[P^\circ_s]$ vanishes and the Chevalley sequence splits. The splitting is unique because $\Hom(A_s, \GG_a) = 0$.
		\end{proof}
		
		\begin{proposition} \label{prop:splitting and albanese}
			Let $f : X \to S$ be a minimal $\delta$-regular abelian fibration and $P$ its t-automorphism scheme. Assume $P_s$ has a linear part $\GG_a$. Then
			\begin{enumerate}
				\item There exist unique splittings of algebraic groups
				\[ P^\circ_s = \GG_a \times A_s ,\qquad \check P^\circ_s = \GG_a \times A_s^* ,\qquad \Pic^\circ_{X_s} = \GG_a \times \Alb_s^\vee ,\]
				where $A_s$, $A_s^*$, and $\Alb_s$ are isogenous abelian varieties of dimension $n-1$.
				\item The $\GG_a$-action on $X_s$ is along the fibers of $\alb : X_s \to \Alb_s$.
			\end{enumerate}
		\end{proposition}
		\begin{proof}
			(1) Apply \Cref{thm:splitting} to the N\'eron models $P$ and $\check P$. For the splitting of $\Pic^\circ_{X_s}$, notice that it is isogenous to $\check P^\circ_s$ (\Cref{prop:dual Neron model is Picard}) and $\Alb_s$ is the dual of its maximal abelian subvariety (\Cref{thm:Albanese}). (2) The $P_s$-action descends to $\Alb_s$ and $\alb$ is equivariant by the universal property. The $\GG_a$-action on $\Alb_s$ is necessarily trivial.
		\end{proof}
		
		\begin{remark} \label{rmk:abelian varieties are not dual}
			If $f$ is non-multiple, then we have $\Pic^\circ_{X_s} = \check P_s^\circ$ by \Cref{prop:dual Neron model is Picard} and hence $A_s^* = \Alb_s^\vee$. If $f$ is multiple, then this may not be the case.
			
			Moreover, we caution the reader that the abelian variety parts $A_s$ of $P_s$ and $A_s^*$ of $\check P_s$ may \emph{not} be dual to each other, in contrast to \Cref{prop:dual Neron model}(2). \Cref{thm:classification of unstable fibers} will construct counterexamples: when $f$ is unstable with nontrivial Albanese stabilizer $G$ (see \Cref{def:Albanese stabilizer}), the abelian varieties $A_s^* = \Alb_s^\vee = (A_s / G)^\vee$ and $A_s$ are not dual in general.
		\end{remark}
		
		\begin{corollary} \label{cor:dimension of Albanese}
			\begin{enumerate}
				\item If $P_s$ is an abelian variety, then $\Alb_s$ has dimension $n$.
				\item If $P_s$ has a linear part $\GG_m$, then $\Alb_s$ has dimension $\le n-1$. Moreover, the equality holds if and only if $f$ has a torsion shear.
				\item If $P_s$ has a linear part $\GG_a$, then $\Alb_s$ has dimension $n-1$.
			\end{enumerate}
		\end{corollary}
		\begin{proof}
			(1) If $P_s$ is abelian, then so is $\check P_s$ and hence is $\Pic^\circ_{X_s}$ by \Cref{prop:dual Neron model is Picard}. (2) Similarly, if $P_s$ has a linear part $\GG_m$ then so is $\Pic^\circ_{X_s}$. Hence $\Alb_s^{\vee} \subset \Pic^\circ_{X_s}$ has dimension at most $n-1$. Note further that $\Pic^\circ_{X_s}$ has a maximal abelian subvariety of dimension $n-1$ $\Longleftrightarrow$ $\check P^\circ_s$ has a maximal abelian subvariety of dimension $n-1$ by \Cref{prop:dual Neron model is Picard} $\Longleftrightarrow$ $[\check P^\circ_s] \in \Ext^1 (\check A_s, \GG_m)$ is torsion by \cite[Lemma~1(b)]{mat-miy66}. (3) This follows from \Cref{prop:splitting and albanese}.
		\end{proof}

	\subsection{Untangles} \label{sec:untangling process}
		The untangling process will be the key to our study of unstable fibrations. We present three versions of it. First, \Cref{lem:untangle primitive} is a primitive version that applies to slightly broader situations. Next, its specialization \Cref{prop:untangle for unstable} is the most useful version that applies to unstable fibrations. Finally, \Cref{lem:untangle for multiple} is a generalization of \Cref{lem:untangle primitive} to multiple fibrations. The last version is a bit more technical to formulate, so we postpone this to \S \ref{sec:classification of multiple fibers}.
		
		\begin{lemma} \label{lem:untangle primitive}
			Let $f : X \to S$ be a minimal $\delta$-regular abelian fibration. Assume that $f$ is non-multiple and $\dim \Alb_s = n-1$. Then the following holds for each isogeny of abelian torsors $u : A \to \Alb_s$.
			\begin{enumerate}
				\item Shrinking $S$ \'etale locally, there exists a unique minimal $\delta$-regular abelian fibration $\tilde f : \tilde X \to S$ with a cartesian diagram
				\[\begin{tikzcd}
					\tilde X \arrow[d, "u"] \arrow[dd, bend right=45, "\tilde f"'] & \tilde X_s \arrow[r] \arrow[l, hook'] \arrow[d] & A \arrow[d, "u"] \\
					X \arrow[d, "f"] & X_s \arrow[r, "\alb"] \arrow[l, hook'] \arrow[d] & \Alb_s \\
					S & \{ s \} \arrow[l, hook']
				\end{tikzcd}.\]
				
				\item Let $G$ be the Galois group of $u$ and $\tilde P$ the t-automorphism scheme of $\tilde f$. Then there exists a left exact sequence $0 \to G \to \tilde P \to P$.
			\end{enumerate}
		\end{lemma}
		\begin{proof}
			(1) Let $G$ be the Galois group of $u$ and $\hat G = \Hom(G, \GG_m)$ its character group. We can realize the isogeny $u$ as
			\[ u : A = \Spec_{\Alb_s} \big( \bigoplus_{\chi \in \hat G} L_{\chi} \big) \to \Alb_s ,\qquad\mbox{where}\quad \hat G \subset \Alb_s^\vee \ \ \mbox{by}\ \ \chi \mapsto [L_{\chi}] .\]
			Note that $\hat G$ is sent to $\Alb_s^\vee = \Pic^\circ_{\Alb_s}$ because the N\'eron--Severi group of $\Alb_s$ is torsion-free, and the map is injective because $A$ is connected. The base change $\tilde X_s \coloneq X_s \times_{\Alb_s} A \to X_s$ is therefore defined by the image $\alb^* (\hat G) \subset \Pic^\circ_{X_s}$ of $\hat G$ sent by an \emph{injective} homomorphism (\Cref{thm:Albanese})
			\[ \Alb_s^\vee \xhookrightarrow{\alb^*} \Pic^\circ_{X_s} .\]
			Since $\Pic^\circ_f$ is smooth, shrinking $S$ if necessary, we can uniquely extend this to torsion sections $\hat G \subset \Pic^\circ_f(S)$.
			
			Take the four-term exact sequence from the Leray spectral sequence
			\[ 0 \to \Pic(S) \to \Pic(X) \to \Pic_f(S) \to \Br(S) = H^2 (S, \GG_m) .\]
			The Picard and Brauer groups of strictly henselian dvrs are trivial \cite[Tag~03QO]{stacks-project}. Therefore, we may uniquely lift $\hat G \subset \Pic^\circ_f(S)$ to $\hat G \subset \Pic(X)$ after shrinking $S$ \'etale locally \cite[\S 2.2.2]{col-the-sko:brauer_group}. In other words, we obtain torsion line bundles $M_\chi$ on $X$ for every $\chi \in \hat G$. The Galois covering $u : \tilde X = \Spec_X \big( \bigoplus_{\chi \in \hat G} M_\chi \big) \to X$ sits in the desired diagram.
			
			At this moment, we need to use the non-multiple fiber assumption to show $\tilde f$ is an abelian fibration. The non-multiple assumption on $f$ implies that $\Pic^\circ_f$ is separated (\Cref{prop:dual Neron model is Picard}). Hence $\hat G \subset \Pic^\circ_f$ is nontrivial over $S_0$, showing that the fibers of $\tilde X_0 \to S_0$ are (connected) abelian torsors. This means $\tilde f : \tilde X \to S$ is a minimal abelian fibration. The $\delta$-regularity will follow from (2).
			
			\medskip
			
			(2) The isogeny $\tilde P_0 \to P_0$ over $S_0$ extends to an isogeny $\tilde P \to P$ over $S$ by \cite[Proposition~7.3.6]{neron}. More precisely, the short exact sequence $0 \to G \to \tilde P_0 \to P_0 \to 0$ over $S_0$ extends to a left exact sequence $0 \to G \to \tilde P \to P$ (e.g., \cite[Proposition~3.29]{kim24}).
		\end{proof}
		
		\begin{remark} \label{rmk:untangle primitive}
			The non-multiple assumption in \Cref{lem:untangle primitive} was used only at the end for the separatedness of $\Pic^\circ_f$. If $f$ is multiple, then $\Pic^\circ_f$ can be non-separated (\Cref{prop:dual Neron model is Picard}--\ref{rmk:dual Neron model is Picard}) and hence some of the torsion sections $\hat G \subset \Pic^\circ_f$ may be trivial over $S_0$. If this is the case, then $\tilde f : \tilde X \to S$ has disconnected fibers over $S_0$ and hence not an abelian fibration. Avoiding this carefully will give a generalization of the lemma that works for multiple fibrations (\Cref{lem:untangle for multiple}).
		\end{remark}
		
		Specializing to unstable fibrations, we have a preferred choice of an isogeny $u : A \to \Alb_s$ in the lemma.
		
		\begin{definition} \label{def:Albanese stabilizer}
			Let $f : X \to S$ be a minimal $\delta$-regular abelian fibration. Assume $f$ is unstable. Then its \emph{Albanese stabilizer} is the stabilizer group $G \subset A_s$ of the $A_s$-action on $\Alb_s$ in \Cref{prop:splitting and albanese}. The \emph{untangling isogeny} of $f$ is the isogeny of abelian torsors with Galois group $G$:
			\[ u : A_s \to \Alb_s .\]
		\end{definition}
		
		\begin{proposition} \label{prop:untangle for unstable}
			In \Cref{lem:untangle primitive}, assume further that $f$ is unstable and take $u : A_s \to \Alb_s$ to be the untangling isogeny in \Cref{def:Albanese stabilizer}. Then the following holds.
			\begin{enumerate}
				\item $\tilde X_s = C \times A_s$ for a proper curve $C$ and the base change of the Albanese morphism is a projection map $\tilde X_s \to A_s$.
				\item Let $G$ be the Albanese stabilizer. Then there exists a short exact sequence
				\[ 0 \to G \to \tilde P \to P \to 0 \quad\mbox{such that}\quad G \cap \tilde P^\circ = 0 .\]
				In particular, over $s \in S$, there exists a short exact sequence
				\[ 0 \to G \to \pi_0 (\tilde P_s) \to \pi_0 (P_s) \to 0 \quad\mbox{and an isomorphism} \ \ \tilde P^\circ_s = P^\circ_s .\]
			\end{enumerate}
		\end{proposition}
		
		\begin{definition}
			Let $f : X \to S$ be an unstable minimal $\delta$-regular abelian fibration. The unstable minimal $\delta$-regular abelian fibration $\tilde f : \tilde X \to S$ in \Cref{prop:untangle for unstable} is called the \emph{untangle} of $f$.
		\end{definition}
		
		\begin{proof}[Proof of \Cref{prop:untangle for unstable}]
			(1) Recall from \Cref{prop:splitting and albanese} that $X_s$ admits an $A_s$-action and $\alb : X_s \to \Alb_s$ is an $A_s$-equivariant locally trivial morphism. Say its fiber is a proper curve $C$. The base change $\tilde X_s \to A_s$ is still $A_s$-equivariant, so the $A_s$-action on $\tilde X_s$ is free. The free $A_s$-quotient of $\tilde X_s$ is an $A_s$-torsor \cite[Tag~06PH]{stacks-project}
			\[ \tilde X_s \to C \cong \tilde X_s / A_s .\]
			Since $C \subset \tilde X_s$ was a fiber, this $A_s$-torsor has a section and hence is trivial.
			
			\medskip
			
			(2) The left exactness of the sequence is proved in \Cref{lem:untangle primitive}. To show its right exactness, take
			\[ \tilde X' = \tilde X \setminus \Sing(\tilde X_s) \quad\mbox{and}\quad X' = X \setminus \Sing(X_s) \]
			as usual and identify $\tilde P = \tilde X'$ and $P = X'$ using the non-multiple fiber assumption and \Cref{prop:P-action existence}. Since $\tilde X' \to X'$ is surjective, $\tilde P \to P$ is surjective as claimed.
			
			It remains to show $G \cap \tilde P^\circ_s = 0$. Both $\alb : X_s \to \Alb_s$ and $u : A_s \to \Alb_s$ are $P_s$-equivariant, so their fiber product $\tilde X_s = C \times A_s$ admits a $P_s$-action and $\tilde X_s \to X_s$ is equivariant. This yields a $P_s$-equivariant, Galois, finite and \'etale morphism
			\[ \tilde X_s' \to X_s' \]
			with Galois group $G$. But the $P_s$-action on $X_s'$ is already free, so this means not only $P_s$ acts on $\tilde X_s'$ freely but the covering $\tilde X_s' \to X_s'$ is a disjoint union of isomorphic copies of $X_s'$. Identifying $P_s = X_s'$ and $\tilde P_s = \tilde X_s'$ shows $|\pi_0(\tilde P_s)| = |\pi_0(P_s)| \cdot |G|$. Therefore, the composition $G \hookrightarrow \tilde P_s \twoheadrightarrow \pi_0 (\tilde P_s)$ is injective, or equivalently $G \cap \tilde P^\circ_s = 0$.
		\end{proof}

\section{Classification of non-multiple central fibers} \label{sec:classification of non-multiple fibers}
	We classify minimal $\delta$-regular abelian fibration with non-multiple fibers in this section. \Cref{thm:classification of semistable fibers}, \ref{thm:classification of unstable fibers}, and \ref{thm:classification of component group} are the three main results. The first two results classify semistable and unstable singular fibers, and the last classifies the component group $\pi_0(P_s)$.

	\subsection{Semistable fibrations} \label{sec:classification of semistable fibers}
		We classify the semistable central fibers in this subsection. For the notion of identifying two subschemes, see \S \ref{sec:conductor}.
		
		\begin{theorem} \label{thm:classification of semistable fibers}
			Let $f : X \to S$ be a minimal $\delta$-regular abelian fibration. Assume $f$ is semistable. Then $X_s$ is isomorphic to one of the following.
			\begin{enumerate}
				\item[\textnormal{($\mathrm I_0$)}] An abelian torsor of dimension $n$.
				\item[\textnormal{($\mathrm I_1$)}] A $\PP^1$-bundle over $A_s$ with two sections $0$ and $\infty$ identified by a translation by the shear $a \in A_s$ in \Cref{def:shear}.
				\item[\textnormal{($\mathrm I_r$)}] A cycle of $r \ (\ge 2)$ isomorphic copies of $\PP^1$-bundles $E_i \to A_s$ whose sections $0_i$ and $\infty_i$ in \Cref{lem:E normalization structure-semistable} are successively identified by $\infty_i = 0_{i+1}$, with an overall translation by the shear $a \in A_s$.
			\end{enumerate}
		\end{theorem}
		
		\begin{definition} \label{def:Kodaira type semistable}
			The \emph{Kodaira type} of a semistable fiber $X_s$ in \Cref{thm:classification of semistable fibers} is as indicated.
		\end{definition}
		
		\begin{remark}
			Our type $\mathrm I_r$ corresponds to Matsushita's type $\tilde A_r$ in Theorem~1.2(2) and Figure~2 (left picture) of \cite{mat01}. Our type $\mathrm I_r$ corresponds to Hwang--Oguiso's \cite{hwang-ogu09} type $\mathrm I_{er}$ or $\mathrm I_{\infty}$, and coincides with \cite[\S 6.2]{EFGMS25}. The former case arises when the shear $a \in A_s$ is torsion of order $e$, and the latter case arises when it is not torsion.
		\end{remark}
		
		\begin{remark} \label{rmk:classification of semistable fibers}
			By \Cref{thm:classification of semistable fibers}, the \emph{slope} and \emph{shear}
			\[ \check a = [P_s^\circ] \in \check A_s = \Ext^1 (A_s, \GG_m) ,\qquad a = [\check P_s^\circ] \in A_s = \Ext^1 (\check A_s, \GG_m) \]
			determine the $\PP^1$-bundle $E_i \to A_s$ and the gluing datum of $E_i$'s, respectively. We note that these two are \emph{dependent} to each other, since they are invariants of the N\'eron models of the dual abelian schemes $P_0$ and $\check P_0$. For example, if the abelian scheme $P_0 \to S_0$ is principally polarized, then we have an isomorphism $P_s \cong \check P_s$ whence the equality of the two extension classes. That is, for each choice of a $\PP^1$-bundle $E_i$, there is a \emph{unique way} of gluing them into a cycle (when $P$ is principally polarized).
		\end{remark}
		
		The rest of this subsection proves \Cref{thm:classification of semistable fibers}.
		
		\begin{lemma} \label{lem:intersection number}
			Let $f : X \to S$ be an arbitrary minimal $\delta$-regular abelian fibration. Assume $X_s$ is reducible.
			\begin{enumerate}
				\item Let $E \subset X_s$ be an irreducible component and $R$ $(=\PP^1$ or $\PP^1 \cup \PP^1)$ the fiber of $p : E \to \bar A$ in \Cref{prop:E}. Then $(E \cdot R) = -2$.
				
				\item Let $E, E' \subset X_s$ be two distinct irreducible components, and $R \subset E$ and $R' \subset E'$ be the fibers of $p$ and $p'$.
				\begin{enumerate}
					\item If $E \cap E' = \emptyset$, then $(E \cdot R') = (E' \cdot R) = 0$.
					\item If $E \cap E' \neq \emptyset$, then $(E \cdot R')$ and $(E' \cdot R)$ are positive.
				\end{enumerate}
			\end{enumerate}
		\end{lemma}
		\begin{proof}
			(1) This follows the idea of \cite[Proposition~4.9]{hwang-ogu09}. The intersection number $(E \cdot R)$ is the degree of the line bundle $\mathcal O_R(E)$, which by the adjunction formula is numerically equivalent to
			\[ \mathcal O_R(E) \equiv \mathcal O_R (K_X + E) = \mathcal O_R(K_E) .\]
			Note that $R \subset E$ has a trivial normal bundle because $p$ is locally trivial, so again by adjunction we have
			\[ \mathcal O_R(K_E) = \mathcal O_R(K_R) \otimes \det N_{R/E}^{\vee} = \mathcal O_R(K_R) .\]
			Both cases $R = \PP^1$ and $\PP^1 \cup \PP^1$ have $\deg K_R = -2$.
			
			\medskip
			
			(2a) is clear. (2b) The intersection $E \cap E'$ is a disjoint union of abelian torsors of dimension $n-1$ by \Cref{lem:P-orbit has dimension n-1}, so it intersects with $R$ and $R'$ at points. Hence $(E \cdot R')$ and $(E' \cdot R)$ are positive.
		\end{proof}
		
		\begin{proof}[Proof of \Cref{thm:classification of semistable fibers} (without the claim about shears)]
			By \Cref{thm:Kollar semistable}, $X_s = \sum_{i=1}^r E_i$ is reduced. If $P_s$ is an abelian variety, then $X_s$ is a $P_s$-torsor by \Cref{prop:P-action existence} (type $\mathrm I_0$). Assume from now on $P_s$ has a linear part $\GG_m$. If $X_s = E$ is integral, then $f$ is a Weierstrass model and hence it is of type $\mathrm I_1$ by \Cref{prop:Weierstrass fiber semistable}.
			
			\medskip
			
			Assume $X_s = \sum_{i=1}^r E_i$ with $r \ge 2$. Every irreducible component $E_i$ admits an $R_i$-fiber bundle morphism $p_i : E_i \to \bar A_i$ with fiber $R_i = \PP^1$ or $\PP^1 \cup \PP^1$ by \Cref{prop:E}. For each $j = 1, \cdots, n$, \Cref{lem:intersection number}(1) yields an identity
			\begin{equation} \label{eq:key formula}
				\sum_{i \neq j} (E_i \cdot R_j) = 2 ,
			\end{equation}
			which was the main ingredient of Kodaira's original proof in \cite[Theorem~6.2, (6.5)]{kod63} and its generalization in \cite[Proposition~4.10--4.12]{hwang-ogu09}. Hwang--Oguiso called it a \emph{key formula}. Set $j = 1$ in \eqref{eq:key formula}. There are only two numerical possibilities:
			\begin{enumerate}[label=\textnormal{(\roman*)}]
				\item $(E_r \cdot R_1) = (E_2 \cdot R_1) = 1$, and $E_i \cap E_1 = \emptyset$ for $i \neq 1, 2, r (\ge 3)$.
				\item $(E_2 \cdot R_1) = 2$, and $E_i \cap E_1 = \emptyset$ for $i > 2$.
			\end{enumerate}
			Before analyzing the cases, notice that every irreducible component $E_i$ is isomorphic to each other. This is because $X_s$ is reduced: by \Cref{prop:P-action existence}, the $P_s$-action on $X_s$ acts transitively on the set of irreducible components of $X_s$. Hence either every $E_i$ is a $\PP^1$-bundle over $A_s$ or every $E_i$ is a $\PP^1 \cup \PP^1$-bundle over $\bar A$, a degree $2$ quotient of $A_s$.
			
			\medskip
			
			(i) Suppose that every $E_i$ is a $\PP^1$-bundle over $A_s$. Note that $E_r \cap E_1$ and $E_1 \cap E_2$ are both $(n-1)$-dimensional orbits by \Cref{lem:P-orbit has dimension n-1}. By \Cref{lem:E normalization structure-semistable}, $E_1 = \tilde E_1$ has precisely two orbits of dimension $n-1$, the two sections $0$ and $\infty$ of $p_1 : E_1 \to A_s$. If $E_r \cap E_1 = E_1 \cap E_2 = 0$ (or $\infty$), then $E_r$, $E_1$, and $E_2$ create three branches around $0$, violating the demi-normality in \Cref{thm:Kollar semistable}. This shows $E_r \cap E_1 = 0$ and $E_1 \cap E_2 = \infty$. Same discussion applies to every component, meaning $X_s = \sum E_i$ is a cycle of isomorphic copies of $\PP^1$-bundles over $A_s$. This is type $\mathrm I_r$ for $r \ge 3$.
			
			Suppose that every $E_i$ is a $\PP^1 \cup \PP^1$-bundle over $\bar A$. The two intersections $E_r \cap E_1$ and $E_1 \cap E_2$ are still $(n-1)$-dimensional orbits in $E_1$, which are different by the demi-normality of $X_s$. According to \Cref{lem:E normalization structure-semistable} and the description of the $\PP^1 \cup \PP^1$-bundle, $E_1$ has two orbits of dimension $n-1$: one isomorphic to $A_s$ and other one isomorphic to $\bar A$, the singular locus of $E_1$ and the only section of $p_1 : E_1 \to \bar A$. None of the possible configurations of $E_r$, $E_1$, and $E_2$ can make $X_s$ demi-normal. Hence this case is impossible.
			
			\medskip
			
			(ii) Suppose first that every $E_i$ is a $\PP^1$-bundle over $A_s$. Then a similar argument shows that $E_1$ and $E_2$ meet twice along two $(n-1)$-dimensional orbits in \Cref{lem:E normalization structure-semistable}. This is type $\mathrm I_2$. Suppose that every $E_i$ is a $\PP^1 \cup \PP^1$-bundle over $\bar A$. Again a similar argument leaves only one possibility: $X_s = E_1 + E_2$ where $E_1$ and $E_2$ intersect along the $(n-1)$-dimensional orbits $A_s \subset E_1$ and $A_s \subset E_2$. Note that the equality $(E_2 \cdot R_1) = 2$ holds because $R_1 = \PP^1 \cup \PP^1$ meets with $A_s$ twice. It remains to rule out this possibility.
			
			\medskip
			
			Since the intersection $E_1 \cap E_2$ is isomorphic to $A_s$, the two $\PP^1 \cup \PP^1$-bundles $p_1$ and $p_2$ patch together and yields a morphism $\alb : X_s \to \bar A$, the Albanese morphism of $X_s$ (use the diagram in \Cref{prop:E} and \S \ref{sec:conductor}). Applying \Cref{lem:untangle primitive} to the degree $2$ isogeny $A_s \to \bar A$, we obtain a new semistable, minimal, and $\delta$-regular abelian fibration $\tilde f : \tilde X \to S$ whose central fiber $\tilde X_s$ has four irreducible components. Hence it must be of Kodaira type $\mathrm I_4$ by the discussion above. The covering $\tilde X_s \to X_s$ has Galois group $G = \ZZ/2 \subset \tilde P_s$. Since $\Sing(\tilde X_s) = \bigsqcup^4 A_s$ and $\Sing(X_s) = A_s \sqcup \bar A \sqcup \bar A$ for $\bar A = A_s / a$, the $\ZZ/2$-action on $\Sing(\tilde X_s)$ swaps two $A_s$'s and preserves two $A_s$'s. This contradicts the following lemma (recall $G \cap P_s^\circ = 0$ by \Cref{prop:untangle for unstable}), ruling out the last case.
		\end{proof}
		
		\begin{lemma} \label{lem:stabilizer of Ir}
			Assume $f : X \to S$ has type $\mathrm I_r$ for $r \ge 1$. Then every nontrivial stabilizer of the $P_s$-action on $X_s$ is $\GG_m$.
		\end{lemma}
		\begin{proof}
			Write $Z = \Sing(X_s) = \bigsqcup_{i=1}^r (E_i \cap E_{i+1}) \cong A_s^{\sqcup r}$ (index taken mod $r$). The group $P_s$ acts on $X_s' = X_s \setminus Z$ freely and on $Z$ with nontrivial stabilizers (\Cref{prop:P-action property} and \ref{prop:P-action existence}). Let us first show that the stabilizers of $Z$ are subgroups of $P_s^\circ$. Assume on the contrary that $f \in P_s \setminus P_s^\circ$ preserves a component $E_r \cap E_1$. Then the cyclic configuration of $X_s$ forces $r$ to be even and $f(E_i) = E_{r+1-i}$ for all $i$. Hence $f^2 \in P_s$ preserves all $E_i$, or $f^2 \in P_s^\circ$. Taking any $g \in P_s^\circ$ with $g^2 = f^2$ (because $P_s^\circ$ is divisible) and adjusting $f$ by $f - g$, we may assume $f^2 = \id$. In this case, $f$ has a fixed locus $E_r \cap E_1$ and $E_{r/2} \cap E_{r/2+1}$.
			
			Now take a degree $2$ quotient $\bar f : \bar X = X / f \to S$. Around a point $x \in E_r \cap E_1$ (resp. $E_{r/2} \cap E_{r/2+1}$), $f$ fixes $E_r \cap E_1$ and swaps the nearby components $E_1$ and $E_r$. Hence the $f$-fixed locus of $T_x X \cong k^{n+1}$ is $n$-dimensional. The Chevalley--Shepherd--Todd theorem shows $\bar X$ is regular, and hence $\bar f$ is a semistable minimal $\delta$-regular abelian fibration. But $\bar X_s = X_s / f$ is a \emph{chain} of $r/2$ $\PP^1$-bundles over $A_s$. This falls into neither of the case (i) or (ii) derived from \eqref{eq:key formula}, a contradiction.
			
			Therefore, every stabilizer is contained in $P_s^\circ$ and is linear, so it has a neutral component $\GG_m$. But the connected components of $Z = \Sing(X_s)$ are isomorphic to $A_s$, so the stabilizers must be precisely $\GG_m$.
		\end{proof}
		
		\begin{remark} \label{rmk:intersection number}
			The intersection numbers $(E_i \cdot R_j)$ and $(E_j \cdot R_i)$ in \eqref{eq:key formula} may a priori be different. Hwang--Oguiso's strategy was to make this equality hold by defining a symmetric bilinear form $q(R_i, R_j)$ on the rational curves $R_i$, but allowing possibly infinite chains of them due to a potential non-torsion shear. See the proof of \cite[Proposition~4.10]{hwang-ogu09}. Our proof dealt with the possible discrepancy $(E_i \cdot R_j) \neq (E_j \cdot R_i)$. \emph{A posteriori}, $(E_i \cdot R_j) = (E_j \cdot R_i)$ turns out to hold in the semistable fibrations. This is no longer the case for unstable fibrations due to a potential tangle.
		\end{remark}

		\subsubsection{Gluing and shear}
			To complete the proof of \Cref{thm:classification of semistable fibers}, it remains to show that the shear $a \in A_s$ (the extension class of $[\check P_s^\circ] \in \Ext^1 (\check A_s, \GG_m) = A_s(k)$) is equal to the gluing datum described in \Cref{thm:classification of semistable fibers}. This claim will be about a single fiber $X_s$ but not a degeneration $f$, so to ease the notation, let us prove this in the following setup.
			
			\medskip
			
			Let $X$ be a \emph{type $\mathrm I_r$ variety} for $r \ge 1$ in \Cref{thm:classification of semistable fibers}, a reduced projective scheme over $k$ obtained by identifying the sections $\infty_i$ and $0_{i+1}$ in $r$ isomorphic copies of $\PP^1$-bundles $p : E_i \to A$ over an abelian torsor $A$ of dimension $n-1$. Take a normalization of $X$ along the node $E_r \cap E_1$ but nothing more. Such a partial resolution $\nu : \tilde X \to X$ admits a locally trivial and equivariant Albanese morphism
			\[ \alb : \tilde X \to A ,\]
			whose fiber $C$ is a \emph{chain} of $r$ number of $\PP^1$'s. The original $X$ is obtained back by gluing two sections (see \S \ref{sec:conductor})
			\[ 0 = 0_1 , \ \infty = \infty_r : A \to \tilde X, \qquad 0_A = 0(A) , \ \ \infty_A = \infty(A) \subset \tilde X .\]
			More precisely, the gluing is determined by a morphism
			\begin{equation} \label{eq:contraction with shear}
				0_A \sqcup \infty_A \to A \quad\mbox{combining}\quad \alb : 0_A \to A ,\quad t_a \circ \alb : \infty_A \to A ,
			\end{equation}
			where $t_a : A \to A$ is a translation by \emph{some} element $a \in A$. Together with \Cref{prop:dual Neron model is Picard}, the following proposition concludes \Cref{thm:classification of semistable fibers}.
			
			\begin{proposition} \label{prop:shear}
				Let $X$ be a type $\mathrm I_r$ variety as above, using the contraction morphism \eqref{eq:contraction with shear} determined by $a \in A$. Then the pullback under the morphism $\nu : \tilde X \to X$ defines a short exact sequence
				\[ 0 \to \GG_m \to \Pic^\circ_X \xlongrightarrow{\nu^*} \Pic^\circ_{\tilde X} = \check A \to 0 ,\]
				whose extension class is precisely $a \in \Ext^1(\check A, \GG_m) = A$.
			\end{proposition}
			\begin{proof}
				We write $X_a = X$ to emphasize that its construction depends on the choice of $a \in A$. Consider the Poincar\'e line bundle on $\check A \times A$ rigidified along $\{ 0 \} \times A$ and $\check A \times \{ 0 \}$. Such a line bundle defines a universal $\GG_m$-torsor $B \to \check A \times A$ ($\GG_m$-biextension) whose restriction to $\check A_a \coloneq \check A \times \{ a \}$ is a $\GG_m$-torsor $0 \to \GG_m \to B_a \to \check A_a \to 0$ with an extension class $a \in A$. We claim that the two $\GG_m$-torsors $\Pic^\circ_{X_a}$ and $B_a$ over $\check A_a$ are isomorphic.
				
				Suppose first that $a \in A$ is a torsion point of order $e$. Use the universal property of pushouts (\Cref{prop:conductor pushout}) to construct a commutative diagram
				\[\begin{tikzcd}
					0_A \sqcup \infty_A \arrow[r, phantom, "\subset"] \arrow[d] & \tilde X_a \arrow[r, "\alb"] \arrow[d, "\nu"] & A \arrow[d] \\
					A \arrow[r, phantom, "\subset"] & X_a \arrow[r, "\alpha"] & A/a
				\end{tikzcd}.\]
				Here $\alpha : X_a \to A/a$ is a locally trivial morphism with a fiber $\tilde C$, a Kodaira singular curve of type $\mathrm I_{er}$. Hence $\alpha^* : (A/a)^\vee \to \Pic_{X_a}^\circ$ is an injective homomorphism. Since $\alb^* : \check A \to \Pic^\circ_{\tilde X_a}$ is an isomorphism, the composition $\alpha^* \circ \nu^*$ is an isogeny $(A/a)^\vee \hookrightarrow \Pic_{X_a}^\circ \twoheadrightarrow \check A$. This means the $\GG_m$-torsor $\Pic_{X_a}^\circ \to \check A$ has an $(A/a)^\vee$-section, or equivalently the short exact sequence $0 \to \GG_m \to \Pic_{X_a}^\circ \to \check A \to 0$ is defined by the extension class $a \in A$.
				
				\medskip
				
				To prove the statement for arbitrary $a \in A$, notice first that our construction of $X = X_a$ for each $a \in A$ forms a universal family as follows. Start with a trivial family of $\tilde X$ with two sections $0$ and $\infty$ over $A$:
				\[\begin{tikzcd}[column sep=normal]
					\tilde X \times A \arrow[r, "\alb \times \id"] \arrow[rd, "\pr_2"'] & A \times A \arrow[d, "\pr_2"] \\
					& A
				\end{tikzcd}, \qquad\quad 0, \infty : A \times A \to \tilde X \times A .\]
				Identifying the two sections by $(x,a) \sim (x+a,a)$ as in \eqref{eq:contraction with shear}, we obtain a contraction $\nu : \tilde X \times A \to \mathfrak X$ over $A$, whose fiber over $a \in A$ constructs $X_a$:
				\[\begin{tikzcd}[column sep=normal]
					\tilde X \times A \arrow[r, "\alb \times \id"] \arrow[d, "\nu"] & A \times A \arrow[d, "\pr_2"] \\
					\mathfrak X \arrow[r] & A
				\end{tikzcd}.\]
				The relative neutral Picard scheme of $\mathfrak X \to A$ is a $\GG_m$-torsor $\nu^* : \Pic_{\mathfrak X/A}^\circ \to \check A \times A$. Comparing this with the Poincar\'e $\GG_m$-torsor $B \to \check A \times A$, the two are isomorphic over $\check A_s = \check A \times \{ a \}$ for every torsion point $a \in A$ by the previous paragraph. The seesaw principle (e.g., Proposition in \cite[\S 10]{mum:abelian_variety}) concludes that they are isomorphic for every $a \in A$.
			\end{proof}

	\subsection{Unstable fibrations} \label{sec:classification of unstable fibers}
		Here is a classification of unstable fibers.
		
		\begin{theorem} \label{thm:classification of unstable fibers}
			Let $f : X \to S$ be a minimal $\delta$-regular abelian fibration. Assume $f$ is unstable. Then $\alb : X_s \to \Alb_s$ is equivariantly isomorphic to
			\[ (C \times A_s) / G \to A_s / G ,\]
			where $C$, $G$, and the action are classified as follows.
			\begin{itemize}
				\item $C$ is an unstable Kodaira singular curve: it is of type $\mathrm {II}$, $\mathrm{III}$, $\mathrm{IV}$, $\mathrm{I}_r^*$ for $r \ge 0$, $\mathrm{II}^*$, $\mathrm{III}^*$, or $\mathrm{IV}^*$.
				
				\item The Albanese stabilizer $G$ is classified as in \Cref{table:classification of G}.
				
				\item $G$ acts on $C \times A_s$ diagonally: it acts on $A_s$ freely by translations, and acts on $C$ in a unique way.
			\end{itemize}
		\end{theorem}
		
		\begin{table}[t]
			\begin{tabular}{|c||c|cc|cc|ccc|} \hline
				Kod. type of $C$ & $\mathrm {II}$ &  \multicolumn{2}{|c|}{$\mathrm {III}$} & \multicolumn{2}{|c|}{$\mathrm {IV}$} & \multicolumn{3}{|c|}{$\mathrm I_r^*$ (even $r$)} \\ \hline
				Alb. stab. $G$ & $0$ & $0,$ & $\ZZ/2$ & $0,$ & $\ZZ/3$ & $0,$ & $\ZZ/2,$ & $(\ZZ/2)^{\times 2}$ \\ \hline\hline
				Kod. type of $C$ & $\mathrm {II}^*$ & \multicolumn{2}{|c|}{$\mathrm {III}^*$} & \multicolumn{2}{|c|}{$\mathrm {IV}^*$} & \multicolumn{3}{|c|}{$\mathrm I_r^*$ (odd $r$)} \\ \hline
				Alb. stab. $G$ & $0$ & $0,$ & $\ZZ/2$ & $0,$ & $\ZZ/3$ & $0,$ & $\ZZ/2,$ & $\ZZ/4$ \\ \hline
			\end{tabular}
			
			\medskip
			
			\caption{Classification of the Albanese stabilizer $G$ for each unstable Kodaira type curve $C$. The \emph{Kodaira type} of the unstable fiber $X_s = (C \times A_s)/G$ is (Kodaira type of $C$)/(order of $G$).}
			\label{table:classification of G}
		\end{table}
		
		\begin{definition} \label{def:Kodaira type unstable}
			The \emph{Kodaira type} of an unstable fiber $X_s$ in \Cref{thm:classification of unstable fibers} is denoted by
			\[ (\mbox{Kodaira type of the curve } C) / (\mbox{order of } G) .\]
			For example, when $C$ has type $\mathrm{IV}$ and $G = 0$, we call it a type $\mathrm {IV}$ fiber. When $C$ has type $\mathrm{IV}$ and $G = \ZZ/3$, we call it a type $\mathrm{IV}/3$ fiber. When $C$ has type $\mathrm I_0^*$ and $G = (\ZZ/2)^{\times 2}$, we call it a type $\mathrm I_0^*/4$ fiber.
		\end{definition}
		
		\begin{remark}
			The Albanese morphism $\alb : X_s \to \Alb_s$ is a locally trivial fiber bundle with fiber $C$. Its fibers are called \emph{characteristic cycles} in \cite{hwang-ogu09}. \Cref{thm:classification of unstable fibers} was proved in \cite[Corollary~4.18--19]{kim-laza-mar26} under the assumption that $f$ is isotrivial. Comparing it with Matsushita's result, we note that some types in \cite[Table~4--5]{mat01} (e.g., $\mathrm I_0^*$-1, $\mathrm I_0^*$-3, $\mathrm {IV}$-1, and $\mathrm{IV}^*$-1) only arise as multiple fibers, therefore not appearing in \Cref{thm:classification of unstable fibers} but only in \Cref{thm:classification of multiple fiber v}. Finally, when $C$ has type $\mathrm I_r^*$ for $r \ge 2$ even, the N\'eron component group $\pi_0 (\tilde P_s) = (\ZZ/2)^{\times 2}$ in \Cref{table:classification of Neron component group} has three different subgroups $\ZZ/2$, but our notation does not distinguish them.
		\end{remark}
		
		\begin{proof}[Proof of \Cref{thm:classification of unstable fibers} (without the classification of Albanese stabilizers)]
			Take an untangle $\tilde X \to X$ in \Cref{prop:untangle for unstable}, so that $\tilde f : \tilde X \to S$ has a central fiber $\tilde X_s = C \times A_s$. We have $X_s \cong \tilde X_s / G$ for the Albanese stabilizer $G$ in \Cref{def:Albanese stabilizer}. Denote by $\tilde P$ the t-automorphism scheme of $\tilde f$.
			
			\medskip
			
			Let us first classify the curve $C$. If $C$ is integral, then $\tilde P^\circ_s = \GG_a \times A_s$ acts on the integral variety $\tilde X_s = C \times A_s$ generically freely and with nontrivial stabilizers $\GG_a$. Hence \Cref{prop:prop:Weierstrass fiber unstable} determines $C$: its unique singular point $x$ is locally defined by $\hat {\mathcal O}_{C, x} = k[[t^{\delta+1}, \cdots]]$. The constant $\delta$ is captured by the dimension of the linear part of $\Pic^\circ_C$, which is $1$ by the $\delta$-regularity assumption. This concludes that $C$ is isomorphic to the cuspidal cubic curve.
			
			Assume now $C$ is reducible and write $\tilde X_s = C \times A_s = \sum_{i=1}^r a_i E_i$. Notice that every irreducible component $E_i$ is isomorphic to $\PP^1 \times A_s$. Fix a curve $C \subset \tilde X_s$ and set $R_i = C \cap E_i$, a rational curve in $E_i$. We claim $(E_i \cdot R_j) = (E_j \cdot R_i)$ for every $1 \le i, j \le r$. The variety $\tilde X_s = C \times A_s$ has an embedding dimension $n+1$ everywhere, so $C$ has an embedding dimension $2$. For each point $p \in R_i \cap R_j \subset C$, we can assume $(C, p) \subset (\AA^2, 0)$ and hence the local intersection multiplicity at $p$ is
			\[ (E_i \cdot R_j)_{(\tilde X, p)} = (R_i \cdot R_j)_{(\AA^2,0)} = (E_j \cdot R_i)_{(\tilde X, p)} ,\]
			proving the desired equality $(E_i \cdot R_j) = (E_j \cdot R_i)$. (Note that this completely fails in $X$. That is, an irreducible component $\bar E \subset X_s$ may be a $\PP^1 \cup \PP^1$-bundle in \Cref{prop:E}(ii) and the above numerical equality may fail for $\bar E_i, \bar E_j \subset X_s$.)
			
			Since we have an equality $(E_i \cdot R_j) = (E_j \cdot R_i)$, we can pretend $q(E_i, E_j) = (E_i \cdot R_j)$ is a symmetric bilinear form and use the numerical equality
			\[ \sum_{i \neq j} a_i \cdot q(E_i, E_j) = 2 \qquad\mbox{for}\quad j = 1, 2, \cdots, n \]
			as in \Cref{lem:intersection number} and \eqref{eq:key formula}. Now Kodaira's original proof \cite[Theorem~6.2, (6.5)]{kod63} applies and computes the desired multiplicities $a_i$ and configurations of $E_i$'s as claimed. Note that Kodaira's argument is purely numerical with the following exceptions, as already pointed out in \cite[p.~1003]{hwang-ogu09}:
			\begin{itemize}
				\item Kodaira's Lemma 6.1 gives a geometric criterion for $\tilde X_s$ to be non-multiple. We are simply assuming $\tilde X_s$ is non-multiple within this theorem.
				\item Kodaira's equation (6.4) is used to compute the arithmetic genera of irreducible components $E \subset \tilde X_s$. We have already proved $E = \PP^1 \times A_s$.
			\end{itemize}
			This classifies the curve $C$ as claimed.
			
			\medskip
			
			It remains to classify $G$ and its action on $\tilde X_s = C \times A_s$. The action is by t-automorphisms from the inclusion $G \subset P_s$ in \Cref{prop:untangle for unstable}. By construction in \Cref{lem:untangle primitive}, the action on the second factor $A_s$ is a translation by $G \subset A_s$, so the $G$-action on $C \times A_s$ is of the form
			\[ g \cdot (x,y) = (g_y(x), y + g) \qquad\mbox{for}\quad (x,y) \in C \times A_s, \quad g_y : C \to C .\]
			This yields a homomorphism $A_s \to \Aut_C$ by $y \mapsto g_y$, which is necessarily constant as in \Cref{lem:E normalization structure-unstable} because all irreducible components of $C$ are rational curves (\Cref{lem:automorphism of chain of rational curves}). Hence the $G$-action is diagonal.
			
			Finally, recall from \Cref{prop:untangle for unstable}(2) the inclusion $G \subset \pi_0(\tilde P_s)$. Hence it only remains to compute $\pi_0(\tilde P_s)$. This will be done in \Cref{thm:classification of component group}(2) and \S \ref{sec:Neron component groups and Albanese stabilizers}. Note that $|\pi_0(\tilde P_s)|$ is the number of reduced components of $C$ because of the isomorphism $\tilde X_s' = \tilde X_s \setminus \Sing(\tilde X_s) \cong \tilde P_s$ (\Cref{prop:P-action existence}). Hence we already know $|\pi_0(\tilde P_s)| = 1$, $2$, $3$, or $4$, and the only ambiguity arises in the last case where $\pi_0(\tilde P_s)$ can be either $\ZZ/4$ or $(\ZZ/2)^{\times 2}$.
		\end{proof}

	\subsection{N\'eron component groups}
		If a minimal $\delta$-regular abelian fibration $f : X \to S$ has a section, then $X' = X \setminus \Sing(X_s)$ is isomorphic to the t-automorphism scheme $P$ by \Cref{prop:P-action existence}. Conversely, let $P \to S$ be a $\delta$-regular N\'eron model of an abelian scheme over $S_0$. Then it admits a unique minimal model $f : X \to S$ by \Cref{thm:minimal model} and \Cref{prop:minimal delta-regular abelian fibration is unique} with an isomorphism $X' \cong P$.
		
		\begin{definition} \label{def:Kodaira type of t-automorphism group}
			Let $P \to S$ be a $\delta$-regular N\'eron model of an abelian scheme over $S_0$.
			\begin{enumerate}
				\item The \emph{minimal compactification} of $P$ is a unique minimal $\delta$-regular abelian fibration $f : X \to S$ with $X' = X \setminus \Sing(X_s) \cong P$.
				\item The \emph{Kodaira type} of $P$ is the Kodaira type of its minimal compactification.
			\end{enumerate}
		\end{definition}
		
		The main result of this subsection is the computation of the N\'eron component group of $P$, generalizing Kodaira \cite[Table~1 in p.604]{kod63} and N\'eron \cite[Table in pp.124--125]{neron:original_article} into higher dimensions.
		
		\begin{theorem} \label{thm:classification of component group}
			Let $P \to S$ be a $\delta$-regular N\'eron model of an abelian scheme over $S_0$.
			\begin{enumerate}
				\item If $P$ is semistable of type $\mathrm I_r$, then $\pi_0(P_s)$ is isomorphic to $\ZZ/r$.
				\item If $P$ is unstable, then $\pi_0(P_s)$ is classified in \Cref{table:classification of Neron component group}. In particular, it is either a cyclic group of order $\le 4$ or $(\ZZ/2)^{\times 2}$
				\item If $\check P \to S$ is the N\'eron model of the dual abelian scheme $\check P_0$, then $\pi_0(P_s)$ and $\pi_0(\check P_s)$ are isomorphic.
			\end{enumerate}
		\end{theorem}
		
		\begin{table}[t]
			\begin{tabular}{|c||c|c|c|c|c|c|} \hline
				Type of $X_s$ & $\mathrm I_r$ & $\mathrm{II}$ & \multicolumn{1}{c}{$\mathrm{III}$} & $\mathrm{III}/2$ & \multicolumn{1}{c}{$\mathrm{IV}$} & $\mathrm{IV}/3$ \\ \hline
				$\pi_0(P_s)$ & $\ZZ/r$ & $0$ & \multicolumn{1}{c}{$\ZZ/2$} & $0$ & \multicolumn{1}{c}{$\ZZ/3$} & $0$ \\ \hline\hline
				Type of $X_s$ & \multirow{2}{*}{-} & $\mathrm{II}^*$ & \multicolumn{1}{c}{$\mathrm{III}^*$} & $\mathrm{III}^*/2$ & \multicolumn{1}{c}{$\mathrm{IV}^*$} & $\mathrm{IV}^*/3$ \\ \cline{1-1} \cline{3-7}
				$\pi_0(P_s)$ & & $0$ & \multicolumn{1}{c}{$\ZZ/2$} & $0$ & \multicolumn{1}{c}{$\ZZ/3$} & $0$ \\ \hline\hline
				Type of $X_s$ & \multicolumn{1}{c}{$\mathrm I_{\operatorname{ev}}^*$} & \multicolumn{1}{c}{$\mathrm I_{\operatorname{ev}}^*/2$} & $\mathrm I_{\operatorname{ev}}^*/4$ & \multicolumn{1}{c}{$\mathrm I_{\operatorname{odd}}^*$} & \multicolumn{1}{c}{$\mathrm I_{\operatorname{odd}}^*/2$} & $\mathrm I_{\operatorname{odd}}^*/4$ \\ \hline
				$\pi_0(P_s)$ & \multicolumn{1}{c}{$(\ZZ/2)^{\times 2}$} & \multicolumn{1}{c}{$\ZZ/2$} & $0$ & \multicolumn{1}{c}{$\ZZ/4$} & \multicolumn{1}{c}{$\ZZ/2$} & $0$ \\ \hline
			\end{tabular}
			
			\medskip
			
			\caption{Classification of the N\'eron component group $\pi_0(P_s)$ for each Kodaira type of $X_s$. Here $X$ is a compactification of $P$ as a minimal abelian fibration.}
			\label{table:classification of Neron component group}
		\end{table}
		
		\begin{remark} \label{rmk:classification of component group}
			The classification of the Albanese stabilizers in \Cref{thm:classification of unstable fibers} is not completed at the moment. If we classify $\pi_0(P_s)$ when the minimal compactification $X$ has no Albanese stabilizers, then this will conclude both \Cref{thm:classification of component group}(2) and the remaining claims of \Cref{thm:classification of unstable fibers}.
			
			If $P$ is semistable or unstable with a trivial Albanese stabilizer, then $\pi_0(P_s)$ is isomorphic to the discriminant group of the Cartan matrix of ADE affine Dynkin diagrams. This Cartan matrix is the \emph{intersection matrix} of the minimal compactification $\check f: \check X \to S$ of the dual N\'eron model $\check P$.
			
			Though $P$ and $\check P$ have isomorphic N\'eron component groups, they \emph{may not} have the same Kodaira type. For example, we will see that type $\mathrm I_r^*/2$ and $\mathrm I_{2r}^*/2$ are dual to each other when $r$ is odd.
		\end{remark}
		
		The rest of this subsection proves \Cref{thm:classification of component group} when $P$ and $\check P$ are semistable. We will follow the idea of \cite[Theorem~9.6.1]{neron} and \cite[\S 10.4.2]{liu:ag}, but note that $P$ and $\check P$ are not isomorphic in our case (they are no longer Jacobians of a family of curves). The proof of the unstable case of \Cref{thm:classification of component group} will be postponed to the next section, after we discuss \Cref{prop:semistable reduction theorem untangled ver} and its consequences.

		\subsubsection{Semistable case}
			\Cref{thm:classification of component group} is clear when $P$ is an abelian scheme, so let us assume that $P$ is strictly semistable. Let $f : X \to S$ be its minimal compactification, whose central fiber $X_s = \sum_{i=1}^r E_i$ is strictly semistable of type $\mathrm I_r$ for $r \ge 1$.
			
			\begin{lemma} \label{lem:classification of component group:semistable NS}
				For a type $\mathrm I_r$ fiber $X_s$, we have an isomorphism
				\[ \NS(X_s) = \ZZ^r \times \NS(A_s) .\]
				In particular, $\NS(X_s)$ is torsion-free.
			\end{lemma}
			\begin{proof}
				The singular locus $Z = \Sing(X_s)$ is a disjoint union of $r$ isomorphic copies of $A_s$, and the normalization $\nu : \tilde X_s \to X_s$ is a disjoint union of $r$ isomorphic copies of a \emph{Zariski-locally trivial} $\PP^1$-bundle over $A_s$. The exact sequence $0 \to \GG_{m, X_s} \to \nu_* \GG_{m, \tilde X_s} \to \GG_{m, Z} \to 0$ of fppf abelian sheaves on $X_s$ induces
				\[ 0 \to \GG_m \to \Pic_{X_s} \xlongrightarrow{\nu^*} \Pic_{\tilde X_s} = (\ZZ \times \Pic_{A_s})^{\times r} \to \Pic_Z = \Pic_{A_s}^{\times r} ,\]
				where the last homomorphism sends $(L_i) \in \Pic_{A_s}^{\times r}$ to $(L_i - L_{i+1}) \in \Pic_{A_s}^{\times r}$. Its kernel is isomorphic to $\Pic_{A_s}$ and this concludes the lemma.
			\end{proof}
			
			Therefore, for each line bundle $L$ on $X_s$, we can define
			\[ c_1(L) = (\deg L, \ [L]) \quad \in \ \ \NS(X_s) = \ZZ^r \times \NS(A_s) .\]
			The part $\deg L \in \ZZ^r$ is often called the \emph{multidegree} of $L$. It is a collection of its \emph{$E_i$-degree} for $i = 1, \cdots, r$, an integer $(L \cdot R_i)$ for a rational curve $R_i$ in the $\PP^1$-bundle $E_i$. The \emph{total degree} of $L$ is the sum of all $E_i$-degrees. Two line bundles $L$ and $L'$ are \emph{numerically equivalent} if $c_1(L) = c_1(L')$.
			
			\begin{definition}
				The \emph{intersection matrix} of $X_s$ is an $r \times r$ symmetric matrix over $\ZZ$ whose rows are the multidegrees of $\mathcal O_{X_s} (E_i)$. In other words, its $(i,j)$-th entry is the intersection number $q(E_i, E_j) \coloneq (E_i \cdot R_j)$, and it is the Cartan matrix of the type $\tilde A_{r-1}$ affine Dynkin diagram by \Cref{thm:classification of semistable fibers}.
			\end{definition}
			
			It is convenient to fix a special line bundle $M$ on $X_s$. When $r \ge 2$, use \Cref{lem:classification of component group:semistable NS} to fix a line bundle $M$ with
			\[ \deg M = (0, \cdots, 0, -1, 1) \in \ZZ^r \quad\mbox{and}\quad [M] = 0 \in \NS(A_s) .\]
			When $r = 1$, simply take $M = \mathcal O_{X_s}$.
			
			\begin{lemma} \label{lem:classification of component group:semistable 1}
				$M^{\otimes r}$ is numerically equivalent to $\mathcal O_{X_s} (E_1 + 2E_2 + \cdots + (r-1)E_{r-1})$.
			\end{lemma}
			\begin{proof}
				Every $\mathcal O_{X_s} (E_i)$ has a multidegree $(0, \cdots, 0, 1, -2, 1, 0, \cdots 0) \in \ZZ^r$ (where $-2$ is at the $i$-th entry) and has trivial $\NS(A_s)$.
			\end{proof}
			
			\begin{lemma} \label{lem:classification of component group:semistable 2}
				Let $L$ be a line bundle on $X_s$ with a total degree $0$ and $[L] = 0 \in \NS(A_s)$. Then there exists a unique integer $0 \le i < r$ such that $L \otimes M^{\otimes -i}$ is numerically equivalent to a linear combination of $\mathcal O_{X_s} (E_i)$'s.
			\end{lemma}
			\begin{proof}
				The intersection matrix of $X_s$ defines a semi-negative symmetric bilinear form on $\ZZ^r = \ZZ \{ E_i : i = 1, \cdots, r \}$, or equivalently a linear map of rank $r-1$:
				\[ \deg : \ZZ^r = \ZZ \{ E_i \} \to \ZZ^r = \ZZ \{ E_i^\vee \} ,\qquad E_i \mapsto q(E_i, -) .\]
				Since $\mathcal O_{X_s}(E_i)$'s have a total degree $0$, the image of this homomorphism is contained in the kernel of the summation map $\ZZ^r \to \ZZ$. The cohomology group of the complex $\ZZ^r \to \ZZ^r \to \ZZ$ is the \emph{discriminant group} of the $\tilde A_r$ Cartan matrix
				\[ D(\tilde A_r) \cong \ZZ/r = \langle -E_{r-1}^\vee + E_r^\vee \rangle .\]
				This proves the claim.
			\end{proof}
			
			\begin{lemma} \label{lem:classification of component group:semistable 3}
				Let $L$ be a line bundle on $X$. Then $L_s = L_{|X_s}$ has a total degree $0$ and $[L_s] = 0 \in \NS(A_s)$ if and only if $L$ defines a section of $\Pic^\bullet_f$.
			\end{lemma}
			\begin{proof}
				($\Longleftarrow$) Recall from \Cref{prop:dual Neron model is Picard} the isomorphism $\check P = \Pic^\bullet_f / E$, where $E \subset \Pic^\bullet_f$ is the closure of the identity section. Since the N\'eron model $\check P$ is quasi-projective, a sufficient power of $L$ is numerically equivalent to a line bundle contained in $E$. A line bundle contained in $E$ is trivial over $S_0$, so it is necessarily a linear combination of $\mathcal O_X(E_i)$'s. Each $\mathcal O_{X_s}(E_i)$ has a trivial class in $\NS(A_s)$ and total degree $0$. Hence, so is the power of $L$. Now use the fact that $\NS(X_s) = \ZZ^r \times \NS(A_s)$ is torsion-free.
				
				($\Longrightarrow$) By \Cref{lem:classification of component group:semistable 1}--\ref{lem:classification of component group:semistable 2}, $L_s^{\otimes r}$ is numerically equivalent to a linear combination of $\mathcal O_{X_s} (E_i)$. The latter is a section of $E \subset \Pic^\bullet_f$, so $L^{\otimes r}$ defines a section of $\Pic^\bullet_f$. Since $\pi_0 (\Pic_{X_\eta}) = \NS(X_{\eta})$ is torsion-free for the generic point $\eta$ (because $X_\eta$ is an abelian torsor), $L$ defines a section of $\Pic^\bullet_f$.
			\end{proof}
			
			\begin{proof}[Proof of \Cref{thm:classification of component group}: semistable case]
				Assume that $X_s = \sum_{i=1}^r E_i$ is strictly semistable of type $\mathrm I_r$ for $r \ge 1$. \Cref{lem:classification of component group:semistable 2}--\ref{lem:classification of component group:semistable 3} and \Cref{prop:dual Neron model is Picard} show the isomorphism $\pi_0 (\check P_s) \cong \ZZ/r$ (generated by $M$). To compute $\pi_0(P_s)$, let us consider the dual N\'eron model $\check P \to S$ and its compactification $\check f : \check X \to S$. This is strictly semistable and hence necessarily of type $\mathrm I_{r'}$ for some $r' \ge 1$. Applying \Cref{lem:classification of component group:semistable 2}--\ref{lem:classification of component group:semistable 3} to $\check P$ and $\check f$ instead, we conclude $\pi_0(P_s) \cong \ZZ/r'$. But recall from \Cref{prop:P-action existence} that $X'_s = X_s \setminus \Sing(X_s)$ is a $P_s$-torsor. This shows $P_s$ has precisely $r$ connected components and shows $r = r'$.
			\end{proof}
			
			\begin{corollary} \label{cor:pi_0-action on X_s}
				Assume $f$ is strictly semistable of type $\mathrm I_r$. Then $\pi_0(P_s) \cong \ZZ/r$ acts on $X_s = \sum_{i=1}^r E_i$ by cycling the components $E_i$.
			\end{corollary}
			\begin{proof}
				Fix any $f \in P_s$ that generates the component group $\pi_0(P_s) \cong \ZZ/r$. Then $f$ permutes $E_i$ without fixing any of them because $X_s'$ is a $P_s$-torsor. Moreover, by \Cref{lem:stabilizer of Ir}, $f$ permutes the connected components of $\Sing(X_s) = \bigsqcup_{i=1}^r (E_i \cap E_{i+1})$ without fixing any of them. The only possibility is $f(E_i) = E_{i+1}$ for all $i$.
			\end{proof}

		\subsubsection{Unstable case}
			Similar idea applies to the unstable case, but (1) it cannot compare the Albanese stabilizers of $P$ and $\check P$, and (2) it does not distinguish types $\mathrm I_r^*$ for odd and even $r$. We will complete the proof of \Cref{thm:classification of component group} (unstable case) in \S \ref{sec:Neron component groups and Albanese stabilizers}, after discussing the semistable reduction theorem in \Cref{prop:semistable reduction theorem untangled ver}.

\section{Base change behaviors} \label{sec:general base change}
	Let $f : X \to S$ be an arbitrary minimal $\delta$-regular abelian fibration, possibly multiple. Consider a finite Galois morphism $p : T \to S$ of degree $d$ that is totally ramified at $t \in T$ over $s \in S$. The morphism $T_0 = T \setminus \{ t \} \to S_0$ is finite \'etale, and the base change $Y_0 = X_0 \times_{S_0} T_0 \to T_0$ is an abelian torsor. Taking its relative minimal model $g : Y \to T$, we obtain a commutative diagram
	\[\begin{tikzcd}
		Y \arrow[r, dashed] \arrow[d, "g"] & X \arrow[d, "f"] \\
		T \arrow[r, "p"] & S
	\end{tikzcd}.\]
	We call such $Y$ a \emph{minimal base change} of $X$: it is the minimal model of the base change $f_T : X_T \to T$. The goal of this section is to study the relation between $X$ and $Y$. The t-automorphism schemes will play a crucial role as usual. We denote by $P \to S$ and $Q \to T$ the t-automorphism schemes of $f$ and $g$, respectively.
	
	As a result, we prove the remaining claims in \Cref{thm:classification of component group} and conclude the classification of Albanese stabilizers in \Cref{thm:classification of unstable fibers}.

	\subsection{The inertia actions} \label{sec:inertia action}
		Since $Y_0 \to X_0$ and $Q_0 \to P_0$ are Galois coverings, both $Y_0$ and $Q_0$ admit free $\mu_d$-actions. Such actions extend to $Y$ and $Q$ by the minimal model and N\'eron mapping property, respectively. We will first study these $\mu_d$-actions.

		\subsubsection{Inertia action on the t-automorphism group}
			This subsection follows the results of Edixhoven \cite{edi92}.
			
			\begin{proposition} \label{prop:psi-action}
				\begin{enumerate}
					\item There exists a natural automorphism $\psi$ of order $d$ on $Q$ making $Q \to T$ equivariant.
					\item The induced $\langle \psi \rangle$-action on $Q_t$ is by group automorphisms:
					\[ \psi(a \cdot b) = \psi(a) \cdot \psi(b) \qquad\mbox{for}\quad a, b \in Q_t .\]
				\end{enumerate}
			\end{proposition}
			\begin{proof}
				(1) Write $\zeta : T \to T$ for the Galois $\mu_d$-action on $T$. Then the equivariant $\mu_d = \langle \psi_0 \rangle$-action on $Q_0$ is modeled by an isomorphism $\Psi_0 : Q_0 \to \zeta^* Q_0 \coloneq Q_0 \times_{\zeta} T_0$ over $T_0$ with a cocycle condition
				\[ (\zeta^*)^{d-1} \Psi_0 \,\circ\, \cdots \,\circ\, \zeta^* \Psi_0 \,\circ\, \Psi_0 = \id_{Q_0} .\]
				Both $Q$ and $\zeta^* Q$ are N\'eron models over $T$, so $\Psi_0$ uniquely extends to an isomorphism $\Psi : Q \to \zeta^* Q$ over $T$ with the same cocycle condition. That is, the $\langle \psi \rangle$-action uniquely extends to $Q$ equivariantly over $T$.
				
				\medskip
				
				(2) The group law $Q_0 \times_{T_0} Q_0 \to Q_0$ is $\langle \psi_0 \rangle$-equivariant over $T_0$, or the diagram
				\[\begin{tikzcd}
					Q \times_T Q \arrow[r] \arrow[d, "\psi \times \psi"] & Q \arrow[d, "\psi"] \arrow[r] & T \arrow[d, "\zeta"] \\
					Q \times_T Q \arrow[r] & Q \arrow[r] & T
				\end{tikzcd}\]
				commutes over $T_0$. The diagram automatically commutes over $T$ because $Q$ is separated (or $Q$ is N\'eron). Restricting it over $t \in T$ yields the claim.
			\end{proof}
			
			To minimize confusions, we write the $\mu_d$-action on $Q$ by the $\langle \psi \rangle$-action. We note that the induced $\langle \psi \rangle$-action on $Q_t$ may \emph{not} be faithful, i.e., the order of the automorphism $\psi : Q_t \to Q_t$ can be strictly smaller than $d$.
			
			\begin{definition}
				The $\langle \psi \rangle$-action on the central fiber $Q_t$ is called the \emph{inertia action}.
			\end{definition}
			
			The inertia action on $Q_t$ is related to the original $P_s$ as follows. Use the N\'eron mapping property to extend the isomorphism $P_{T_0} = Q_0$ over $T_0$ to a homomorphism over $T$:
			\[ p^* : P_T = P \times_S T \to Q \]
			Its restriction over $t \in T$ yields a \emph{pullback} homomorphism of central fibers
			\begin{equation} \label{eq:pullback of t-automorphism}
				p^* : P_s \to Q_t .
			\end{equation}
			
			\begin{theorem} [{\cite[Theorem~5.3]{edi92}}] \label{thm:Edixhoven}
				Let $p^* : P_s \to Q_t$ be the pullback homomorphism \eqref{eq:pullback of t-automorphism}. Then
				\begin{enumerate}
					\item $\im p^*$ is the inertia-invariant subgroup $Q_t^\psi \subset Q_t$.
					\item $\ker p^*$ is unipotent.
				\end{enumerate}
			\end{theorem}
			
			In particular, the abelian variety part of $Q_t$ is always \emph{larger} than that of $P_s$. This significantly restricts the behavior of $p^*$ for $\delta$-regular N\'eron models.
			
			\begin{corollary} \label{cor:Edixhoven}
				Suppose that $P \to S$ is $\delta$-regular.
				\begin{enumerate}
					\item If $P_s$ is an abelian variety, then $p^*$ is an isomorphism.
					\item If $P_s$ has a linear part $\GG_m$, then so does $Q_t$ and $p^*$ is an open immersion. Moreover, there exists a short exact sequence
					\[ 0 \to \pi_0 (P_s) \xlongrightarrow{p^*} \pi_0(Q_t) \to \ZZ/d \to 0 .\]
					\item If $P_s$ has a linear part $\GG_a$, then $Q_t$ has a torsion shear.
				\end{enumerate}
			\end{corollary}
			\begin{proof}
				Everything immediately follows from \Cref{thm:Edixhoven} except for the short exact sequence claim in (2).
				
				\medskip
				
				Assume $P_s$ has a linear part $\GG_m$ and $\pi_0(P_s) = \ZZ/r$ (semistable part of \Cref{thm:classification of component group}). The minimal compactification $f : X \to S$ of $P$ has a strictly semistable central fiber $X_s = \sum_{i=1}^r E_i$ of type $\mathrm I_r$. Write $Z = \Sing(X_s) = \bigsqcup_{i=1}^r A_s$. The usual base change
				\[ \bar g : \bar Y = X_T \to T \]
				is singular along the preimage of $Z$, which is locally described by an equation $X_T = (t^d = xy) \subset \Spec k[[x,y, t, z_1, \cdots, z_{n-1}]]$. It has canonical singularities (type $\mathrm A_{d-1}$ Du Val) and admits a \emph{crepant} resolution by successive blowups along $A_s$. The resulting exceptional divisor is a chain of $d-1$ number of $\PP^1$-bundles over $A_s$. Therefore, the minimal model $g : Y \to T$ has a central fiber consisting of a cycle of $r + (d-1)r = dr$ number of $\PP^1$-bundles, proving $\pi_0 (Q_t) = \ZZ/dr$ again by \Cref{thm:classification of component group}.
			\end{proof}

		\subsubsection{Inertia action on the minimal model}
			The same argument produces the \emph{inertia action} on $Y_t$. Again to minimize potential confusions, we write such a $\mu_d$-action on $Y$ the $\langle \varphi \rangle$-action.
			
			\begin{proposition} \label{prop:varphi-action}
				\begin{enumerate}
					\item There exists a natural automorphism $\varphi$ of order $d$ on $Y$ making $g : Y \to T$ equivariant.
					\item The $Q_t$-and $\langle \varphi \rangle$-actions on $Y_t$ are related by
					\[ \varphi(a \cdot y) = \psi(a) \cdot \varphi(y) \qquad\mbox{for}\quad a \in Q_t , \ y \in Y_t .\]
				\end{enumerate}
			\end{proposition}
			\begin{proof}
				(1) Same argument as \Cref{prop:psi-action}, but use \Cref{prop:minimal delta-regular abelian fibration is unique} to extend an automorphism $\varphi_0 : Y_0 \to Y_0$ to $\varphi : Y \to Y$. (2) Same as the above lemma.
			\end{proof}
			
			\begin{remark}
				Note that the second item means that the $Q_t$- and $\langle \varphi \rangle$-actions on $Y_t$ \emph{do not} commute in general. However, the $Q_t^\psi$- and $\langle \varphi \rangle$-actions commute. Note also that the inertia $\langle \varphi \rangle$-action on $Y_t$ may not be faithful in general.
			\end{remark}
			
			We may therefore consider the quotient $\bar X = Y / \varphi$ and a commutative diagram
			\begin{equation} \label{diag:base change}
			\begin{tikzcd}
				Y \arrow[r, "q"] \arrow[d, "g"] & \bar X \arrow[r, dashed, "\mathrm{bir}"] \arrow[d, "\bar f"] & X \arrow[ld, bend left=25, "f"] \\
				T \arrow[r, "p"] & S
			\end{tikzcd}.
			\end{equation}
			Since $\bar f : \bar X \to S$ is klt-minimal, \Cref{thm:Po-action} applies and endows it with a $P^\circ$-action. The following extends this to a $P$-action.
			
			\begin{proposition} \label{prop:P-action on bar X}
				There exists a unique $P$-action on $\bar X$ extending the $P_0$-torsor structure on $\bar X_0 = X_0$. Consequently, the birational map $\bar X \dashrightarrow X$ is $P$-equivariant.
			\end{proposition}
			\begin{proof}
				The base change $P_T \to T$ is a smooth group scheme with an equivariant $\mu_d$-action generated by $\zeta : P_T \to P_T$ (pulled back from $\zeta : T \to T$). The pullback homomorphism $p^* : P_T \to Q$ of t-automorphism schemes endows $Y$ with a $P_T$-action with the following commutative diagram:
				\[\begin{tikzcd}[column sep=normal]
					P \times_S Y \arrow[r, phantom, "="] \arrow[d, "\id \times \varphi"] & P_T \times_T Y \arrow[r, "p^* \times \id"] \arrow[d, "\zeta \times \varphi"] & Q \times_T Y \arrow[r] \arrow[d, "\psi \times \varphi"] & Y \arrow[d, "\varphi"] \\
					P \times_S Y \arrow[r, phantom, "="] & P_T \times_T Y \arrow[r, "p^* \times \id"] & Q \times_T Y \arrow[r] & Y
				\end{tikzcd}.\]
				The rightmost diagram is nothing but \Cref{prop:varphi-action}, and the first two diagrams are commutative because they commute over $T_0$. Hence $P \times_S Y \to Y$ is $\langle \varphi \rangle$-equivariant, and its quotient is the desired group action morphism $P \times_S \bar X \to \bar X$.
			\end{proof}
			
			For simplicity, let us assume that the $\langle \varphi \rangle$-action on $\bar X_s$ is faithful in the following two lemmas (if not, factorize $T \to S$ into two steps). Define a closed subscheme $\frac{1}{d}\bar X_s \subset \bar X$ by
			\[ \tfrac{1}{d} \bar X_s \coloneq Y_t / \varphi \ \ \subset \ \ \bar X = Y / \varphi .\]
			It is of pure codimension $1$, has no embedded points, and is a closed subscheme of $\bar X_s$.
			
			\begin{lemma} \label{lem:1/d X_s}
				Assume that the $\langle \varphi \rangle$-action on $\bar X_s$ is faithful. Then there exists an equality of Weil divisors $\bar X_s = d \cdot \frac{1}{d} \bar X_s$ on $\bar X$. In particular, the multiplicity of $\bar X_s$ is divisible by $d$.
			\end{lemma}
			\begin{proof}
				Notice an isomorphism of schemes
				\[ \bar X_s = \big( q^{-1}(\bar X_s) \big) / \varphi = Y_d / \varphi \qquad\mbox{for}\quad Y_d = g^{-1} (\Spec k[\epsilon]/\epsilon^d) .\]
				Since the closed immersion $Y_t \subset Y_d$ is set-theoretically equal and their multiplicities differ by a factor of $d$, the same holds for $\frac{1}{d} \bar X_s \subset \bar X_s$.
			\end{proof}
			
			\begin{lemma} \label{lem:P-action on 1/d X_s}
				The $P_s$-action on $\bar X_s$ restricts to that on $\frac{1}{d} \bar X_s$.
			\end{lemma}
			\begin{proof}
				Since the $Q_t^\psi$- and $\langle \varphi \rangle$-actions on $Y_t$ commute by \Cref{prop:varphi-action}, the $Q_t^\psi$-action on $Y_t$ descends to $\frac{1}{d} X_s = Y_t / \varphi$. This means we can endow $\frac{1}{d} X_s$ with a $P_s$-action via $p^* : P_s \to Q_t^\psi$ (so it may not be a faithful action). To show this action is compatible with the $P_s$-action on $X_s$, take a $\mu_d$-equivariant diagram
				\[\begin{tikzcd}[column sep=normal]
					P_T \times_T Y \arrow[r, "p^* \times \id"] \arrow[d] & Q \times_T Y \arrow[r] & Y \arrow[d] \\
					P_T \times_T \bar X_T \arrow[rr] & & \bar X_T
				\end{tikzcd}.\]
				The lower vertical arrow is \Cref{prop:P-action on bar X} and $Y \to \bar X_T$ is from the universal property of fiber product. Restricting the diagram over $t \in T$, replacing $Q_t$ by $Q_t^\psi$ using \Cref{thm:Edixhoven}, and then taking the $\langle \varphi \rangle$-quotient concludes the two actions agree.
			\end{proof}

	\subsection{Semistable base change of unstable fibrations}
		In this subsection, we discuss an alternative interpretation of \Cref{thm:classification of unstable fibers} using a base change to a semistable fibration. This will complete the promised proof of \Cref{thm:classification of component group} and will also be extensively used in the future \Cref{thm:classification of multiple fiber iv} and \ref{thm:classification of multiple fiber v}.
		
		According to the semistable reduction theorem (e.g., \cite[Theorem~7.4.1]{neron}), every unstable abelian fibration $f : X \to S$ admits a semistable base change $g : Y \to T$, i.e., a finite (totally ramified) base change $T \to S$ such that $f_T$ has a semistable minimal model $g$. We make this explicit for minimal $\delta$-regular abelian fibrations as follows. Let us first consider the case with no Albanese stabilizers.
		
		\begin{proposition} \label{prop:semistable reduction theorem untangled ver}
			Let $f : X \to S$ be an unstable minimal $\delta$-regular abelian fibration with trivial Albanese stabilizer, and let $g : Y \to T$ be its semistable base change of degree $d$.
			\begin{enumerate}
				\item If $f$ has type $\mathrm I_0^*$, $\mathrm{II}$, $\cdots$, or $\mathrm{IV}^*$, then $g$ is smooth with a central fiber
				\[ Y_t \cong Q_t \cong E \times A_s \quad \mbox{for an elliptic curve } E ,\]
				and its inertia action is $\varphi(x,y) = \psi(x,y) = (\omega^{\pm1} x, y)$ for $(x,y) \in E \times A_s$. Here $\omega$ is the primitive $d'$-th root of unity for $d' \mid d$, which is classified as follows:
				\begin{enumerate}
					\item If $f$ has type $\mathrm{II}$ or $\mathrm{II}^*$, then $j(E) = 0$ and $d' = 6$.
					\item If $f$ has type $\mathrm{III}$ or $\mathrm{III}^*$, then $j(E) = 1728$ and $d' = 4$.
					\item If $f$ has type $\mathrm{IV}$ or $\mathrm{IV}^*$, then $j(E) = 0$ and $d' = 3$.
					\item If $f$ has type $\mathrm I_0^*$, then $d' = 2$.
				\end{enumerate}
				
				\item If $f$ has type $\mathrm I_r^*$ for $r \ge 1$, then $d = 2k$ is even and $g$ is strictly semistable of type $\mathrm I_{2kr}$ with a central fiber
				\[ Y_t \cong C \times A_s \ \ \supset \ \ Q_t \cong \ZZ/2kr \times \GG_m \times A_s .\]
				Considering $C$ as $C^\nu = \ZZ/2kr \times \PP^1$ with identifications $(i,\infty) = (i+1, 0)$, the inertia involutions are $\varphi(i,z,y) = \psi(i,z,y) = (-i, \frac{\zeta^i}{z}, y)$ for $(i,z,y) \in C \times A_s$ or $\ZZ/2kr \times \GG_m \times A_s$. Here $\zeta$ is a primitive $2k$-th root of unity.
			\end{enumerate}
		\end{proposition}
		
		The following deals with the case when $f$ has a nontrivial Albanese stabilizer.
		
		\begin{proposition} \label{prop:semistable reduction theorem tangled ver}
			Let $f : X \to S$ be an unstable minimal $\delta$-regular abelian fibration with a nontrivial Albanese stabilizer $G$, and $\tilde X \to X$ its untangle in \Cref{prop:untangle for unstable}. Then every semistable base change $g : Y \to T$ of $f$ admits a cartesian diagram
			\[\begin{tikzcd}
				\tilde Y \arrow[r, dashed] \arrow[d] & \tilde X \arrow[d] \\
				Y \arrow[r, dashed] & X
			\end{tikzcd},\]
			such that $\tilde g : \tilde Y \to T$ has a trivialized central fiber as in \Cref{prop:semistable reduction theorem untangled ver}. Letting $Q$ and $\tilde Q$ be the t-automorphism schemes $g$ and $\tilde g$, the following holds.
			\begin{enumerate}
				\item There exists a short exact sequence $0 \to G \to \tilde Q \to Q \to 0$. In particular, $Q_t = \tilde Q_t / G$ and $Q^\circ_t = \tilde Q^\circ_t / (G \cap \tilde Q^\circ_t)$.
				
				\item Assume that $\tilde f$ has type $\mathrm I_0^*$, $\cdots$, or $\mathrm{IV}^*$. Then $G \hookrightarrow \tilde Q_t = E \times A_s$ decomposes into two injections $\tau : G \hookrightarrow E^\omega \subset E$ and $\sigma : G \hookrightarrow A_s$.
				
				\item Assume that $\tilde f$ has type $\mathrm I_r^*$ for $r \ge 1$. Then recalling the unique splitting $\tilde Q_t^\circ = \GG_m \times A_s$, we have
				\[ G \cap \tilde Q_t^\circ = \begin{cases}
					\langle (-1,a) \rangle \qquad & \mbox{for odd } r \\
					\langle (-1,a) \rangle \quad\mbox{or}\quad 0 \qquad & \mbox{for even } r \\
				\end{cases}.\]
				Here $a \in A_s$ is a nonzero $2$-torsion point. Moreover, the Albanese morphism $\alb : Y_t \to \Alb_t$ is always a free $G$-quotient of $\pr_2 : \tilde Y_t = C \times A_s \to A_s$.
			\end{enumerate}
		\end{proposition}
		
		\begin{definition} \label{def:isogeny and conjugate of Kodaira type}
			The \emph{conjugate} of a Kodaira type $\mathrm T = \mathrm{II}$, $\cdots$, $\mathrm{IV}^*/3$ is $\mathrm T^*$. Two Kodaira types are \emph{isogenous} if they are the same up to Albanese stabilizers.
		\end{definition}
		
		For example, the conjugate of $\mathrm{III}^*/2$ is $\mathrm{III}/2$, and $\mathrm I_r^*/2$ and $\mathrm I_r^*/4$ are isogenous Kodaira types.
		
		\begin{corollary} \label{cor:semistable reduction theorem}
			Let $f : X \to S$ be an unstable minimal $\delta$-regular abelian fibration and $g : Y \to T$ its arbitrary semistable base change of degree $d$. Then
			\begin{enumerate}
				\item Up to isogeny and conjugate, the Kodaira type of $f$ is determined by the type of $g$ and the order of the inertia automorphism $\psi$ on $Q_t$ as in \Cref{table:semistable reduction theorem}.
				
				\item Assume $g$ is smooth and $d = d'$ is minimal in \Cref{prop:semistable reduction theorem untangled ver}(1). Assume that the generator of $\mu_d$ acts on $T$ by $\omega$. Then $f$ has a $*$-type if and only if $\psi$ acts on the elliptic curve $E \subset Q_t$ by $\omega^{-1}$.
				
				\item The Albanese stabilizer $G$ of $f$ is captured by the non-triviality of the central fiber $Y_t$ of $g$.
			\end{enumerate}
		\end{corollary}
		
		\begin{table}[h]
			\begin{tabular}{|c||c|c|c|c|c|} \hline
				Type of $\tilde X$ & $\mathrm{II}$, $\mathrm{II}^*$ & $\mathrm{III}$, $\mathrm{III}^*$ & $\mathrm{IV}$, $\mathrm{IV}^*$ & $\mathrm I_0^*$ & $\mathrm I_r^* \ (r \ge 1)$ \\ \hline
				Type of $\tilde Y$ & $\mathrm I_0$ & $\mathrm I_0$ & $\mathrm I_0$ & $\mathrm I_0$ & $\mathrm I_{2kr}$ \\ \hline
				$\ord \psi$ & $6$ & $4$ & $3$ & $2$ & $d = 2k$ \\ \hline
			\end{tabular}
			
			\caption{Unstable fibrations $X$ (up to Albanese stabilizer) are determined by $\ord \psi$ for any semistable base change $Y$.}
			\label{table:semistable reduction theorem}
		\end{table}
		
		The rest of this subsection is devoted to the proof of these results. We will keep using the notation and results in the previous subsection, and assume that $f$ has a trivial Albanese stabilizer until the last moment when we prove \Cref{prop:semistable reduction theorem tangled ver} and \Cref{cor:semistable reduction theorem}.
		
		\begin{lemma} \label{lem:inertia action is faithful}
			Consider the inertia $\langle \varphi \rangle$-action on $Y_t$. Then
			\begin{enumerate}
				\item If the inertia $\mu_d = \langle \varphi \rangle$-action on $Y_t$ is not faithful, then $g$ is smooth and there exists a unique factorization
				\[\begin{tikzcd}[column sep=normal]
					Y \arrow[r] \arrow[d, "g"] & Y^- \arrow[r] \arrow[d, "g^-"] & X \arrow[d, "f"] \\
					T \arrow[r, "\mu_{d/d'}"] & T^- \arrow[r, "\mu_{d'}"] & S
				\end{tikzcd},\]
				where $d' \mid d$, $g^-$ is smooth, and the inertia $\mu_{d'}$-action on $Y^-_t$ is faithful.
				
				\item $\varphi^i : Y_t \to Y_t$ has a fixed point for every $i \in \ZZ$.
			\end{enumerate}
		\end{lemma}
		\begin{proof}
			(1) Assume $\varphi^k$ acts as the identity on $Y_t$ for $d = kd'$. Then the fixed locus of $\varphi^k : Y \to Y$ is set-theoretically $Y_t$ and forces $Y_{t, \red}$ to be regular (the fixed locus of a finite order automorphism of a regular scheme is again regular). Since $Y$ is not multiple, this means $Y_t$ is regular whence $Y \to T$ is smooth. The intermediate quotient $g^- : Y^- = Y / \varphi^k \to T^- = T/\mu_k$ is again an abelian torsor.
			
			(2) Suppose we have $d = ij$ for integers $1 < i, j < d$ and consider a $\mu_j = \langle \varphi^i \rangle$-action on $Z$. Since $f : X \to S$ has a section (after shrinking $S$), so does $Y / \varphi^i \to T/\mu_j$. But this has a central fiber of multiplicity $j$ by \Cref{lem:1/d X_s}, so the total space $Y / \varphi^i$ cannot be regular. Hence the $\langle \varphi^i \rangle$-action on $Y_t$ cannot be free.
		\end{proof}
		
		Thanks to \Cref{lem:inertia action is faithful}(1), up to replacing $Y$ with $Y^-$ we may assume $d = d'$ and the inertia $\langle \varphi \rangle$-action is faithful. The pullback homomorphism in \eqref{eq:pullback of t-automorphism} is
		\begin{equation} \label{eq:prop:semistable reduction theorem for unstable}
			p^* : P_s \cong \pi_0(P_s) \times \GG_a \times A_s \to Q_t ,
		\end{equation}
		and it necessarily has $\ker p^* = \GG_a$ and $\im p^* = Q_t^\psi = \pi_0(P_s) \times A_s \subset Q_t$ by \Cref{thm:Edixhoven}.
		
		\begin{lemma} \label{lem:semistable reduction theorem:splitting}
			\begin{enumerate}
				\item $\bar X_s \cong d \cdot \bar C \times A_s$ and $\frac{1}{d} \bar X_s \cong \bar C \times A_s$ for a proper curve $\bar C$.
				\item Taking $A_s \subset P_s$, the $A_s$-action on $\frac{1}{d} \bar X_s = \bar C \times A_s$ is only on the second factor by translations.
			\end{enumerate}
		\end{lemma}
		\begin{proof}
			By \Cref{prop:P-action on bar X}, the birational map $\bar X \dashrightarrow X$ is $P$-equivariant. Hence its undefined loci are $P_s$-invariant closed subsets of $\bar X$ and $X$ of dimension $\le n-1$. By $P^\circ_s = \GG_a \times A_s$ and \Cref{lem:P-orbit has dimension n-1} they are $A_s$-orbits with finite stabilizers. This means $\bar X$ and $X$ can be connected by a sequence of blowups and blowdowns along $A_s$-orbits. Since $X$ is untangled, the $A_s$-action on $X_s = C \times A_s$ is free and acting only on the second factor by translations. Therefore, all intermediate birational models connecting $\bar X$ and $X$ have central fibers of the form $C_i \times A_s$ with translation $A_s$-action on the second factor. In particular, same holds for $\bar X_s$ and hence $\frac{1}{d} \bar X_s$.
		\end{proof}

		\subsubsection{When $g$ is smooth}
			The abelian fibration $g$ is either smooth or strictly semistable. Let us first consider the case when $g$ is smooth. Then $Q_t$ is an abelian variety and $Y_t$ is a $Q_t$-torsor. Identify them using a $\varphi$-fixed point $0 \in Y_t$ in \Cref{lem:inertia action is faithful}(2) as a reference point. Then the two inertia actions $\psi$ on $Q_t$ and $\varphi$ on $Y_t$ can be identified as well, thanks to \Cref{prop:varphi-action}
			\[ \varphi(a \cdot 0) = \psi(a) \cdot \varphi(0) = \psi(a) \cdot 0 \qquad\mbox{for}\quad a \in Q_t .\]
			As a result, we do not need to distinguish $Q_t$ and $Y_t$ (and their inertia actions).
			
			\begin{lemma} \label{lem:semistable reduction theorem:Qt splits}
				$Q_t = Y_t \cong E \times A_s$ for an elliptic curve $E$. 
			\end{lemma}
			\begin{proof}
				We had $\frac{1}{d} \bar X_s \cong \bar C \times A_s$ from \Cref{lem:semistable reduction theorem:splitting}. The composition $Y_t \to Y_t / \varphi = \frac{1}{d} \bar X_s \to A_s$ is $A_s$-equivariant whence locally trivial. Any $A_s$-orbit in $Y_t$ constructs its section, so $Y_t \cong E \times A_s$ for an elliptic curve $E$.
			\end{proof}
			
			\begin{lemma} \label{lem:semistable reduction theorem:varphi}
				Choosing an appropriate isomorphism $Q_t = Y_t \cong E \times A_s$, the inertia action is of the form
				\[ \psi(x,y) = \varphi (x,y) = (\omega^{\pm1} x, y) \qquad\mbox{for}\quad (x,y) \in E \times A_s .\]
			\end{lemma}
			\begin{proof}
				Start with an isomorphism $Q_t \cong E \times A_s$ in \Cref{lem:semistable reduction theorem:Qt splits}. Recall from \eqref{eq:prop:semistable reduction theorem for unstable} that $\{ 0 \} \times A_s \subset Q_t$ is fixed by $\psi$. Hence the group automorphism $\psi$ of an abelian variety $E \times A_s$ is necessarily of the form
				\[ \psi(x,y) = (\alpha x, y + \xi(x)) \qquad\mbox{for}\quad \alpha \in \Aut_0(E) ,\quad \xi : E \to A_s \mbox{ homomorphism} .\]
				If $x \in E^\alpha$, then $\varphi = \psi$ acts on $\{ x \} \times A_s$ as a translation by $\xi(x) \in A_s$. If it is a nontrivial translation, then the quotient $\frac{1}{d} \bar X_s = (E \times A_s)/\varphi$ cannot be of the form $\bar C \times A_s$ in \Cref{lem:semistable reduction theorem:splitting}. This proves $\xi(x) = 0$ for every $x \in E^\alpha$. The isogeny $1 - \alpha : E \to E$ has a kernel $E^\alpha$, so this factorizes $\xi$ into
				\[\begin{tikzcd}[column sep=tiny]
					E \arrow[rr, "\xi"] \arrow[rd, "1-\alpha"'] & & A_s \\
					& E \arrow[ru, "\eta"']
				\end{tikzcd}, \qquad\mbox{where } \ \eta : E \to A_s \mbox{ is a homomorphism} .\]
				Define an automorphism $u$ on $E \times A_s$ by $u(x,y) = (x, y + \eta(x))$. Direct computation shows $u \circ \psi \circ u^{-1}$ is an automorphism on $E \times A_s$ by $(x,y) \mapsto (\alpha x, y)$.
				
				To show $\alpha = \omega^{\pm1}$, consider $\pi_0(P_s) = \pi_0(Q_t^\psi) = \pi_0 (E^\alpha)$. The order of $\pi_0(P_s)$ is already decided at the end of the proof of \Cref{thm:classification of unstable fibers}: it is either $1$, $2$, $3$, or $4$. For each of the possibility, there exists a unique $\alpha \in \Aut_0(E)$ with $|E^\alpha| = |\pi_0(P_s)|$. This shows $\alpha = \omega^{\pm1}$ and $j(E)$ as desired in \Cref{prop:semistable reduction theorem untangled ver}.
			\end{proof}
			
			This completes the proof of \Cref{prop:semistable reduction theorem untangled ver}(1). In fact, we can describe the singularities of $\bar X = Y /\varphi$ in each type and hence the explicit decomposition of the birational map $\bar X \dashrightarrow X$ into a sequence of blowups and blowdowns, copying the process in \cite[\S V.10, p.209]{bar-hulek-pet-van:surface}. We also note the following equivalences, generalizing \cite[Proposition~4.24--25]{kim-laza-mar26}: $\psi$ acts on $E$ by $\zeta^{-1}$ $\Longleftrightarrow$ $\bar Y$ has canonical singularities $\Longleftrightarrow$ $Y \to \bar Y$ is a birational morphism $\Longleftrightarrow$ $f$ has a $*$-type $\mathrm I_0^*$, $\mathrm{II}^*$, $\mathrm{III}^*$, or $\mathrm{IV}^*$. If $d \neq 2$ and $\psi$ acts on $E$ by $\zeta$, then $f$ has a non-$*$-type $\mathrm{II}$, $\mathrm{III}$, or $\mathrm{IV}$.

		\subsubsection{When $g$ is strictly semistable}
			Suppose now $g$ is strictly semistable of type $\mathrm I_R$ for $R \ge 1$. In the previous subsection, we have shown that $g$ is smooth if and only if $X_s$ has type $\mathrm I_0^*$, $\mathrm{II}$, $\cdots$, or $\mathrm{IV}^*$. Hence $X_s$ in this case necessarily has type $\mathrm I^*_r$ for $r \ge 1$. Again the pullback homomorphism \eqref{eq:prop:semistable reduction theorem for unstable} must satisfy $\ker p^* = \GG_a$ and $\im p^* = Q_t^\psi = \pi_0(P_s) \times A_s \subset Q_t$ by \Cref{thm:Edixhoven}. This means $Q_t$ contains an abelian subvariety of dimension $n-1$, so $Q_t$ has a torsion shear by \Cref{cor:dimension of Albanese}. But more is true:
			
			\begin{lemma} \label{lem:semistable reduction theorem:Qt splits 2}
				$Q_t^\circ = \GG_m \times A_s$.
			\end{lemma}
			\begin{proof}
				As shown, $p^*$ induces an inclusion $A_s \hookrightarrow Q_t^\circ$. Take the Chevalley homomorphism $Q_t^\circ \to B_t$ to an abelian variety $B_t$ of dimension $n-1$. We claim that the composition $A_s \hookrightarrow Q_t^\circ \twoheadrightarrow B_t$ is an isomorphism. This will prove the desired splitting, which is unique because $\Hom(A_s, \GG_m) = 0$.
				
				According to the classification in \Cref{thm:classification of semistable fibers}, $\Sing (Y_t)$ is a disjoint union of isomorphic copies of $B_t$. Recall that $A_s \subset \im p^* = Q_t^\psi \subset Q_t$ acts on $\Sing(Y_t)$, and commutes with $\varphi$ by \Cref{prop:varphi-action}. Since the $A_s$-action on $\Sing(Y_t)$ is by translations, if $\varphi$ fixes a connected component $B_t \subset \Sing(Y_t)$ then it acts on it by translations (because it commutes with all translations). Hence $\Sing(Y_t)/\varphi$ is a disjoint union of quotients of $B_t$. But such a quotient has a free $A_s$-action since $\frac{1}{d} \bar X_s$ has a free $A_s$-action by \Cref{lem:semistable reduction theorem:splitting}. This forces the translation $A_s$-action on $B_t$ to be free, or $A_s = B_t$ as claimed.
			\end{proof}
			
			\begin{corollary} \label{cor:semistable reduction theorem:psi-action on neutral}
				\begin{enumerate}
					\item $R$ is even.
					\item The inertia $\langle \psi \rangle$-action on $Q_t \cong \ZZ/R \times \GG_m \times A_s$ satisfies
					\[ \psi(0,z,a) = (0,\tfrac{1}{z},a) ,\quad \psi(\tfrac{R}{2},z,a) = (\tfrac{R}{2},\pm \tfrac{1}{z},a) \qquad\mbox{for}\quad (z,a) \in \GG_m \times A_s .\]
				\end{enumerate}
			\end{corollary}
			\begin{proof}
				Fix an isomorphism $Q_t \cong \ZZ/R \times \GG_m \times A_s$. Recall that $Q_t^\psi \cong \pi_0(P_s) \times A_s$ has a neutral component $A_s$. Hence $\psi$ must act on $Q_t^\circ$ by $\psi(0,z,a) = (0,z^{-1},a)$. From it we deduce $(Q_t^\circ)^\psi \cong \mu_2 \times A_s$ and $|\pi_0(Q_t^\circ)^\psi| = 2$. Since $X_s$ has type $\mathrm I_r^*$, we have $|\pi_0(P_s)| = |\pi_0(Q_t^\psi)| = 4$ by \Cref{thm:classification of unstable fibers} (again, we have not computed $\pi_0(P_s)$ yet but know its order is $4$). Hence $Q_t^\psi$ is contained in precisely one more component $\frac{R}{2} \in \pi_0 (Q_t) = \ZZ/R$ (see \cite[Remark~5.4.1]{edi92}) and $R$ is even.
				
				\medskip
				
				Write $\psi_i : \GG_m \times A_s \to \GG_m \times A_s$ for the restriction of $\psi$ to the $R/2$-th component of $Q_t$. Again $Q_t^\psi = \pi_0(P_s) \times A_s$, so $(\GG_m \times A_s)^{\psi_i}$ consists of two copies of $A_s$. This means $\psi_i$ acts trivially on the $A_s$-factor and fixes two points on $\GG_m$. Finally, since $\psi$ respects group operations, we have an identity $\psi(R/2,z,a)^2 = \psi(0, z^2, 2a) = (0, z^{-2}, 2a)$. This makes $\psi_i(z)^2 = z^{-2}$ for all $z \in \GG_m$, or $\psi_i(z) = \pm z^{-1}$.
			\end{proof}
			
			\begin{lemma} \label{lem:semistable reduction theorem:varphi 2}
				There exists an isomorphism $Y_t \cong C \times A_s$ for a Kodaira type $\mathrm I_R$ semistable curve $C$. This induces an identification of $Q_t$ with $Y_t' = Y_t \setminus \Sing(Y_t)$ that respects the inertia $\langle \psi \rangle$- and $\langle \varphi \rangle$-actions.
			\end{lemma}
			\begin{proof}
				Since $R$ is even by \Cref{cor:semistable reduction theorem:psi-action on neutral}, let us temporarily write $R = 2u$. Write $Y_t = \sum_{i=0}^{2u-1} E_i$, where every irreducible component $E_i$ is isomorphic to $\PP^1 \times A_s$ as a compactification of $Q_t^\circ = \GG_m \times A_s$. By \Cref{lem:inertia action is faithful}(2), $\varphi : Y_t \to Y_t$ admits at least one fixed point $y$.
				
				\medskip
				
				We first claim that every fixed point is contained in $Y_t' = Y_t \setminus \Sing(Y_t)$. This is similar to the proof of \Cref{lem:stabilizer of Ir}. Assume on the contrary that $\varphi$ fixes a point $y \in E_0 \cap E_1 \subset \Sing(Y_t)$ (in fact, $\varphi$ fixes the entire $E_0 \cap E_1$ by \Cref{prop:varphi-action}). If $\varphi$ preserves $E_1$, then \Cref{cor:semistable reduction theorem:psi-action on neutral} shows that $\varphi$ swaps $0_1$ and $\infty_1$, a contradiction. Hence $\varphi$ swaps $E_0$ and $E_1$, and hence $\varphi(E_i) = E_{2u+1-i}$ for all $i$. The fixed locus of $\varphi$ is precisely $(E_0 \cap E_1) \sqcup (E_u \cap E_{u+1}) = A_s \sqcup A_s$. But then for every point $y \in E_0 \cap E_1$, $\varphi$ acts around $y$ trivially in the direction of $A_s$ and swaps the two irreducible components $E_0$ and $E_1$. This means $\varphi$ acts on the tangent space $T_y Y \cong k^{n+1}$ by fixing an $n$-dimensional subspace. Hence the quotient $Y/\varphi = \bar X$ is regular by the Chevalley--Shephard--Todd theorem. This implies $\bar X \cong X$ by the uniqueness of the minimal abelian fibration, a contradiction.
				
				\medskip
				
				We may thus assume $\varphi(y) = y$ for a point $y \in Y_t'$. Assume without loss of generality $y \in E_0$. Then a similar argument as above trivializes the $Q_t$-torsor $Y_t'$, identifying their inertia actions. Recall from \Cref{cor:semistable reduction theorem:psi-action on neutral} that $\psi$ acts trivially on the second factor of $Q_t^\circ = \GG_m \times A_s$. Since $\varphi$ preserves the component $E_0$, this proves $\varphi$ acts trivially on the $A_s$-direction of $E_0 = \PP^1 \times A_s$. This property propagates to every component $E_i$ inductively along the intersections $E_i \cap E_{i+1}$. Note that $\varphi$ fixes (at least) two components $E_0$ and $E_u$ by \Cref{cor:semistable reduction theorem:psi-action on neutral}. Hence $\varphi$ preserves every rational curve in $E_0$ and $E_u$.
				
				Start with a unique rational curve $C_0 \subset E_0$ containing $y$. Then $\varphi$ swaps the two points $p_1 = C_0 \cap E_1$ and $p_u = C_0 \cap E_{2u-1}$. Take the rational curves $C_1 \subset E_1$ through $p_1$ and $C_{2u-1} \subset E_{2u-1}$ through $p_{2u}$. Set $p_2 = C_1 \cap E_2$ and $p_{2u-1} = C_{2u-1} \cap E_{2u-2}$. Then $\varphi$ swaps $p_2$ and $p_{2u-1}$. This inductive process defines $p_u \in E_{u-1} \cap E_u$ and $p_{u+1} \in E_{u+1} \cap E_u$ that are swapped by $\varphi$. Now take the unique rational curves $C_u, C_u' \subset E_u$ through $p_u$ and $p_{u+1}$, respectively. Then $\varphi$ swaps these two rational curves. But $\varphi$ preserves every rational curve in $E_u$ by \Cref{cor:semistable reduction theorem:psi-action on neutral}(2), so this makes $C_u = C_u'$. That is, $Y_t$ has no shears.
			\end{proof}
			
			\begin{corollary} \label{cor:semistable reduction theorem:psi-action}
				\begin{enumerate}
					\item $d$ is even and $R$ is divisible by $d$. Write $d = 2k$ and $R = 2kr$ for positive integers $k$ and $r$.
					\item There exists a primitive $2k$-th root of unity $\zeta \in \GG_m$ such that the inertia $\langle \psi \rangle$-action on $Q_t \cong \ZZ/2kr \times \GG_m \times A_s$ is
					\[ \psi(i,z,a) = (-i,\tfrac{\zeta^i}{z}, a) \qquad\mbox{for}\quad (i,z,a) \in \ZZ/2kr \times \GG_m \times A_s .\]
				\end{enumerate}
			\end{corollary}
			\begin{proof}
				We may from now on identify the $\langle \varphi \rangle$- and $\langle \psi \rangle$-actions by \Cref{lem:semistable reduction theorem:varphi 2}. We can now say $\psi$ has order $d$ because $\varphi$ had order $d$ by \Cref{lem:inertia action is faithful}(1). This concludes $d$ is even by \Cref{cor:semistable reduction theorem:psi-action on neutral}, so write $d = 2k$.
				
				Note that $\varphi$ either preserves or reverses the \emph{orientation} of $C = \sum_{i=0}^{R-1} C_i$. Since $\psi(0,z,a) = (0,z^{-1},a)$ by \Cref{cor:semistable reduction theorem:psi-action on neutral}, $\varphi$ must reverse the orientation. Shifting the index $i$  appropriately, we can assume $\varphi(E_i) = E_{-i}$. This makes $\psi$ of the form
				\[ \psi(i,z,a) = (-i, \psi_i (z,a)) \qquad\mbox{for}\quad \psi_i : \GG_m \times A_s \to \GG_m \times A_s. \]
				Since $\psi$ respects group operations, we have a sequence of identities
				\[ \psi(i,z,a) = \psi(i,1,0) \cdot \psi(0,z,a) = (-i, \psi_i(1,0)) \cdot (0,z^{-1},a) = (-i, \psi_i(1,0) \cdot (z^{-1},a)) .\]
				Hence $\psi_i(z,a) = \zeta_i \cdot (z^{-1},a)$ for $\zeta_i \in \GG_m \times A_s$. Since $\psi(i,z^i, ia) = \psi(1,z,a)^i = (1,\zeta_1 \cdot (z^{-1},a))^i = (i, \zeta_1^i \cdot (z^{-i}, ia))$, we have $\zeta_i = \zeta^i$ for a single $\zeta \in \GG_m \times A_s$. Since $\varphi$ acted on the $A_s$-direction by translations, $\zeta \in \GG_m$. Since $\psi$ had order $2k$, $\zeta$ is a primitive $2k$-th root of unity. Finally, $\zeta_R = \zeta^R = 1$ shows $R$ is divisible by $2k$.
			\end{proof}
			
			Notice that $\varphi^2(i,z,a) = (i, \zeta^{-2i}z,a)$ acts on $Y_t$ freely. Therefore, replacing $Y$ with the intermediate quotient $g^- : Y^- = Y / \varphi^2 \to T^- = T/\mu_k$, we may assume $k = 2$ and $\varphi$ is an involution. The fixed locus of the involution $\varphi$ consists of $A_s \sqcup A_s \subset C_0 \times A_s$ and $A_s \sqcup A_s \subset C_r \times A_s$. These are reflected by the singular locus $\sqcup^4 A_s \subset \bar X = Y/\varphi$, which are subsequently blown up into four $\PP^1 \times A_s$ to yield a type $\mathrm I_r^*$ of $X_s$. This completes the proof of \Cref{prop:semistable reduction theorem untangled ver}(2).

		\subsubsection{N\'eron component groups and Albanese stabilizers} \label{sec:Neron component groups and Albanese stabilizers}
			We can finally complete the promised proof of the unstable case of \Cref{thm:classification of component group} and the classification of Albanese stabilizers in \Cref{thm:classification of unstable fibers}.
			
			\begin{proof} [Proof of \Cref{thm:classification of component group}(2) and \Cref{thm:classification of unstable fibers}]
				We need to compute $\pi_0(P_s)$ when $P$ has no Albanese stabilizers. In this case, $P$ admits a semistable base change to $Q \to T$ by \Cref{prop:semistable reduction theorem untangled ver} with $Q_t = E \times A_s$ or $\ZZ/2r \times \GG_m \times A_s$. The component group $\pi_0(P_s)$ can be computed by $\pi_0 (Q_t^\psi)$. In the former case, the inertia action is $\psi(x,y) = (\omega^{\pm1}x, y)$ and hence $Q_t^\psi = E^\omega \times A_s$ has the desired component groups. In the latter case, the inertia action is $\psi(i,z,y) = (-i, (-1)^i z^{-1}, y)$, so again $Q_t^\psi$ is isomorphic to $(\ZZ/2)^{\times 2} \times A_s$ when $r$ is even, and $\ZZ/4 \times A_s$ when $r$ is odd.
			\end{proof}

			Let us now assume the Albanese stabilizer is nontrivial. \Cref{cor:semistable reduction theorem} follows from our discussions above and the following proof of \Cref{prop:semistable reduction theorem tangled ver}, so we omit it.
			
			\begin{proof} [Proof of \Cref{prop:semistable reduction theorem tangled ver}]
				Consider the untangle $\tilde f : \tilde X \to S$ of $X$ and take its semistable base change $\tilde g : \tilde Y \to T$ using \Cref{prop:semistable reduction theorem untangled ver}. Let $G \subset \tilde P_s = \pi_0(\tilde P_s) \times \GG_a \times A_s$ be the Albanese stabilizer of $X$. The pullback $p^* : \tilde P_s \to \tilde Q_t$ has a kernel $\GG_a$ by \Cref{thm:Edixhoven}, so it sends $G$ isomorphically to a subgroup $p^*G \subset \tilde Q_t$. Hence we obtain a section $p^*G \subset \tilde Q$ and a $p^*G$-action on $\tilde Y$. We claim this action is free. If so, then the quotient $g : Y = \tilde Y/p^*G \to S$ is a minimal abelian fibration and hence the semistable base change of the original $f$.
				
				If $\tilde g$ is smooth then $\tilde Q_t$ acts on $\tilde Y_t$ freely, so the $p^*G$-action is free. Suppose that $\tilde g$ is strictly semistable, i.e., $\tilde f$ has type $\mathrm I_r^*$ and $\tilde g$ has type $\mathrm I_{2r}$ for $r \ge 1$. The only nontrivial stabilizer of $\tilde Y_t$ is $\GG_m \subset \tilde Q_t$ by \Cref{lem:stabilizer of Ir}. From the description of \Cref{prop:semistable reduction theorem untangled ver}(2), we have
				\[ \im p^* \cap \GG_m = \tilde Q_t^{\tilde \psi} \cap \GG_m = \{ \pm 1 \} .\]
				Therefore, if $p^*G$ acts non-freely on $\tilde Y_t$, then $G \subset \tilde P_s = \pi_0(\tilde P_s) \times \GG_a \times A_s$ contains an element $f = (\alpha, 0, 0)$ swapping the two arms $E_1$ and $E_2$, and $E_3$ and $E_4$ of the type $\mathrm I_r^*$ curve $C = E_1 + E_2 + 2(\sum_{i=0}^r F_i) + E_3 + E_4$. However, such $f$ does not act freely on $\tilde X$ because it does not have any translation on the $A_s$-direction, contradicting the fact that $\tilde X \to X$ is a free $G$-quotient. This shows $p^*G$ acts on $\tilde Y$ freely.
				
				\medskip
				
				(1) Since $\tilde Y \to Y$ is a free $p^*G$-quotient of non-multiple abelian fibrations, we can imitate the proof of \Cref{prop:untangle for unstable}(2). We caution the reader that the second part of the proof does not apply in this situation. or $p^*G$ may have a nontrivial intersection with $\tilde Q_t^\circ$. This is because the proof was based on the observation $Q_t^\circ = \tilde Q_t^\circ$, which does not hold in this case.
				
				\medskip
				
				(2) Recall from \Cref{prop:semistable reduction theorem untangled ver} an isomorphism $\tilde Q_t = E \times A_s$. Write $p^* : G \hookrightarrow E \times A_s$ by $p^*(g) = (\tau(g), \sigma(g))$ for $\tau : G \to E$ and $\sigma : G \to A_s$. By \Cref{prop:untangle for unstable}, we have $G \cap \tilde P_s^\circ = 0$ and hence $p^*G \cap (\{ 0 \} \times A_s) = 0$. This means $\tau : G \to E$ is injective. If $\sigma(g) = 0$ for some nontrivial $g \in G$, then the quotient $\bar Y = \tilde Y/p^*(g)$ would have had a central fiber $\bar Y_t = E/p^*(g) \times A_s$. This would mean that $\bar X = \tilde X/g$ has a trivial Albanese stabilizer, a contradiction. Hence $\tau : E \to A_s$ is injective, too. Finally, because the $p^*G$-action on $E \times A_s$ commutes with $\tilde \varphi(x,y) = (\omega^{\pm1}x,y)$, we have $\tau(g) \in E^\omega$.
				
				\medskip
				
				(3) We have an isomorphism $\tilde Q_t = \ZZ/2r \times \GG_m \times A_s$ and the classification of Albanese stabilizers $G$ in \Cref{thm:classification of unstable fibers}. Similar to above, notice that $\im p^* \cap \tilde Q_t^\circ = \{ \pm 1 \} \times A_s$. This shows $p^*G \cap \tilde Q_t^\circ = 0$ or $\langle (-1, a) \rangle$ for a $2$-torsion point $a \in A_s$. Since we have already proved $p^*G \cap \GG_m = 0$ above, $a \neq 0$.
				
				\medskip
				
				The $p^*G$-action on $\tilde Y_t = C \times A_s$ is diagonal. If $r$ is odd, then $1 \in \ZZ/4 = \pi_0 (\tilde P_s)$ is $(r, \sqrt{-1}, a) \in \tilde Q_t$ for a torsion point $a \in A_s$ of order $4$. This proves the claim. If $r$ is even, then $G \subset (\ZZ/2)^{\times 2} = \pi_0(P_s)$ is generated by $(0, 1, a)$ and $(r, 1, b)$ for $a, b \in A_s$ of order $2$. Assume on the contrary $a = b$. Then $g = (r, 1, 0) \in \pi_0(P_s)$ acts on the $A_s$-factor of $\tilde Y_t$ trivially. Since $\ZZ/2 = \langle g \rangle$ acts on $C$ by preserving the order and without fixed points, $\bar C = C / g$ is a Kodaira singular curve of type $\mathrm I_r$. But $\bar Y = \tilde Y/g \to T$ has a central fiber $\bar Y_t = \bar C \times A_s$ and it is a semistable base change of $\bar X = \tilde X/g \to S$, so \Cref{prop:semistable reduction theorem untangled ver} shows that such $\bar X$ has type $\mathrm I_{\frac{r}{2}}^*$. Contradiction.
			\end{proof}
			
			\begin{lemma} \label{lem:classification of component group:dual inertia action}
				The dual of the $\langle \psi \rangle$-action on $Q_t$ is the $\langle \check \psi^{-1} \rangle$-action on $\check Q_t$.
			\end{lemma}
			\begin{proof}
				Recall from \Cref{prop:psi-action} how the inertia $\langle \psi \rangle$-action on $Q_t$ was defined. Take the Galois action $\zeta : T \to T$ of the base, and consider an isomorphism $\Psi : Q \to \zeta^* Q$ of abelian schemes over $T$ (with a cocycle condition). Its dual is an isomorphism of dual abelian schemes $\check \Psi : \zeta^* \check Q \to \check Q$, whose inverse is an isomorphism $(\check \Psi)^{-1} : \check Q \to \zeta^* \check Q$ that describes the $\langle \check \psi \rangle$-action by definition. Comparing the central fibers of the two descriptions yields the lemma.
			\end{proof}
			
			\begin{proof} [Proof of \Cref{thm:classification of component group}(3): unstable case]
				Assume that $P$ has an Albanese stabilizer $G$. Then \Cref{prop:semistable reduction theorem tangled ver} shows $Q_t = (E \times A_s)/G$ or $(\ZZ/2r \times \GG_m \times A_s)/G$. Suppose first that $Q_t = (E \times A_s)/G$. Then the dual t-automorphism group $\check P$ has a semistable base change $\check Q \to T$, whence the central fiber $\check Q_t = ((E \times A_s)/G)^\vee$. By \Cref{lem:classification of component group:dual inertia action}, its inertia action is $\check \psi^{-1}$. To compute its invariant group, notice that we have a further $G$-quotient $Q_t \to E/G \times A_s/G$. Its dual is a $\hat G$-quotient
				\[ (E/G)^\vee \times (A_s/G)^\vee \to \check Q_t = ((E \times A_s)/G)^\vee .\]
				Since $\check \psi^{-1}$ acts on $(E/G)^\vee \times (A_s/G)^\vee$ by $\check \psi(\check x, \check y) = (\omega^{\pm1}\check x, \check y)$, computation of the $\check \psi$-invariant shows its Albanese stabilizer is $\hat G$, which is isomorphic to $G$.
				
				\medskip
				
				Suppose that $P$ has type $\mathrm I_r^*/G$ for odd $r$ (so $Q_t = (\ZZ/2r \times \GG_m \times A_s)/G$). In this case, $G = 0$, $\ZZ/2$, or $\ZZ/4$. For each $G$, we can compute $\pi_0(Q_t)$ and $Q_t^\circ$ as follows.
				\begin{center}
					\begin{tabular}{|c|c||c|c||c|c|} \hline
						Type of $P$ & $G$ & $\pi_0(Q_t)$ & $Q_t^\circ$ & $\pi_0(\check Q_t)$ & $\check Q_t^\circ$ \\ \hline\hline
						$\mathrm I_r^*$ & $0$ & $\ZZ/2r$ & $\GG_m \times A_s$ & $\ZZ/2r$ & $\GG_m \times \check A_s$ \\ \hline
						$\mathrm I_r^*/2$ & $\ZZ/2$ & $\ZZ/2r$ & $(\GG_m \times A_s)/2$ & $\ZZ/2r$ & $\GG_m \times (A_s/2)^\vee$\\ \hline
						$\mathrm I_r^*/4$ & $\ZZ/4$ & $\ZZ/r$ & $(\GG_m \times A_s)/2$ & $\ZZ/r$ & $(\GG_m \times (A_s/4)^\vee) / 2$ \\ \hline
					\end{tabular}
				\end{center}
				To compute $\check Q_t$, notice that it has the same type as $Q_t$ by the semistable version of the theorem, so $\pi_0(\check Q_t) \cong \pi_0(Q_t)$. The abelian variety parts of $Q_t$ and $\check Q_t$ are dual by \Cref{prop:dual Neron model}, and $(A_s/G)^\vee \subset \check Q_t^\circ$ by \Cref{prop:semistable reduction theorem tangled ver}(3). This computes $\check Q_t^\circ$. For example, when $G = \ZZ/4$, $\check Q_t^\circ$ has an abelian variety part $(A_s/2)^\vee$ and contains $(A_s/4)^\vee$. Do the same computation for even $r$ and collect all possible pairs $(Q_t, \check Q_t)$ for every type of $P$. Direct comparison shows that $(Q_t, \check Q_t) \cong (\check Q'_t, Q'_t)$ implies $G \cong G'$. This concludes the theorem. For example, type $\mathrm I_r^*/2$ and $\mathrm I_{2r}^*/2$ are dual to each other for odd $r$, yet their $\pi_0$-groups are the same.
			\end{proof}
			
			\begin{remark}
				The Albanese stabilizer $G$ of an unstable fibration sits in the short exact sequence $0 \to G \to A_s \to \Alb_s \to 0$. Taking the dual, we obtain $0 \to \hat G \to \Alb_s^\vee = A_s^* \to \check A_s = \Alb_{\check X_s} \to 0$. This suggests that the Albanese stabilizers of $P$ and $\check P$ are (more precisely) Cartier dual to each other.
			\end{remark}

\section{Classification of multiple central fibers} \label{sec:classification of multiple fibers}
	Throughout this section, let $f : X \to S$ be a minimal $\delta$-regular abelian fibration whose central fiber $X_s$ has multiplicity $m > 1$. As usual, $P \to S$ is its t-automorphism scheme. We will see in a moment that there exists a preferred base change $g : Y \to T$ with a t-automorphism scheme $Q \to T$.
	
	Now the classification will be divided into three part. When $P$ is semiabelian, we classify $X_s$ in \Cref{thm:classification of multiple fiber i} and \ref{thm:classification of multiple fiber iii}, and call them type $m \cdot \mathrm I_r$ for $r \ge 0$. When $P$ is not semiabelian (so $P_s$ has a linear part $\GG_a$), we divide the case by two depending on the behavior or $Q$. When $Q$ is still semiabelian, we classify $X_s$ in \Cref{thm:classification of multiple fiber ii} and \ref{thm:classification of multiple fiber iv}, and call them type $m \cdot \mathrm I_r^{\pm}$ for $r \ge 0$. Finally, when both $P$ and $Q$ are not semiabelian, we classify $X_s$ in \Cref{thm:classification of multiple fiber v}. In each case, we will have to patiently analyze all numerically possibilities.

	\subsection{Non-multiple base change}
		Take a finite totally ramified Galois morphism $p : T \to S$ of degree $m$, which is precisely the multiplicity of $X_s$. In this special case, the minimal model $g : Y \to T$ can be explicitly constructed as a normalized base change
		\begin{equation} \label{diag:degree m base change}
		\begin{tikzcd}
			Y \arrow[r, "q"] \arrow[d, "g"] & X \arrow[d, "f"] \\
			T \arrow[r, "p"] & S
		\end{tikzcd} \qquad\mbox{for}\quad Y = (X \times_S T)^\nu .
		\end{equation}
		
		\begin{lemma} \label{lem:non-multiple base change}
			Let $T \to S$ be a totally ramified base change of degree $m$ and $Y$ the normalized base change of $X$. Then
			\begin{enumerate}
				\item $q$ is \'etale.
				\item $g$ is a minimal $\delta$-regular abelian fibration that is non-multiple.
			\end{enumerate}
		\end{lemma}
		\begin{proof}
			This is essentially \cite[Proposition~III.9.1]{bar-hulek-pet-van:surface} and \cite[Proposition~3.1]{hwang-ogu11}. (1) The statement is local on $X$, so do the local computation. (2) The total space $Y$ is regular and $K_Y$ is numerically $g$-trivial, so $g$ is a minimal abelian fibration. It is $\delta$-regular by \Cref{cor:Edixhoven}. The central fiber $Y_t$ has multiplicity $1$ by \Cref{lem:1/d X_s}: the normalization clears up its multiplicity.
		\end{proof}
		
		This simplifies the diagram \eqref{diag:base change} with $\bar X = X$, so the results in \S \ref{sec:inertia action} apply directly to $X$:
		
		\begin{corollary}
			\begin{enumerate}
				\item The inertia $\langle \varphi \rangle$-action on $Y_t$ is free of order $m$.
				\item The quotient $\bar X = Y / \varphi$ in \eqref{diag:base change} coincides with $X$. \qed
			\end{enumerate}
		\end{corollary}
		
		As in \S \ref{sec:inertia action}, define a closed subscheme $\frac{1}{m} X_s \subset X_s \subset X$ either as a quotient $\frac{1}{m} X_s = Y_t / \varphi$ or as a scheme associated to the Cartier divisor (\Cref{lem:1/d X_s}):
		\[ \tfrac{1}{m} X_s = \sum_{i=1}^r \tfrac{a_i}{m} \cdot E_i .\]
		This constructs a sequence of morphisms
		\begin{equation} \label{eq:q and i}
		\begin{tikzcd}
			Y_t \arrow[r, "q"] & \tfrac{1}{m}X_s \arrow[r, hookrightarrow, "i"] & X_s
		\end{tikzcd}
		\end{equation}
		where the former is a free $\langle \varphi \rangle$-quotient and the latter is an infinitesimal thickening. The following is simply \Cref{thm:Edixhoven}--\ref{cor:Edixhoven}, but we state it again more explicitly for clarity.
		
		\begin{lemma} \label{lem:classification of pullback}
			The pullback of t-automorphisms $p^* : P_s \to Q_t$ is one of the following. Throughout, $A_s$ and $B_t$ are abelian varieties of dimension $n-1$.
			\begin{enumerate}[label = \textnormal{(\roman*)}, ref = \roman*]
				\item An isomorphism between abelian varieties of dimension $n$.
				
				\item A homomorphism $P_s = \pi_0(P_s) \times \GG_a \times A_s \twoheadrightarrow \pi_0(P_s) \times A_s \hookrightarrow Q_t$, where $Q_t$ is an abelian variety of dimension $n$.
				
				\item An open immersion such that:
				\begin{itemize}
					\item $p^* : P_s^\circ \to Q_t^\circ$ is an isomorphism between a $\GG_m$-extension of $A_s$ and $\GG_m$-extension of $B_t$.
					\item $\coker p^* = \ZZ/m$.
				\end{itemize}
				
				\item A homomorphism $P_s = \pi_0(P_s) \times \GG_a \times A_s \twoheadrightarrow \pi_0(P_s) \times A_s \hookrightarrow Q_t$, where $Q^\circ_t$ is a $\GG_m$-extension of $B_t$ and $A_s \hookrightarrow Q^\circ_t \twoheadrightarrow B_t$ is an isogeny.
				
				\item A homomorphism whose induced map on neutral components is $\GG_a \times A_s \to \GG_a \times B_t$, where $\GG_a \to \GG_a$ is either an isomorphism or $0$, and $A_s \to B_t$ is an isomorphism. \qed
			\end{enumerate}
		\end{lemma}
		
		\begin{table}[h]
			\begin{tabular}{|c||c|c|c|} \hline
				& $0$ & $\GG_m$ & $\GG_a$ \\ \hline\hline
				$0$ & (i) & - & - \\ \hline
				$\GG_m$ & - & (iii) & - \\ \hline
				$\GG_a$ & (ii) & (iv) & (v) \\ \hline
			\end{tabular}
			\caption{Five possible types of $p^* : P_s \to Q_t$ in \Cref{lem:classification of pullback}. The rows indicate the linear part of $P_s$ and the columns indicate the linear part of $Q_t$.}
			\label{table:comparison of P and Q}
		\end{table}
		
		Case (i) and (iii) will become Kodaira types $m \cdot \mathrm I_r$ for $r \ge 0$, case (ii) and (iv) will become Kodaira types $m \cdot \mathrm I_r^\pm$ for $r \ge 0$, and the last case (v) will become Kodaira types $m \cdot \mathrm I_r^*$ for $r \ge 0$, $m \cdot \mathrm{II}$, $\cdots$.

		\subsubsection{Albanese morphism and untangle for multiple fibers}
			The following two results reduce the study of $X_s$ to that of $\frac{1}{m} X_s$. The first lemma is nothing but \Cref{lem:P-action on 1/d X_s}.
			
			\begin{lemma}
				$\frac{1}{m} X_s$ admits a $P_s$-action (it may not be faithful). \qed
			\end{lemma}
			
			\begin{proposition} \label{prop:Albanese of multiple fiber}
				The Albanese torsors of $\frac{1}{m} X_s$ and $X_s$ are isomorphic, and their Albanese morphisms are related by the diagram
				\[\begin{tikzcd}
					\frac{1}{m} X_s \arrow[r, "i", hook] \arrow[d, "\alb"'] & X_s \arrow[dl, bend left=25, "\alb"] \\
					\Alb_s
				\end{tikzcd}.\]
			\end{proposition}
			\begin{proof}
				Since $h^0 (\frac{1}{m}X_s, \mathcal O) \le h^0 (Y_t, \mathcal O) = 1$, we can apply \Cref{thm:Albanese}--\ref{cor:Albanese}. Note that we can say more about the pullback $i^* : \Pic^\circ_{X_s} \to \Pic^\circ_{\frac{1}{m}X_s}$ in this case: when $\Pic^\circ_{X_s}$ has a linear part $\GG_a$, then it splits as $\Pic^\circ_{X_s} = \GG_a \times \Alb_s^\vee$ by \Cref{prop:splitting and albanese}. This implies that $i^*$ induces an inclusion $\Alb_s^\vee \hookrightarrow \Pic^\circ_{\frac{1}{m} X_s}$ whose cokernel is a unipotent group by \Cref{lem:Picard scheme of thickening}.
			\end{proof}
			
			We end this subsection with a variant of \Cref{lem:untangle primitive} for multiple fibers. This will be later used in \Cref{thm:classification of multiple fiber iv} and \ref{thm:classification of multiple fiber v} to descend the untangle of $Y$ to $X$. The statement is unfortunately a bit technical.
			
			\begin{lemma} [Improvement of \Cref{lem:untangle primitive}] \label{lem:untangle for multiple}
				Let $f : X \to S$ be a minimal $\delta$-regular abelian fibration with $P_s$ having a linear part $\GG_a$. Assume that
				\begin{itemize}
					\item The Albanese morphism $\frac{1}{m}X_s \to \Alb_s$ in \Cref{prop:Albanese of multiple fiber} is a free $\varphi$-quotient of the Albanese morphism $Y_t \to \Alb_t$.
				\end{itemize}
				Let $u : A \to \Alb_s$ be an isogeny of abelian torsors such that $A \times_{\Alb_s} \Alb_t$ is connected. Then the conclusions of \Cref{lem:untangle primitive} hold for $u$.
			\end{lemma}
			
			When $f$ is non-multiple, then the bullet assumption in the lemma is satisfied (because $m = 1$ and $Y = X$) and $u$ can be arbitrary. Hence \Cref{lem:untangle for multiple} recovers its original version when $P_s$ has a liner part $\GG_a$. When $f$ is multiple, the assumptions in the lemma will be verified for the cases (iv) and (v) of \Cref{lem:classification of pullback} (see \Cref{lem:Albanese of X is quotient of Albanese of Y} and \ref{lem:multiple fiber v:albanese}).
			
			\begin{lemma} \label{lem:kernel of Pic to P check}
				There exists a commutative diagram of algebraic groups
				\[\begin{tikzcd}
					\Pic^\circ_{X_s} \arrow[r, "i^*"] \arrow[d, "\kappa"] & \Pic^\circ_{\frac{1}{m} X_s} \arrow[d, "q^*"] \\
					\check P^\circ_s \arrow[r, "p^*"] & \check Q^\circ_t
				\end{tikzcd},\]
				where $p^*$ is the pullback constructed in \eqref{eq:pullback of t-automorphism}, $i^*$ and $q^*$ are pullbacks by \eqref{eq:q and i}, and $\kappa$ is the isogeny constructed in \Cref{prop:dual Neron model is Picard}. Moreover, the following holds.
				\begin{enumerate}
					\item There exist isomorphisms $\kappa : \ker i^* \to \ker p^*$ and $i^* : \ker \kappa \to \ker q^*$.
					\item $\ker i^* = 0$ or $\GG_a$, and $\ker \kappa \subset \ZZ/m$.
				\end{enumerate}
			\end{lemma}
			\begin{proof}
				Start with a commutative diagram of group schemes over $T$:
				\[\begin{tikzcd}
					(\Pic^\circ_f)_T \arrow[r, "q^*"] \arrow[d, "\kappa"] & \Pic^\circ_g \arrow[d, phantom, "\parallel"] \\
					\check P^\circ_T \arrow[r, "p^*"] & \check Q^\circ
				\end{tikzcd}.\]
				It is commutative over $T_0$ and hence over $T$, and the equality on the right is from \Cref{prop:dual Neron model is Picard}. Restricting this diagram over $t \in T$ and decomposing $\Pic^\circ_{X_s} \to \Pic^\circ_{Y_t} = \check Q_t^\circ$ into $q^* \circ i^*$ constructs the desired commutative diagram. To show (1), notice that $\ker q^* \subset \ZZ/m$ is finite by \Cref{lem:Picard scheme of free quotient}, $\ker \kappa$ is finite by \Cref{prop:dual Neron model is Picard}, $\ker i^*$ is unipotent by \Cref{lem:Picard scheme of thickening} (in fact, it is either $0$ or $\GG_a$), and finally $\ker p^*$ is unipotent by \Cref{thm:Edixhoven}. These force the conclusions in (1) to hold. Now (2) follows from $\ker \kappa = \ker q^* \subset \ZZ/m$.
			\end{proof}
			
			\begin{proof}[Proof of \Cref{lem:untangle for multiple}]
				Recall from the proof of \Cref{lem:untangle primitive} that the isogeny $u : A \to \Alb_s$ (with Galois group $G$) defines a subgroup $\hat G \subset \Alb_s^\vee$. The Albanese morphism $\alb : X_s \to \Alb_s$ defined an injection $\alb^* : \Alb_s^\vee \hookrightarrow \Pic^\circ_{X_s}$ by \Cref{thm:Albanese}. As pointed out in \Cref{rmk:untangle primitive}, the only part in the original lemma where the non-multiple fiber assumption was used was to show $\alb^*(\hat G) \subset \Pic^\circ_f$ is nontrivial over $S_0$. Recall from \Cref{prop:dual Neron model is Picard} that $\Pic^\circ_f$ has a separated quotient $\kappa : \Pic^\circ_f \to \check P^\circ$. A section of $\Pic^\circ_f$ that is trivial over $S_0$ passes through the torsion points $\ker (\kappa : \Pic^\circ_{X_s} \to \check P^\circ_s)$ of the central fiber $\Pic^\circ_{X_s}$. Therefore, it is enough to show that $\alb^*(\hat G) \subset \Pic^\circ_{X_s}$ has a trivial intersection with $\ker \kappa \subset \Pic^\circ_{X_s}$.
				
				\medskip
				
				Take the diagram in \Cref{lem:kernel of Pic to P check}. The claim $\alb^* (\hat G) \cap \ker \kappa = 0$ in $\Pic^\circ_{X_s}$ is equivalent to $\alb^* (\hat G) \cap \ker q^* = 0$ in $\Pic^\circ_{\frac{1}{m}X_s}$ by \Cref{prop:Albanese of multiple fiber} and \Cref{lem:kernel of Pic to P check}. The given bullet assumption is equivalent to a cartesian diagram
				\[\begin{tikzcd}
					Y_t \arrow[r, "q"] \arrow[d, "\alb"] & \tfrac{1}{m} X_s \arrow[d, "\alb"] \\
					\Alb_t \arrow[r, "r"] & \Alb_s
				\end{tikzcd},\]
				so we have $\ker q^* = \ker r^*$ and again $\alb^*(\hat G) \cap \ker q^* = 0$ is equivalent to $\hat G \cap \ker r^* = 0$ in $\Alb_s$. Since $\hat G$ is associated to $u : A \to \Alb_s$ and $\ker r^*$ is associated to $\Alb_t \to \Alb_s$, this is equivalent to saying that the fiber product $A \times_{\Alb_s} \Alb_t$ is connected (abelian torsor), which we are assuming.
			\end{proof}

	\subsection{Base change to smooth fibrations} \label{sec:base change to smooth}
		The new minimal $\delta$-regular abelian fibration $g : Y \to T$ may be smooth, strictly semistable, or unstable. We will analyze them separately throughout the remaining sections from now on.
		
		\medskip
		
		The first case is when $g$ is smooth, or the case (i) and (ii) in \Cref{lem:classification of pullback} and \Cref{table:comparison of P and Q}. In this case, the central fiber $Y_t$ is an abelian torsor of dimension $n$ by \Cref{lem:classification when P is smooth}.
		
		\begin{theorem} [Type $\mathrm I_0$] \label{thm:classification of multiple fiber i}
			In \Cref{lem:classification of pullback}(i), $\frac{1}{m}X_s$ is an abelian torsor of dimension $n$.
		\end{theorem}
		
		\begin{theorem} [Type $\mathrm I_0^+$] \label{thm:classification of multiple fiber ii}
			In \Cref{lem:classification of pullback}(ii), there exist a unique positive integer $d$ dividing $m$ with the following properties.
			\begin{enumerate}
				\item $d = 2$, $3$, $4$, or $6$. Moreover,
				\begin{enumerate}
					\item $d = 2$ if and only if $P$ has type $\mathrm I_0^*$, $\mathrm I_0^*/2$, or $\mathrm I_0^*/4$.
					\item $d = 6$ if and only if $j(E) = 0$ and $P$ has type $\mathrm{II}$ or $\mathrm{II}^*$.
					\item $d = 4$ if and only if $j(E) = 1728$ and $P$ has type $\mathrm{III}$, $\mathrm{III}/2$, $\mathrm{III}^*$, or $\mathrm{III}^*/2$.
					\item $d = 3$ if and only if $j(E) = 0$ and $P$ has type $\mathrm{IV}$, $\mathrm{IV}/3$, $\mathrm{IV}^*$, or $\mathrm{IV}^*/3$.
				\end{enumerate}
				
				\item $\frac{1}{m} X_s$ is isomorphic to $(E \times A_s)/(\langle \tilde \varphi \rangle \times G)$, where $E$ is an elliptic curve, $G$ is the Albanese stabilizer of $P$ described in \Cref{prop:semistable reduction theorem tangled ver}(2), and
				\[ \tilde \varphi (x,y) = (\omega^{\pm1} x, y + a) \qquad\mbox{for}\quad (x,y) \in E \times A_s .\]
				Here $\omega \in \Aut_0 (E)$ is a group automorphism of order $d$, and $a \in A_s$ is a torsion point of order $m$
			\end{enumerate}
		\end{theorem}
		
		\begin{definition}
			The \emph{Kodaira type} of the above cases is the indicated type multiplied by $m$.
		\end{definition}
		
		\begin{proof} [Proof of \Cref{thm:classification of multiple fiber i}]
			Since $p^* : P_s \to Q_t$ is an isomorphism of abelian varieties, \Cref{lem:classification when P is smooth} already concludes the theorem. The following is a stronger version. By \Cref{lem:P-action on 1/d X_s}, the $P_s$-action on $\frac{1}{m} X_s = X_{s, \red}$ is induced from the $Q_t$-action on $Y_t$. Hence $\frac{1}{m}X_s$ admits a transitive $P_s$-action, proving it is an abelian torsor. The $\langle \varphi \rangle$-action on $Y_t$ is by translations and $\frac{1}{m} X_s = Y_t / \varphi$.
		\end{proof}
		
		The rest of this subsection is devoted to the proof of \Cref{thm:classification of multiple fiber ii}.
		
		\medskip
		
		Since $P$ is unstable whose semistable base change $Q$ is smooth, $P$ has type $\mathrm I_0^*$, $\mathrm{II}$, $\cdots$, $\mathrm{IV}^*/3$ by \Cref{cor:semistable reduction theorem}. Let us make this more precise. Let $f^\sharp : X^\sharp \to S$ be a minimal compactification of $P$ and consider its Albanese stabilizer $G$. Shrinking $S$, we may assume that $g : Y \to T$ has a section. Then both $Y_0$ and $X^\sharp_{T_0}$ are isomorphic to $Q_0 \to T_0$. Hence $Y$ and $X^\sharp_T$ are birational, or $Y$ is the (unique) minimal model of $X^\sharp_T$. By \Cref{prop:semistable reduction theorem tangled ver}, the untangle $\tilde X^\sharp \to X^\sharp$ pulls back to a $G$-covering $\tilde Y \to Y$ with a central fiber $\tilde Y_t = E \times A_s$ for an elliptic curve $E$. We have a diagram
		\begin{equation} \label{diag:multiple fiber ii:untangle of Y}
		\begin{tikzcd}
			\tilde Y \arrow[d] \\
			Y \arrow[r] & X
		\end{tikzcd}.
		\end{equation}
		Moreover, $Y_t = \tilde Y_t / G = (E \times A_s)/G$ is a free quotient by $G \hookrightarrow \tilde Q_t = E \times A_s$ where $\tau : G \hookrightarrow E$ and $\sigma : G \hookrightarrow A_s$ are both injective. The $\mu_m = \langle \tilde \psi \rangle$-action on $\tilde Q_t = E \times A_s$ is of the form $\tilde \psi (x,y) = (\omega^{\pm1} x ,y)$ for $(x,y) \in E \times A_s$ and $\omega \in \Aut_0(E)$ a primitive $d$-th root of unity for $d = 2,3,4,$ or $6$.
		
		\medskip
		
		Note that $\tilde Y$ admits an inertia automorphism $\tilde \varphi^\sharp$ as a semistable base change of $\tilde X^\sharp$. However, it is a priori unclear whether there exists an inertia automorphism $\tilde \varphi$ on $\tilde Y$ lifting the inertia automorphism $\varphi$ on $Y$. An equivalent statement is to fill in the topright corner of \eqref{diag:multiple fiber ii:untangle of Y} by a new space $\tilde X$. It turns out this is always possible in the current case (ii), but not in case (iv) and (v) (\Cref{lem:multiple fiber iv:untangle of X} and \ref{lem:multiple fiber v:untangle of X}).
		
		\begin{lemma} \label{lem:multiple fiber ii:untangle of X}
			Shrinking $S$ \'etale locally, there exists a Galois \'etale morphism $\tilde X \to X$ completing the topright corner of \eqref{diag:multiple fiber ii:untangle of Y} to a cartesian diagram
			\[\begin{tikzcd}
				\tilde Y \arrow[r] \arrow[d] & \tilde X \arrow[d] \\
				Y \arrow[r] & X
			\end{tikzcd}.\]
			In particular, there exists a $\mu_m = \langle \tilde \varphi \rangle$-action on $\tilde Y$ lifting $\varphi$ on $Y$ and commuting with $G$.
		\end{lemma}
		\begin{proof}
			Let us first lift $\varphi$ to an automorphism $\tilde \varphi$ of $\tilde Y_t$. Recall from \Cref{prop:semistable reduction theorem untangled ver}--\ref{prop:semistable reduction theorem tangled ver} we have an isomorphism $\tilde Q_t = \tilde Y_t = E \times A_s$, and automorphisms
			\[ \tilde \psi(x,y) = (\omega^{\pm1} x, y) \mbox{ for } (x,y) \in \tilde Q_t ,\quad g(x,y) = (x + \tau(g), y + \sigma(g)) \mbox{ for } (x,y) \in \tilde Y_t .\]
			Here $\tau : G \hookrightarrow E^\omega$ and $\sigma : G \hookrightarrow A_s$ are injective homomorphisms. Hence we have a description of $\psi$ on $Q_t = \tilde Q_t / G = (E \times A_s)/G$. Now $\varphi$ on $Y_t$ is described by
			\[ \varphi(x,y) = \psi(x,y) + \varphi(0,0) \qquad\mbox{for}\quad (x,y) \in Y_t = (E \times A_s)/G .\]
			Writing $\varphi(0,0) = (c,a) \in (E \times A_s)/G$, we have $\varphi(x,y) = (\omega^{\pm1}x + c, y+a)$. We may adjust the origin of the abelian torsor $Y_t$ and assume $c = 0$. This uniquely determines $a \in A_s$. Now we can lift $\varphi$ to an automorphism $\tilde \varphi (x,y) = (\omega^{\pm1}x, y+a)$ on $\tilde Y_t$. It commutes with $G$ because $\tau(g)$ is $\omega$-fixed.
			
			\medskip
			
			Lifting the automorphism $\tilde \varphi$ of $\tilde Y_t$ to that of $\tilde Y \to T$ is formal. Recall from \Cref{prop:varphi-action} that the $\varphi$ is modeled as an isomorphism $\Phi : Y \to \zeta^* Y$ of abelian torsors over $T$. Lift it to a unique isomorphism $\tilde \Phi : \tilde Y \to \zeta^* \tilde Y$ of abelian torsors using \Cref{lem:lifting morphism of abelian torsors}. This constructs an equivariant automorphism $\tilde \varphi$ on $\tilde Y$ with condition \Cref{prop:varphi-action}(2). The $\langle \tilde \varphi \rangle$- and $G$-actions commute because they commuted on $\tilde Y_t$.
		\end{proof}
		
		\begin{proof} [Proof of \Cref{thm:classification of multiple fiber ii}]
			\Cref{lem:multiple fiber ii:untangle of X} showed that $\tilde X_s = \tilde Y_t / \tilde \varphi = (E \times A_s) / \tilde \varphi$ for $\tilde \varphi(x,y) = (\omega^{\pm1} x, y+a)$. The original $X_s$ is computed by a further quotient $\tilde X_s/G$. All statements follow immediately.
		\end{proof}

	\subsection{Base change to semistable fibrations}
		Let us study the case when $g : Y \to T$ is strictly semistable, or the cases (iii) and (iv) in \Cref{lem:classification of pullback}.
		
		\begin{theorem} [Type $\mathrm I^k_R$] \label{thm:classification of multiple fiber iii}
			In \Cref{lem:classification of pullback}(iii), $\frac{1}{m}X_s$ is isomorphic to either a type $\mathrm I_1$ semistable fiber in \Cref{thm:classification of semistable fibers} ($k = R = 1$), or a cycle of $\PP^1$-bundles $p_i : E_i \to A$ with two sections $0_i$ and $\infty_i$ for $i = 1, \cdots, R \, (\ge 2)$ that are successively identified by $0_{i+1} = \infty_i$. Moreover, the following holds.
			\begin{enumerate}
				\item $k \mid m$ and $k \mid R$. Write $m = kl$ and $R = kr$ for positive integers $l$ and $r$.
				\item The t-automorphism scheme $P$ of $f$ has type $\mathrm I_r$.
				\item There exists a free $\mu_k = \langle \varphi \rangle$-equivariant $\PP^1$-bundle $p : F \to A_s$ such that $p_i : E_i \to A$ is isomorphic to its free quotient by an automorphism $\zeta^i \circ \varphi$, where $\zeta$ is an automorphism on $F$ acting on every fiber of $p$ by multiplication by a primitive $k$-th root of unity.
				\item $p_i$ and $p_j$ are isomorphic (preserving $0$ and $\infty$) if and only if $i \equiv j \mod k$. Moreover, $A = A_s / a$ for a torsion point $a \in A_s$ of order $k$.
			\end{enumerate}
		\end{theorem}
		
		\begin{theorem} \label{thm:classification of multiple fiber iv}
			In \Cref{lem:classification of pullback}(iv), the following holds.
			\begin{enumerate}
				\item $2 \mid m \mid 2R$. Write $m = 2k$ and $R = kr$ for positive integers $k$ and $r$.
				\item The t-automorphism scheme $P$ of $f$ has type $\mathrm I^*_r$, $\mathrm I^*_r/2$, or $\mathrm I^*_r/4$.
				\item $\frac{1}{m}X_s$ is isomorphic to $(C \times A_s)/(\langle \tilde \varphi \rangle \times G)$, where $C$ is a type $\mathrm I_{2kr}$ Kodaira singular curve and the $\tilde \varphi$ and $G$ are given as one of the following.
				\begin{enumerate}
					\item[\textnormal{($\mathrm I^+_R$)}] $G$ is the Albanese stabilizer of $P$ and $\tilde \varphi$ is an order $2k$ automorphism (using the notation in \Cref{prop:semistable reduction theorem untangled ver}(2))
					\[ \tilde \varphi(i,z,y) = (-i, \tfrac{\zeta^i}{z}, y+a) \qquad\mbox{for}\quad (i,z) \in C ,\quad y \in A_s .\]
					Here $\zeta \in \GG_m$ is a a primitive $2k$-th root of unity and $a \in A_s$ is a torsion point of order $2k$.
					
					\item[\textnormal{($\mathrm I^-_R$)}] $G = 0$ or $\ZZ/2$, and $\tilde \varphi$ is an order $4k$ automorphism
					\[ \tilde \varphi(i,z,y) = (1-i, \tfrac{\varepsilon \zeta^i}{z}, y+a) \qquad\mbox{for}\quad (i,z) \in C ,\quad y \in A_s .\]
					Here $\varepsilon \in \GG_m$ is a constant, $\zeta$ is a primitive $2k$-th root, and $a \in A_s$ has order $4k$.
				\end{enumerate}
			\end{enumerate}
		\end{theorem}
		
		\begin{definition}
			The \emph{Kodaira type} of \Cref{thm:classification of multiple fiber iii} is $m \cdot \mathrm I^k_R$. The Kodaira type of \Cref{thm:classification of multiple fiber iv}(a) is $m \cdot \mathrm I^+_R/|G|$ where $|G| = 1$, $2$, or $4$ is the order of the Albanese stabilizer $G$. The Kodaira type of \Cref{thm:classification of multiple fiber iv}(b) is $m \cdot \mathrm I^-_R/|G|$ for $|G| = 1$ or $2$.
		\end{definition}
		
		\begin{remark}
			In \Cref{thm:classification of multiple fiber iii}(4), $p_i$ and $p_j$ may be isomorphic even when $i \not \equiv j \mod k$, but this isomorphism cannot preserve $0$ and $\infty$. In \Cref{thm:classification of multiple fiber iv}, the quotient $(C \times A_s)/(\langle \tilde \varphi^2 \rangle \times G)$ has type $\mathrm I^k_R$ for $R = kr$ or $2kr$. Hence $\frac{1}{m} X_s$ is a $\mu_2 = \langle \varphi \rangle$-quotient of a (very special) type $\mathrm I^k_R$ fiber. This is the description used in \Cref{thm:main}.
		\end{remark}
		
		The rest of this subsection proves these two theorems.
		
		\medskip
		
		In both (iii) and (iv), $g$ is strictly semistable and $Q^\circ_t$ is a $\GG_m$-extension of an abelian variety $B_t$ of dimension $n-1$. In particular, the central fiber $Y_t$ is reduced and hence its quotient $\frac{1}{m}X_s$ is reduced, too. Note that $Y_t = \sum_i F_i$ is a cycle of $\PP^1$-bundles over $B_t$, so we may consider the \emph{orientation} of this cycle, a preferred direction of the cycle of $\PP^1$-fibers.
		
		\begin{lemma} \label{lem:classification of multiple fiber iii:orientation}
			The following properties characterize the two cases.
			\begin{enumerate}
				\item[\textnormal{(iii)}] $\psi(z) = z$ for $z \in \GG_m \subset Q_t$, and $\varphi$ preserves the orientation of $Y_t$.
				\item[\textnormal{(iv)}] $\psi(z) = 1/z$ for $z \in \GG_m \subset Q_t$, and $\varphi$ reverses the orientation of $Y_t$.
			\end{enumerate}
		\end{lemma}
		\begin{proof}
			The inertia group automorphism $\psi : Q^\circ_t \to Q^\circ_t$ induces that on the linear part $\psi : \GG_m \to \GG_m$, which is either $\psi(z) = z$ or $1/z$. By \Cref{thm:Edixhoven}, case (iii) is when $\psi(z) = z$ and the case (iv) is when $\psi(z) = 1/z$. Recall from \Cref{prop:varphi-action} the equality $\varphi (z \cdot y) = \psi(z) \cdot \varphi(y)$ for $z \in \GG_m$ and $y \in Y_t$. Say $\varphi$ sends a component $F_i$ to $F_j$. If $\psi(z) = z$, then we have $\lim_{z \rightarrow 0} (z \cdot y) = 0_i$ and $\lim_{z \rightarrow 0} \varphi(z \cdot y) = \lim_{z \rightarrow 0} z \cdot \varphi(y) = 0_j$. Hence $\varphi$ preserves the orientation of the cycle. Similarly, if $\psi(z) = 1/z$ then $\varphi$ reverses the orientation.
		\end{proof}
		
		This observation already relates case (iv) to (iii).
		
		\begin{lemma} \label{lem:classification of multiple fiber iv:reduction to iii}
			In case (iv), $m$ is even and \eqref{diag:degree m base change} admits a factorization
			\[\begin{tikzcd}[column sep=normal]
				Y \arrow[r] \arrow[d, "g"] & Z \arrow[r] \arrow[d, "h"] & X \arrow[d, "f"] \\
				T \arrow[r, "\mu_{m/2}"] & U \arrow[r, "\mu_2"] & S
			\end{tikzcd}\]
			with the following properties:
			\begin{enumerate}
				\item The two horizontal arrows in the first row are \'etale.
				\item $h : Z \to U$ is a minimal $\delta$-regular abelian fibration with multiplicity $m/2$.
				\item The left commutative diagram falls into the case \Cref{lem:classification of pullback}(iii).
			\end{enumerate}
		\end{lemma}
		\begin{proof}
			In case (iv), $\ord \psi$ is even and hence $m$ is even. The quotient $Z = Y/\varphi^2$ constructs the diagram. The middle morphism $h$ is minimal because $Z$ is regular. To show the left diagram is case (iii), notice that its Galois group is $\mu_{m/2} = \langle \varphi^2 \rangle$ and $\psi^2(z) = z$ for $z \in \GG_m \subset Q_t$.
		\end{proof}

		\subsubsection{Case (iii)}
			Let us first study the case (iii). By \Cref{lem:classification of pullback}(iii), $P_s^\circ = Q_t^\circ$ so let us write $B_t = A_s$ for notational simplicity. Assume that $P_s$ has type $\mathrm I_r$ for a positive integer $r$ and hence $Q_t$ has type $\mathrm I_{mr}$ (because $\coker p^* = \ZZ/m$). The central fiber $Y_t = \sum_{i=1}^{mr} F_i$ has Kodaira type $\mathrm I_{mr}$, too. Since $\varphi$ preserves the orientation of $Y_t$ by \Cref{lem:classification of multiple fiber iii:orientation}, there exists a positive integer $R \mid mr$ (say $mr = lR$) such that $\varphi(F_i) = F_{i + R}$ for every $1 \le i \le mr = lR$. In particular, $\varphi^l$ preserves every component $F_i$ and $m = \ord \varphi$ is divisible by $l$. We may thus again set $m = kl$ for a positive integer $k$ and yield $R = kr$. In summary, we have
			\begin{equation} \label{eq:definition of klr}
				Y_t = \sum_{j=1}^{klr} F_j ,\quad \varphi(F_j) = F_{j + kr}, \quad R = kr, \quad\mbox{and}\quad m = kl .
			\end{equation}
			
			\begin{lemma}
				In case (iii), \eqref{diag:degree m base change} admits a factorization
				\[\begin{tikzcd}[column sep=normal]
					Y \arrow[r] \arrow[d, "g"] & X^+ \arrow[r] \arrow[d, "f^+"] & X \arrow[d, "f"] \\
					T \arrow[r, "\mu_k"] & S^+ \arrow[r, "\mu_l"] & S
				\end{tikzcd}\]
				with the following properties:
				\begin{enumerate}
					\item The two horizontal arrows in the first row are \'etale.
					\item $f^+ : X^+ \to S^+$ is a minimal $\delta$-regular abelian fibration with multiplicity $k$. Moreover, the central fiber $\frac{1}{k} X_s^+ = \sum_{j=1}^{klr} E_j^+$ satisfies the following.
					\begin{enumerate}
						\item The inertia $\mu_k = \langle \varphi^l \rangle$-action on $Y_t$ preserves every component $F_j$. In particular, $E_j^+$ is a free $\mu_k$-quotient of $F_i$ for every $i$.
						\item The inertia $\mu_l = \langle \varphi \rangle$-action on $X_s^+$ induces an isomorphism $\varphi : E_j^+ \to E_{j + kr}^+$ for every $i$.
					\end{enumerate}
				\end{enumerate}
			\end{lemma}
			\begin{proof}
				The new fibration $f^+$ is obtained by a $\mu_k = \langle \varphi^l \rangle$-quotient $X^+ = Y / \varphi^l$. Note that $\varphi^l$ preserves every irreducible component $F_i$.
			\end{proof}
			
			\begin{corollary}
				$\frac{1}{m} X_s = \sum_{i=1}^{kr} E_i$ has $kr$ number of irreducible components. Moreover, we have an isomorphism $E_j^+ \cong E_{j \mod kr}$ for $j = 1, \cdots, klr$. \qed
			\end{corollary}
			
			To simplify the numerics, we will replace $X$ with $X^+$ and assume $l = 1$ in \eqref{eq:definition of klr}. That is, from now on we assume $l = 1$, $m = k$, $Y_t = \sum_{i=1}^{kr} F_i$, and that the inertia $\mu_k = \langle \varphi \rangle$-action on $Y_t$ preserves every $F_i$, so that $E_i = F_i / \varphi$. It remains to describe the inertia automorphism $\varphi$ (and $\psi$). Take a normalization
			\[ \tilde Y_t = \bigsqcup_{i=1}^{kr} \tilde F_i \qquad\mbox{for isomorphic copies of } \PP^1 \mbox{-bundles } \tilde p_i : \tilde F_i \to B_t .\]
			When $r \ge 2$, note that $F_i = \tilde F_i$ is already smooth.
			
			\begin{lemma} \label{lem:classification of multiple fiber iii:varphi-action on Albanese}
				The $\langle \varphi \rangle$-action on $\tilde F_i$ descends to an action on $B_t$ via $\tilde p_i$. Moreover, the action is a translation by a torsion point $a \in A_s$ of order $k$.
			\end{lemma}
			\begin{proof}
				Since the $Q_t^\circ$- and $\langle \varphi \rangle$-actions on $\tilde F_i$ commute, they descend to commuting actions on the base $B_t$ by the universal property of Albanese morphism. An action commuting with every translation is a translation, so the $\langle \varphi \rangle$-action on $A_s$ is a translation by a torsion point $a_i \in A_s$. If $\varphi^i$ for $i < k$ acts as the identity on $A_s$, then it acts on the $\PP^1$ fibers of $\tilde p_i$ and hence necessarily has a fixed point on $\tilde F_i$, a contradiction. Hence $a_i$ has an order $k$.
				
				To show all $a_i$ are the same, note that the two sections $0_i$ and $\infty_i$ of $\tilde p_i$ are preserved by $\varphi$ (\Cref{lem:classification of multiple fiber iii:orientation}). Since $\tilde p_i$ is equivariant, $\varphi$ acts on them as a translation by $a_i$. The identification $\infty_i = 0_{i+1}$ shows $a_i = a_{i+1}$.
			\end{proof}
			
			\begin{corollary} \label{cor:classification of multiple fiber iii:varphi-action on Sing}
				\begin{enumerate}
					\item The $\langle \varphi \rangle$-action on $\Sing(Y_t) = \bigsqcup_{i=1}^{kr} (F_i \cap F_{i+1})$ is a translation by $a \in A_s$ on every component.
					\item $\tilde E_i \to A_s/a$ is a free $\varphi$-quotient of $\tilde F_i \to A_s$. In particular, it is a $\PP^1$-bundle.\qed
				\end{enumerate}
			\end{corollary}
			
			\begin{lemma} \label{lem:classification of multiple fiber iii:psi-action}
				There exists a primitive $k$-th root of unity $\zeta \in \GG_m \subset Q_t^\circ$ such that the inertia $\langle \psi \rangle$-action on $Q_t$ is
				\[ \psi(f_i) = \zeta^i \cdot f_i \qquad\mbox{for}\quad f_i \in Q_t \mbox{ with } \pi_0(f_i) = i \ \in \ \pi_0(Q_t) = \ZZ/kr .\]
			\end{lemma}
			\begin{proof}
				We may write $Q_t = \ZZ/kr \times Q_t^\circ$ by \Cref{lem:pi_0 sequence splits}, and write its element by $(i, g) \in \ZZ/kr \times Q_t^\circ$. Recall the equality $Q_t^\psi = P_s$ and $\pi_0(P_s) = k \cdot \ZZ/kr \subset \ZZ/kr = \pi_0(Q_t)$, so we already have $\psi(i,g) = (i,g)$ when $k \mid i$. Recall the equality $\varphi ((i,g) \cdot y) = \psi(i,g) \cdot \varphi(y)$ of \Cref{prop:varphi-action}. Since $\varphi$ preserves all $F_i$ by assumption and $\pi_0(Q_t) = \ZZ/kr$ cycles the components $F_i$ (\Cref{cor:pi_0-action on X_s}), this shows that $\psi$ is of the form
				\[ \psi(i,g) = (i, \psi_i (g)) \qquad\mbox{for}\quad \psi_i : Q_t^\circ \to Q_t^\circ .\]
				Since $\psi$ acted trivially on $Q_t^\circ$, we have $\psi(i,g) = \psi(0,g) \cdot \psi(i,1) = (0,g) \cdot (i,\psi_i(1)) = (i, \psi_i(1) \cdot g)$. This shows $\psi_i(g) = \zeta_i \cdot g$ for $\zeta_i = \psi_i(1) \in Q_t^\circ$. The equality $\psi(i,g^i) = \psi(1, g)^i$ shows $\zeta_i = \zeta^i$ for a single $\zeta \in Q_t^\circ$. Furthermore, $\zeta$ is $k$-torsion because $\psi^k = \id$. Again $Q_t^\psi$ is precisely $P_s$, so this implies that $\zeta$ has order precisely $k$.
				
				\medskip
				
				It remains to prove that $\zeta \in Q_t^\circ$ is contained in $\GG_m$. Take a point $y \in F_1 \cap F_{kr}$, a unique rational curve $R \subset F_1$ through $y$, and a unique point $y' = R \cap F_2 \subset F_1 \cap F_2$. Recall from \Cref{lem:stabilizer of Ir} that $Q_t$ acts on $\Sing(Y_t)$ transitively by cycling the components (with stabilizer $\GG_m$). Hence there exists $f_1 \in Q_t$ such that $\pi_0(f_1) = 1$ and $y' = f_1 \cdot y$.
				
				Now apply $\varphi$ to $y, y' \in R \subset F_1$ to yield two points on a rational curve $\varphi(y), \varphi(y') \in \varphi(R) \subset F_1$. By \Cref{cor:classification of multiple fiber iii:varphi-action on Sing}, if we fix any $f_0 \in Q_t^\circ$ whose image under $Q_t^\circ \to A_s$ is $a$, then $\varphi(y) = f_0 \cdot y$ and $\varphi(y') = f_0 \cdot y'$. We now have a sequence of identities
				\begin{align*}
					\varphi(y') &= \varphi(f_1 \cdot y) = \psi(f_1) \cdot \varphi(y) = (\zeta \cdot f_1) \cdot (f_0 \cdot y) \\
					&= \zeta \cdot (f_0 \cdot (f_1 \cdot y)) = \zeta \cdot (f_0 \cdot y') = \zeta \cdot \varphi(y') .
				\end{align*}
				But $\varphi(y') \in F_1 \cap F_2$ has a stabilizer $\GG_m$ again by \Cref{lem:stabilizer of Ir}. This means $\zeta \in \GG_m \subset Q_t^\circ$.
			\end{proof}
			
			The (normalized) irreducible component $\tilde F = \tilde F_i \to A_s$ is a compactification of the $\GG_m$-torsor $Q_t^\circ \to A_s$, or equivalently a $\PP^1$-bundle with two sections $0$ and $\infty$. Hence $\tilde F \setminus \infty \to A_s$ can be considered as a line bundle with the zero section $0$. In this sense, we call $\tilde F = (\tilde F, 0, \infty) \to A_s$ a line bundle.
			
			\begin{lemma}
				There exists an isomorphism $f : \tilde F_i \to \tilde F_j$ of line bundles (i.e., $f(0_i) = 0_j$ and $f(\infty_i) = \infty_j$) such that the $\langle \varphi \rangle$-actions on $\tilde F_i$ and $\tilde F_j$ are related via
				\[ \varphi(f(y)) = \zeta^{j-i} \cdot f(\varphi(y)) .\]
			\end{lemma}
			\begin{proof}
				Take any $f = f_{j-i} \in Q_t$ with $\pi_0(f) = j-i \in \pi_0(Q_t) = \ZZ/kr$. Then it induces an orientation-preserving isomorphism $f : F_i \to F_j$. The claim is the equality
				\[ \varphi(f \cdot y) = \psi(f) \cdot \varphi(y) = \zeta^{j-i} \cdot (f \cdot \varphi (y)) .\qedhere\]
			\end{proof}
			
			\begin{lemma} \label{lem:classification of multiple fiber iii:description of E_i}
				$\tilde E_i = \tilde F_i/\varphi$ and $\tilde E_j = \tilde F_j/\varphi$ are isomorphic as line bundles over $A = A_s/a$ if and only if $i \equiv j \mod k$.
			\end{lemma}
			\begin{proof}
				Write $\tilde F = \tilde F_i$ for simplicity. \Cref{lem:classification of multiple fiber iii:varphi-action on Albanese} constructed a commutative diagram
				\[\begin{tikzcd}
					\tilde F \arrow[r, "\varphi"] \arrow[d, "\tilde p"] & \tilde F \arrow[d, "\tilde p"] \\
					A_s \arrow[r, "t_a"] & A_s
				\end{tikzcd},\]
				or equivalently an isomorphism $\tilde F \to t_a^* \tilde F$ of $\PP^1$-bundles over $A_s$ (with a cocycle condition). Such an isomorphism respects $0$ and $\infty$, so it is an isomorphism of line bundles. That is, the $\langle \varphi \rangle$-action on $\tilde F$ defines a $\ZZ/m$-equivariant line bundle structure on $\tilde F \setminus \infty \to A_s$.
				
				Consider a left exact sequence (e.g., \cite[\S 7.1]{dol:invariant})
				\[ 0 \to \Hom (\ZZ/k, \GG_m) = \mu_k \to \Pic^{\ZZ/k}(A_s) \to \Pic(A_s) .\]
				For each $i$, the datum $(\tilde F_i, \varphi)$ in $\Pic^{\ZZ/k}(A_s)$ has the same image in $\Pic(A_s)$. The previous lemma says that $(\tilde F_i, \varphi)$ and $(\tilde F_j, \varphi)$ differ by $\zeta^{j-i} \in \mu_k$. This shows these $\ZZ/k$-equivariant structures are isomorphic if and only if $i \equiv j \mod k$. In other words, the quotients $\tilde E_i = \tilde F_i/\varphi$ and $E_j = \tilde F_j/\varphi$ are isomorphic as line bundles over $A = A_s/a$ if and only if $i \equiv j \mod k$.
			\end{proof}
			
			\begin{proof} [Proof of \Cref{thm:classification of multiple fiber iii}]
				We have seen in general that $\frac{1}{m} X_s = \sum_{i=1}^{kr} E_i$ consists of $R = kr$ number of irreducible components. Thanks to \Cref{lem:quotient and normalization}--\ref{lem:quotient of conductor}, we may consider $\frac{1}{m}X_s$ as an identification of $\tilde E_i \to A$'s by $\infty_i = 0_{i+1}$. Note that $\tilde E_i \to A$ isomorphic to $\tilde E_i^+ \to A$ for $i = 1, \cdots, kr$, and $\tilde E_i^+ \to A$ is a free $\mu_k$-quotient of $\tilde F_i \to A_s$ by the inertia action $\varphi$, which differs in cycle $k$.
			\end{proof}

		\subsubsection{Case (iv)}
			Let us now study the case (iv). In \Cref{lem:classification of multiple fiber iv:reduction to iii}, we have factorized the original diagram \eqref{diag:degree m base change} into $T \to U \to S$ of degrees $m/2$ and $2$. Since $h : Z \to U$ is a semistable-like multiple fibration, it has type $\frac{m}{2} \cdot \mathrm I^k_R$ for $k \mid \frac{m}{2}$ and $k \mid R$ by \Cref{thm:classification of multiple fiber iii}. Letting $\frac{m}{2} = kl$, $g : Y \to T$ is semistable of type $\mathrm I_{Rl}$. In fact, $l$ is trivial in this case.
			
			\begin{lemma}
				$l = 1$, or equivalently $m = 2k$.
			\end{lemma}
			\begin{proof}
				Consider the $\mu_m = \mu_{2kl} = \langle \varphi \rangle$-action on $Y$. By \Cref{lem:classification of multiple fiber iii:orientation}, $\varphi$ reserves the orientation of the cycle $Y_t = \sum_{i=1}^{Rl} F_i$. Up to reindexing $F_i$, we may assume $\varphi(F_i) = F_{-i}$ or $F_{1-i}$. Note that $\varphi^2$ preserves all $F_i$ in either case. But recall from \eqref{eq:definition of klr} that $\mu_{m/2} = \mu_{kl} = \langle \varphi^2 \rangle$-action on $Y_t$ is of the form $\varphi^2 (F_i) = F_{i + R}$. This shows $l = 1$.
			\end{proof}
			
			Therefore, we have a factorization
			\[\begin{tikzcd}
				Y \arrow[r] \arrow[d, "g"] & Z \arrow[r] \arrow[d, "h"] & X \arrow[d, "f"] \\
				T \arrow[r, "\mu_k"] & U \arrow[r, "\mu_2"] & S
			\end{tikzcd},\]
			where $g$ is semistable of type $\mathrm I_R$ and $h$ is semistable-like multiple of type $k \cdot \mathrm I^k_R$.
			
			\begin{lemma}
				$P \to S$ has type $\mathrm I_r^*$, $\mathrm I_r^*/2$, or $\mathrm I_r^*/4$ for a positive integer $r$.
			\end{lemma}
			\begin{proof}
				The t-automorphism scheme $P$ is unstable, and the degree $2$ morphism $U \to S$ reduces it to semistable t-automorphism scheme of $h$. Apply \Cref{prop:semistable reduction theorem untangled ver}--\ref{prop:semistable reduction theorem tangled ver} to the minimal compactification $f^\sharp : X^\sharp \to S$ of $P$.
			\end{proof}
			
			By \Cref{thm:classification of unstable fibers} and \ref{thm:classification of component group}, $P$ admits an untangle to a type $\mathrm I_r^*$ group
			\[ 0 \to G \to \tilde P \to P \to 0 ,\]
			where $G \subset \pi_0(\tilde P_s) = (\ZZ/2)^{\times 2}$ for even $r$ and $\ZZ/4$ for odd $r$. One similarly tries to \emph{untangle} $X$ into a $G$-covering $\tilde X \to X$ as in \Cref{prop:untangle for unstable} or \Cref{lem:multiple fiber ii:untangle of X}. This is not always possible, so we need to study this carefully.
			
			\medskip
			
			Let $G$ be the Albanese stabilizer of $P$. A similar process to \eqref{diag:multiple fiber ii:untangle of Y} constructs a $G$-covering $\tilde Y \to Y$ with a central fiber $\tilde Y_t = C \times A_s$ for a Kodaira singular curve $C$ of type $\mathrm I_{2kr}$:
			\begin{equation} \label{diag:multiple fiber iv:untangle of Y}
			\begin{tikzcd}
				\tilde Y \arrow[d] \\
				Y \arrow[r] & Z \arrow[r] & X
			\end{tikzcd}.
			\end{equation}
			We want to check whether $\tilde Y \to Y$ descends to $\tilde X \to X$. To do so, recall the homomorphism $p^* : P_s^\circ = \GG_a \times A_s \twoheadrightarrow A_s \subset Q_t^\circ$ produces an $(n-1)$-dimensional abelian subvariety of $Q_t^\circ$. This means that $Y_t$ has an $(n-1)$-dimensional Albanese torsor by \Cref{cor:dimension of Albanese}:
			\[ \alb : Y_t \to \Alb_t = \Alb_{Y_t} .\]
			
			\begin{lemma} \label{lem:Albanese of X is quotient of Albanese of Y}
				The Albanese morphism $\frac{1}{m}X_s \to \Alb_s$ is a free $\varphi$-quotient of the Albanese morphism $Y_t \to \Alb_t$.
			\end{lemma}
			\begin{proof}
				Similar to \Cref{lem:classification of multiple fiber iii:varphi-action on Albanese}, the $\langle \varphi \rangle$-action descends to a translation action on $\Alb_t$. Suppose that $\varphi^i$ for $i < m = 2k$ acts trivially on $\Alb_t$. Then it acts on the fibers $C$ of $Y_t \to \Alb_t$. But every action on a cycle of $\PP^1$'s reversing the orientation has a fixed point, so this contradicts the fact that the $\langle \varphi \rangle$-action on $Y_t$ is free. Now the free quotient
				\[ \alpha : \tfrac{1}{2k} X_s = Y_t/\varphi \to \Alb_t/\varphi \]
				is an $A_s$-equivariant $C$-fiber bundle. The Albanese torsor of $\frac{1}{2k}X_s$ is $\Alb_s$ (\Cref{prop:Albanese of multiple fiber}), so $\alpha$ factors through $\frac{1}{2k}X_s \to \Alb_s \to \Alb_t/\varphi$. The latter morphism $\Alb_s \to \Alb_t/\varphi$ is an isogeny of abelian torsors and hence it is finite. But $\alpha$ already has connected fibers $C$, so $\Alb_s = \Alb_t/\varphi$ and $\alpha$ is the Albanese morphism.
			\end{proof}
			
			Combine \Cref{prop:semistable reduction theorem tangled ver}(3) with \Cref{lem:Albanese of X is quotient of Albanese of Y} to form a cartesian diagram
			\[\begin{tikzcd}
				C \times A_s \arrow[r] \arrow[d, "\pr_2"] & Y_t \arrow[r] \arrow[d, "\alb"] & \tfrac{1}{m} X_s \arrow[d, "\alb"] \\
				A_s \arrow[r, "G"] & \Alb_t \arrow[r, "\mu_{2k}"] & \Alb_s
			\end{tikzcd}.\]
			The composition $A_s \to \Alb_s$ is an isogeny of abelian torsors, so it is automatically Galois. Its Galois group $H$ sits in a short exact sequence
			\begin{equation} \label{eq:multiple fiber iv:ses of H}
				0 \to G \to H \to \mu_{2k} \to 0 .
			\end{equation}
			The following answers the question whether the $G$-covering of $X$ exists.
			
			\begin{lemma} \label{lem:multiple fiber iv:untangle of X}
				The sequence \eqref{eq:multiple fiber iv:ses of H} splits if and only if there exists a Galois \'etale morphism $\tilde X \to X$ that completes \eqref{diag:multiple fiber iv:untangle of Y} into a cartesian diagram
				\[\begin{tikzcd}
					\tilde Y \arrow[r] \arrow[d] & \tilde Z \arrow[r] \arrow[d] & \tilde X \arrow[d] \\
					Y \arrow[r] & Z \arrow[r] & X
				\end{tikzcd}.\]
				In this case, the $\mu_{2k} = \langle \varphi \rangle$-action on $Y$ lifts to a $\mu_{2k} = \langle \tilde \varphi \rangle$-action on $\tilde Y$ commuting with $G$.
			\end{lemma}
			\begin{proof}
				If $H \cong G \times \mu_{2k}$, then take a quotient by $\mu_{2k}$ first to obtain a cartesian diagram of abelian torsors (where $a \in A_s$ has order $2k$)
				\[\begin{tikzcd}
					A_s \arrow[r] \arrow[d] & A_s/a \arrow[d, "u"] \\
					\Alb_t \arrow[r] & \Alb_s
				\end{tikzcd}.\]
				Hence we can apply \Cref{lem:untangle for multiple} (together with \Cref{lem:Albanese of X is quotient of Albanese of Y}) to the isogeny $u : A_s/a \to \Alb_s$ and construct $\tilde X \to X$ as claimed. Conversely, if we can fill in the topright corner of \eqref{diag:multiple fiber v:untangle of Y}, then the inertia $\mu_{2k} = \langle \tilde \varphi \rangle$-action on $\tilde Y$ satisfies $\mu_{2k} \cap G = 0$ and is a lifting of $\varphi$. That is, this induces a splitting $H \cong G \times \mu_{2k}$.
			\end{proof}
			
			\begin{lemma} \label{lem:multiple fiber iv:varphi-action for split H}
				Assume the sequence \eqref{eq:multiple fiber iv:ses of H} splits. Then the $\mu_{2k} = \langle \tilde \varphi \rangle$-action on $Y_t = C \times A_s$ for a type $\mathrm I_{2kr}$ Kodaira curve $C$ is
				\[ \tilde \varphi(i,z,y) = (-i, \tfrac{\zeta^i}{z}, y+a) \qquad\mbox{for}\quad (i,z) \in C , \quad y \in A_s ,\]
				where $\zeta$ is a primitive $2k$-th root of unity and $a \in A_s$ is a torsion point of order $2k$.
			\end{lemma}
			\begin{proof}
				By \Cref{lem:automorphism of chain of rational curves}, $\tilde \varphi$ acts (freely and) diagonally on $\tilde Y_t = C \times A_s$. Hence it acts on $A_s$ by a translation by a torsion point $a \in A_s$ of order $2k$, and acts on the curve $C$ by reversing the orientation.
				
				Shrinking $S$, the torsion point $a \in A_s \subset \tilde Q_t^\circ = \GG_m \times A_s$ extends to a torsion section $a : T \to \tilde Q$ of order $2k$. This defines a free $\mu_{2k}$-action on $\tilde Y$, so its quotient $\bar g : \bar Y = \tilde Y / a \to S$ is a minimal abelian fibration with $\bar Y_t = C \times A_s/a$. The $\langle a \rangle$- and $\langle \tilde \varphi \rangle$-actions on $\tilde Y$ commute by \Cref{prop:varphi-action}, so $\tilde \varphi$ descends to a $\mu_{2k} = \langle \bar \varphi \rangle$-action on $\bar Y$. Now \Cref{prop:semistable reduction theorem untangled ver} (or more precisely \Cref{cor:semistable reduction theorem:psi-action}) determines $\bar \varphi(i,z,\bar y) = (-i, \zeta^i z^{-1}, \bar y)$, so the free action $\tilde \varphi$ must be as claimed.
			\end{proof}
			
			\begin{lemma} \label{lem:multiple fiber iv:non-split H}
				Assume the sequence \eqref{eq:multiple fiber iv:ses of H} does not split. Then $H \cong \ZZ/4k$ or $\ZZ/4k \times \ZZ/2$. Moreover, the generator $\tilde \varphi$ of $\ZZ/4k$ is of the form
				\[ \tilde \varphi(i,z,y) = (1-i, \tfrac{\varepsilon \zeta^i}{z}, y+a) \qquad\mbox{for}\quad (i,z) \in C , \quad y \in A_s ,\]
				where $\varepsilon \in \GG_m$ is a constant, $\zeta \in \GG_m$ is a primitive $2k$-th root of unity, and $a \in A_s$ is a torsion point of order $4k$.
			\end{lemma}
			\begin{proof}
				Let us first show that $H$ is isomorphic to the desired group. If $G = \ZZ/2$, then $H = \ZZ/4k$ is the only nontrivial $\ZZ/2$-extension of $\mu_{2k}$. For the remaining possibilities we need to divide the case by parity of $r$. When $r$ is odd, we can have $G = \ZZ/4$. Take an arbitrary lift $\tilde \varphi \in H$ of $\varphi \in \mu_{2k}$. Then $\tilde \varphi^{2k}$ is a nontrivial element in $G$ because the sequence does not split. Since $\tilde \varphi$ acts on the type $\mathrm I_{2kr}$ curve $C = \sum_{i=1}^{2kr} C_i$ by reversing the orientation, $\tilde \varphi(C_i) = C_{-i}$ or $C_{1-i}$ and hence $\varphi^{2k}$ always preserves every $C_i$. Hence $\tilde \varphi^{2k} \in G = \ZZ/4$ must be $2 \in \ZZ/4$ by the description of \Cref{prop:semistable reduction theorem untangled ver}--\ref{prop:semistable reduction theorem tangled ver}. This means $\tilde \varphi$ has order $4k$ and hence $H \cong \ZZ/4k \times \ZZ/2$ as claimed. Similarly when $r$ is even, $G = (\ZZ/2)^{\times 2}$ contains a unique nontrivial element $g$ that preserves every $F_i$. This shows $\tilde \varphi^{2k} = g$ and again $H \cong \ZZ/4k \times \ZZ/2$.
				
				\medskip
				
				Let us now determine the generator $\tilde \varphi$ of $\ZZ/4k$. In any case above, $\tilde \varphi$ is the lift of $\varphi \in \mu_{2k}$ such that $\tilde \varphi^{2k}$ is a t-automorphism acting as
				\[ \tilde \varphi^{2k} (i,z,y) = (i,-z,y + \sigma) \]
				for a nonzero $2$-torsion point $\sigma \in A_s$ (\Cref{prop:semistable reduction theorem tangled ver}(3)). Assume without loss of generality $G = \ZZ/2$. The automorphism $\tilde \varphi$ on $C \times A_s$ is free and diagonal as usual, so $\tilde \varphi$ must act on $A_s$ as a translation by a torsion point $a \in A_s$ of order $4k$ (with $\sigma = 2ka$). It acts on $C$ by $\tilde \varphi(C_i) = C_{-i}$ or $C_{1-i}$ as above. In the former case, one can show that $\tilde \varphi(i,z,y) = (-i, \zeta^i z^{-1}, y+a)$ or $(-i, -\zeta^i z^{-1}, y+a+\tau)$ as usual. But such $\tilde \varphi$ can never have the $2k$-th power $\tilde \varphi^{2k}$ as described above. Hence $\tilde \varphi$ sends $C_i$ to $C_{1-i}$.
				
				We have already seen in \Cref{cor:semistable reduction theorem:psi-action} that $\tilde \psi (i,z,y) = (-i, \zeta^i z^{-1}, y)$. Now \Cref{prop:varphi-action} holds on $Y_t$, so on its double cover $\tilde Y_t$ we have at least
				\[ \tilde \varphi (i,z,y) = (-i, \zeta^iz^{-1}, y) \cdot \tilde \varphi(0,1,0) \quad\mbox{or}\quad (-i, -\zeta^iz^{-1}, y + \tau) \cdot \tilde \varphi(0,1,0) .\]
				But the latter case is impossible because we have already assumed that $\tilde \varphi$ acts on $A_s$ as a translation by $a$ (not $a + \tau$). Writing $\tilde \varphi(0,1,0) = (1,\varepsilon,a)$, this concludes the desired description of $\tilde \varphi$.
			\end{proof}
			
			\begin{proof} [Proof of \Cref{thm:classification of multiple fiber iv}]
				We have already proved (1) and (2) along the way. When \eqref{eq:multiple fiber iv:ses of H} splits, \Cref{lem:multiple fiber iv:untangle of X}--\ref{lem:multiple fiber iv:varphi-action for split H} applies and (3a) follows from the isomorphism
				\[ \tfrac{1}{2k} X_s = \tfrac{1}{2k} \tilde X_s / G = (C \times A_s) / (\langle \tilde \varphi \rangle \times G) .\]
				Even when \eqref{eq:multiple fiber iv:ses of H} does not split, we still have an isomorphism $\tfrac{1}{2k} X_s = (C \times A_s) / H$, and \Cref{lem:multiple fiber iv:non-split H} concludes.
			\end{proof}

	\subsection{Base change to unstable fibrations}
		Finally, let us consider the case when $g : Y \to T$ is unstable, the last remaining case \Cref{lem:classification of pullback}(v).
		
		\begin{theorem} \label{thm:classification of multiple fiber v}
			In \Cref{lem:classification of pullback}(v), $\frac{1}{m}X_s$ is isomorphic either any unstable fibers in \Cref{thm:classification of unstable fibers} except $\mathrm I_r^*$, $\mathrm I_r^*/2$, and $\mathrm I_r^*/4$ for $r \ge 1$, or a quotient $(C \times A_s) / (\langle \tilde \varphi \rangle \times G)$ for two commuting free actions by $\tilde \varphi$ and $G$ as follows.
			\begin{enumerate}
				\item[\textnormal{($\mathrm I_R^*$)}] $C = E_1 + E_2 + 2(\sum_{j=0}^{mr} F_j) + E_3 + E_4$ is a type $\mathrm I_R^*$ Kodaira singular curve for $R = mr$, the group $G = 0$, $\ZZ/2$, $\ZZ/4$, or $(\ZZ/2)^{\times 2}$ acts on $C \times A_s$ as Albanese stabilizers, and $\langle \tilde \varphi \rangle = \mu_m$ acts on $E_i \times A_s$ and $F_j \times A_s$ as
				\[ \tilde \varphi(z_i,y) = (z_i, y + a) ,\ \tilde \varphi(z_j,y) = (\zeta^j z_j, y + a) \quad\mbox{for}\quad (z_i, y) \in E_i \times A_s, \ (z_j,y) \in F_j \times A_s .\]
				Here $\zeta \in \GG_m$ is a primitive $m$-th root of unity and $a \in A_s$ has order $m$.
				
				\item[\textnormal{$\bullet$}] $\langle \tilde \varphi \rangle = \mu_m$ or $\mu_{2m}$, and $G = \langle g \rangle = 0$ or $\ZZ/2$. They act freely on $C \times A_s$ as
				\[ \tilde \varphi (x,y) = (\alpha(x), y+a) ,\quad g(x,y) = (\beta(x), y+b) \qquad\mbox{for}\quad (x,y) \in C \times A_s ,\]
				where $a, b \in A_s$ are torsion points. All possibilities for $(C, \alpha, \beta)$ are described in \Cref{table:exceptional unstable-like fibers}, where we use the notation $C = \sum_{i=1}^4 E_i + 2F$ for type $\mathrm I_0^*$, $C = \sum_{i=1}^3 E_i$ for type $\mathrm{IV}$, and $C = \sum_{i=1}^3 (E_i + 2F_i) + 3H$ for type $\mathrm{IV}^*$, respectively.
			\end{enumerate}
		\end{theorem}
		
		\begin{table}[t]
			\begin{tabular}{|c||c|c|c|c|c||c|c|} \hline
				Type of $X$ & $C$ & $\alpha$ & $\beta$ & $\langle \tilde \varphi \rangle$ & $G$ & Mult. $m$ & Type of $P$ \\\hline\hline
				$\mathrm I_0^*$-a & $\mathrm I_0^*$ & swaps $E_{1,2}$ & - & $\mu_m$ & $0$ & $2 \mod 4$ & $\mathrm{III}$ or $\mathrm{III}^*$ \\\hline
				$\mathrm I_0^*$-a$/2$ & $\mathrm I_0^*$ & swaps $E_{1,2}$ & swaps $E_{3,4}$ & $\mu_m$ & $\ZZ/2$ & $2 \mod 4$ & $\mathrm{III}/2$ or $\mathrm{III}^*/2$ \\\hline
				$\mathrm I_0^*$-b & $\mathrm I_0^*$ & cycles $E_{1 .. 4}$ & - & $\mu_{2m}$ & $0$ & $2 \mod 4$ & $\mathrm{III}/2$ or $\mathrm{III}^*/2$ \\\hline
				$\mathrm I_0^*$-c & $\mathrm I_0^*$ & cycles $E_{1 .. 3}$ & - & $\mu_m$ & $0$ & $3 \mod 6$ & $\mathrm{II}$ or $\mathrm{II}^*$ \\\hline
				$\mathrm{IV}$-a & $\mathrm{IV}$ & swaps $E_{1,2}$ & - & $\mu_m$ & $0$ & $\pm2 \mod 6$ & $\mathrm{II}$ or $\mathrm{II}^*$ \\\hline
				$\mathrm{IV}^*$-a & $\mathrm{IV}^*$ & swaps $F_{1,2}$ & - & $\mu_m$ & $0$ & $\pm2 \mod 6$ & $\mathrm{II}$ or $\mathrm{II}^*$ \\\hline
			\end{tabular}
			\caption{The additional unstable-like multiple fibers (multiplicity $m$ is suppressed in the type for simplicity).}
			\label{table:exceptional unstable-like fibers}
		\end{table}
		
		\begin{definition}
			The \emph{Kodaira type} of \Cref{thm:classification of multiple fiber v} is either an unstable type or the additional types in \Cref{table:exceptional unstable-like fibers}, both multiplied by $m$.
		\end{definition}
		
		\begin{remark} \label{rmk:classification of multiple fiber v:exceptional cases}
			The additional possibilities in \Cref{table:exceptional unstable-like fibers} were first observed by Matsushita \cite[Table~4]{mat01} in the context of Lagrangian fibrations of symplectic fourfolds, using a different method. Some items in Matsushita's table are divided into two in our descriptions. For example, both type $\mathrm I_0^*/4$ and $\mathrm I_0^*$-a$/2$ correspond to Matsushita's type $\mathrm I_0^*$-$4$.
			
			The additional types are named after the type of the curve $C$ in the theorem (equivalently, \emph{characteristic cycle} in Hwang--Oguiso's terminology). Note however that the type of its t-automorphism group $P$ can be very different.
		\end{remark}
		
		\begin{proposition} \label{prop:classification of multiple fiber v multiplicity}
			In \Cref{thm:classification of multiple fiber v}, the multiplicity $m$ should satisfy the following. Let $T$ be the type of $\frac{1}{m} X_s$.
			\begin{enumerate}
				\item If $T$ is isogenous to $\mathrm{II}$ or $\mathrm{II}^*$, then $\gcd (6, m) = 1$.
				\item If $T$ is isogenous to $\mathrm{III}$ or $\mathrm{III}^*$, then $2 \nmid m$.
				\item If $T$ is isogenous to $\mathrm{IV}$ or $\mathrm{IV}^*$, then $3 \nmid m$.
				\item If $T$ is isogenous to $\mathrm I_0^*$, then $2 \nmid m$.
				\item If $T$ is isogenous to $\mathrm I_R^*$ for $R \ge 1$, then $2 \nmid m$.
				\item For the additional types $T$, see \Cref{table:exceptional unstable-like fibers}.
			\end{enumerate}
		\end{proposition}
		
		The rest of this subsection will be devoted to the proof of \Cref{thm:classification of multiple fiber v} and \Cref{prop:classification of multiple fiber v multiplicity}. Recall from \Cref{lem:classification of pullback}(v) that we have an isomorphism (write $B_t = A_s$ for notational simplicity)
		\[ Q_t^\circ = P_s^\circ = \GG_a \times A_s .\]
		Both Albanese torsors have dimension $n-1$ (\Cref{prop:splitting and albanese} and \ref{prop:Albanese of multiple fiber}):
		\[ \alb : Y_t \to \Alb_t = \Alb_{Y_t} ,\qquad \alb : \tfrac{1}{m}X_s \to \Alb_s .\]
		The former is a locally trivial bundle with fiber $C$, an unstable Kodaira singular curve by \Cref{thm:classification of unstable fibers}. The latter is also a locally trivial $C$-fiber bundle by the following:
		
		\begin{lemma} \label{lem:multiple fiber v:albanese}
			The Albanese morphism $\frac{1}{m}X_s \to \Alb_s$ is a free $\varphi$-quotient of the Albanese morphism $Y_t \to \Alb_t$.
		\end{lemma}
		\begin{proof}
			Same as \Cref{lem:Albanese of X is quotient of Albanese of Y}, but using the following fact instead: every automorphism of an unstable Kodaira singular curve $C$ has a fixed point. The proof of this latter claim is a case-by-case analysis. For example, a type $\mathrm {II}$ curve has one singular point, so any automorphism of $C$ has to fix it.
		\end{proof}
		
		Since $g : Y \to T$ is an unstable fibration, it has the Albanese stabilizer $G$ and untangle $\tilde Y \to Y$ in \Cref{prop:untangle for unstable}, so that $\tilde g : \tilde Y \to T$ has a central fiber $\tilde Y_t = C \times A_s$. Applying \Cref{prop:semistable reduction theorem untangled ver}--\ref{prop:semistable reduction theorem tangled ver} to $Y$, we obtain an additional base change $U \to T$ of degree $d = 2$, $3$, $4$, or $6$ with a commutative diagram
		\begin{equation} \label{diag:multiple fiber v:untangle of Y}
		\begin{tikzcd}
			\tilde Z \arrow[r, dashed] \arrow[d] & \tilde Y \arrow[d] \\
			Z \arrow[r, dashed] \arrow[d, "h"] & Y \arrow[r, "q"] \arrow[d, "g"] & X \arrow[d, "f"] \\
			U \arrow[r, "\mu_d"] & T \arrow[r, "\mu_m"] & S
		\end{tikzcd}.
		\end{equation}
		Let $R \to U$ and $\tilde R \to U$ be the t-automorphism schemes of $h$ and $\tilde h$, respectively. Since $U \to S$ is Galois with Galois group $\mu_{dm}$, $R$ admits a $\mu_{dm} = \langle \psi \rangle$-action by \Cref{prop:psi-action} and $Z$ admits a $\mu_{dm} = \langle \varphi \rangle$-action by \Cref{prop:varphi-action}. Similarly, $\tilde R$ admits a $\mu_d = \langle \tilde \psi^m \rangle$-action and $\tilde Z$ admits a $\mu_d = \langle \tilde \varphi^m \rangle$-action, the lifts of $\psi^m$ and $\varphi^m$. Note however that $\varphi$ (resp. $\psi$) may \emph{not} lift to $\tilde Z$ (resp. $\tilde R$) because the object in the topright corner of \eqref{diag:multiple fiber v:untangle of Y} is missing (similar problem to case (iv) in the diagram \eqref{diag:multiple fiber iv:untangle of Y}). Note also that $G$- and $\langle \tilde \varphi^m \rangle$-actions on $\tilde Z$ commute because the upper left square of \eqref{diag:multiple fiber v:untangle of Y} is cartesian.
		
		\begin{proof} [Proof of \Cref{prop:classification of multiple fiber v multiplicity}(1--5)]
			Suppose that $h$ is smooth, or equivalently $\tilde Y$ has type $\mathrm I_0^*$, $\cdots$, or $\mathrm{IV}^*$. Recall from \Cref{prop:semistable reduction theorem tangled ver}(2) that $\psi^m$ acts on $R_u = (E \times A_s)/G$ by $\psi^m (x,y) = (\omega^{\pm1} x, y)$ for $(x,y) \in R_u$. By \Cref{thm:Edixhoven}, $\psi$ acts trivially on $\{ 0 \} \times A_s \subset R_u$ and hence induces a nontrivial action on their quotient $E/G$. This $\psi$-action on $E/G$ satisfies $\psi^m(x) = \omega^{\pm1} x$. For example, let us consider the types $\mathrm{II}$ and $\mathrm{II}^*$. In this case, the automorphism $\psi$ acts on $E/G$ (with $j$-invariant $0$) such that its $m$-th power $\psi^m$ acts as a primitive $6$-th root of unity. This is impossible when $m$ is divisible by $2$ or $3$, so $m$ has to be coprime to $6$ and concludes (2). Same argument proves (1--4) with a single exception: when $C$ has type $\mathrm I_0^*$ and $j(E) = 1728$, we have $\Aut_0(E) = \mu_4$ whereas $\psi^m$ acts on $E/G$ by $-1$. Therefore, $m$ may be even in this case only when $j(E/G) = 1728$ so that $\Aut(E/G) = \mu_4$ (i.e., when $G = 0$, $E^{\sqrt{-1}}$, or $E[2]$). Even so, $m$ is still not divisible by $4$. This case will fall into the first three additional cases in \Cref{table:exceptional unstable-like fibers}.
			
			Same idea applies to when $h$ is strictly semistable. In this case, notice that $\psi$ fixes $A_s \subset R_u^\circ$ and induces an action on its quotient $\GG_m$ such that $\psi^m(z) = 1/z$. Since this is the only nontrivial group automorphism of $\GG_m$, $m$ cannot be even.
		\end{proof}
		
		Let us rephrase the above proof in the following more precise way:
		
		\begin{corollary} \label{cor:multiple fiber v:d and m are coprime}
			$d$ and $m$ are coprime except for the following single case: $C$ has type $\mathrm I_0^*$, $j(E) = 1728$, $d = 2$, and $m \equiv 2 \mod 4$. Moreover, in this case $\tau : G \hookrightarrow E$ has the image $0$, $E^{\sqrt{-1}}$, or $E[2]$. \qed
		\end{corollary}
		
		The next step is to imitate the untangling process in \Cref{lem:multiple fiber iv:untangle of X}. Take the diagram in \Cref{lem:untangle primitive} and \Cref{lem:multiple fiber v:albanese} to form a cartesian diagram
		\[\begin{tikzcd}
			C \times A_s \arrow[r] \arrow[d] & Y_t \arrow[r, "q"] \arrow[d] & \frac{1}{m} X_s \arrow[d] \\
			A_s \arrow[r, "G"] & \Alb_t \arrow[r, "\mu_m"] & \Alb_s
		\end{tikzcd}.\]
		The bottom row $A_s \to \Alb_s$ has a Galois group $H$ with a short exact sequence
		\begin{equation} \label{eq:multiple fiber v:ses of H}
			0 \to G \to H \to \mu_m \to 0 .
		\end{equation}
		
		\begin{lemma} \label{lem:multiple fiber v:untangle of X}
			\begin{enumerate}
				\item The sequence \eqref{eq:multiple fiber v:ses of H} splits if and only if there exists a Galois \'etale morphism $\tilde X \to X$ with Galois group $G$ that fills in the topright corner of \eqref{diag:multiple fiber v:untangle of Y} to a cartesian diagram
				\[\begin{tikzcd}
					\tilde Z \arrow[r, dashed] \arrow[d] & \tilde Y \arrow[r] \arrow[d] & \tilde X \arrow[d] \\
					Z \arrow[r, dashed] & Y \arrow[r, "q"] & X
				\end{tikzcd}.\]
				In this case, the $\mu_{dm} = \langle \varphi \rangle$-action on $Z$ (resp. $\langle \psi \rangle$-action on $R$) lifts to a $\mu_{dm} = \langle \tilde \varphi \rangle$-action on $\tilde Z$ (resp. $\langle \tilde \psi \rangle$-action on $\tilde R$) commuting with $G$.
				
				\item If $d$ and $m$ are coprime (see \Cref{cor:multiple fiber v:d and m are coprime}), then the sequence \eqref{eq:multiple fiber v:ses of H} uniquely splits.
			\end{enumerate}
		\end{lemma}
		\begin{proof}
			(1) Same as \Cref{lem:multiple fiber iv:untangle of X}. (2) Recall the inclusion $G \subset \pi_0(\tilde Q_t)$. By \Cref{thm:classification of component group} and \Cref{prop:semistable reduction theorem untangled ver}, the prime divisors of $|\pi_0(\tilde Q_t)|$ are always dominated by those of $d$ in every unstable Kodaira type. Therefore, $\gcd(d, m) = 1$ implies $\gcd (|G|, m) = 1$, whence the unique splitting $H = G \times \mu_m$.
		\end{proof}
		
		\Cref{lem:multiple fiber v:classification of G}--\ref{cor:multiple fiber v:classification of varphi} will study the case when $H$ splits. \Cref{lem:multiple fiber v:varphi-action nonsplit}--\ref{cor:multiple fiber v:when H doesnt split} will study the case when $H$ does not split.
		
		\begin{lemma} \label{lem:multiple fiber v:classification of G}
			If the sequence \eqref{eq:multiple fiber v:ses of H} splits, then the statement of \Cref{prop:untangle for unstable}(2) holds. That is, there exists a short exact sequence
			\[ 0 \to G \to \tilde P \to P \to 0 \quad\mbox{such that}\quad G \cap \tilde P^\circ = 0 .\]
			In particular, $0 \to G \to \pi_0 (\tilde P_s) \to \pi_0 (P_s) \to 0$.
		\end{lemma}
		\begin{proof}
			The left exact sequence $0 \to G \to \tilde P_s \to P_s$ is already established in \Cref{lem:multiple fiber v:untangle of X} (it uses \Cref{lem:untangle for multiple}), so it only remains to prove its surjectivity. Since $\tilde Q_t^\circ = Q_t^\circ = \GG_a \times A_s$, we have $\tilde P_s^\circ = P_s^\circ = \GG_a \times A_s$ and hence $G \cap \tilde P_s^\circ = 0$. Therefore, it is enough to show $\pi_0 (P_s) = \pi_0 (\tilde P_s/G)$. Consider the diagram
			\[\begin{tikzcd}
				0 \arrow[r] & G \arrow[r] \arrow[d, phantom, "\parallel"] & \tilde P_s \arrow[r] \arrow[d, "p^*"] & P_s \arrow[d, "p^*"] \\
				0 \arrow[r] & G \arrow[r] & \tilde Q_t \arrow[r] & Q_t \arrow[r] & 0
			\end{tikzcd},\]
			where $p^*$ are pullbacks of t-automorphisms. Their kernels are $0$ or $\GG_a$, and their images are the inertia-invariant subgroups by \Cref{thm:Edixhoven}. We thus have isomorphisms $\pi_0(\tilde P_s) = \pi_0(\tilde Q_t^{\tilde \psi})$ and $\pi_0(P_s) = \pi_0(Q_t^\psi)$. Note that the $\mu_m = \langle \tilde \psi \rangle$- and $G$-actions on $\tilde Q_t$ commute, so $G \subset \tilde Q_t^{\tilde \psi}$. This shows
			\[ \pi_0 (P_s) = \pi_0 (Q_t^\psi) = \pi_0 ((\tilde Q_t/G)^{\tilde \psi}) = \pi_0 (\tilde Q_t^{\tilde \psi}/G) = \pi_0 (\tilde P_s/G) .\qedhere\]
		\end{proof}
		
		\begin{lemma} \label{lem:multiple fiber v:varphi-action smooth}
			Assume that the sequence \eqref{eq:multiple fiber v:ses of H} splits and $h$ is smooth. Then there exists an isomorphism $\tilde Z_u \cong E \times A_s$ such that the inertia action is
			\[ \tilde \varphi(x,y) = (\omega^{\pm 1} x, y + a), \quad (\omega^{\pm 1/2} x, y + a), \ \ \mbox{or} \ \ (\omega^{\pm 1/3} x, y + a) \quad\mbox{for}\quad (x,y) \in E \times A_s .\]
			Here $\omega$ is a primitive $d$-th root of unity as in \Cref{prop:semistable reduction theorem untangled ver} and $a \in A_s$ is a torsion point of order $m$. The latter two only arise in the following four cases:
			\begin{enumerate}
				\item $\tilde Y$ has type $\mathrm I_0^*$, $d = 2$, $j(E) = 1728$, and $m \equiv 2 \mod 4$.
				\item $\tilde Y$ has type $\mathrm I_0^*$, $d = 2$, $j(E) = 0$, and $m \equiv 3 \mod 6$.
				\item $\tilde Y$ has type $\mathrm{IV}$, $d = 3$, $j(E) = 0$, and $m \equiv \pm2 \mod 6$.
				\item $\tilde Y$ has type $\mathrm{IV}^*$, $d = 3$, $j(E) = 0$, and $m \equiv \pm2 \mod 6$.
			\end{enumerate}
		\end{lemma}
		\begin{proof}
			Let us first show that $\tilde \varphi$ can be made into a diagonal automorphism. This is similar to the proof of \Cref{prop:semistable reduction theorem untangled ver}, but keep in mind that \Cref{lem:inertia action is faithful}(2) is false in this case (so $\tilde \varphi$ cannot be identified with $\tilde \psi$). Use \Cref{prop:semistable reduction theorem untangled ver}(1) to fix an isomorphism $\tilde R_u \cong E \times A_s$ for an elliptic curve $E$ such that the $\mu_d = \langle \tilde \psi^m \rangle$-action is $\tilde \psi^m(x,y) = (\omega^{\pm1} x, y)$ for $(x,y) \in E \times A_s$. The first part of the proof of \Cref{lem:semistable reduction theorem:varphi} still applies to $\tilde \psi$, so we have
			\[ \tilde \psi (x,y) = (\alpha x, y + \xi(x)) \qquad\mbox{for}\quad \alpha \in \Aut_0(E), \quad \xi : E \to A_s \mbox{ homomorphism} .\]
			Comparing $\tilde \psi^m$, we have $\alpha^m = \omega^{\pm1}$ and $\xi ((1 + \alpha + \cdots + \alpha^{m-1})x) = 0$ for every $x \in E$. The latter means either $\xi \equiv 0$ or $1 + \alpha + \cdots + \alpha^{m-1} = 0$.
			
			\medskip
			
			We claim $\alpha = \omega^{\pm1}$, $\omega^{\pm1/2}$, or $\omega^{\pm1/3}$, and $\alpha^m$ is never $1$. This will show $\xi \equiv 0$ and hence $\tilde \psi (x,y) = (\alpha x, y)$. In type $\mathrm{II}$, $\mathrm{II}^*$, $\mathrm{III}$, $\mathrm{III}^*$, or $\mathrm I_0^*$ with $j(E) \neq 0,1728$, we have $\Aut_0(E) = \mu_d$. Then $\alpha^m = \omega^{\pm1}$ is a generator of $\Aut_0(E)$, so this forces $\alpha = \omega^{\pm1}$ (and $\gcd (d,m) = 1$ in \Cref{cor:multiple fiber v:d and m are coprime}). This also means $\alpha^m = \omega^{\pm m} \neq 1$. In the remaining types $\mathrm{IV}$, $\mathrm{IV}^*$, or $\mathrm I_0^*$ with $j(E) = 0$ or $1728$, we have $\Aut_0(E) \supsetneq \mu_d$. For example, if $\tilde Y$ has type $\mathrm{IV}$ or $\mathrm{IV}^*$, then $d = 3$ and $\Aut_0(E) = \mu_6$. Hence $\alpha^m = \omega^{\pm1}$ is a primitive third root of unity, concluding $3 \nmid m$. When $m$ is odd, $\alpha = \omega^{\pm1}$ is the only possibility. When $m$ is even, $\alpha$ can be either $\omega^{\pm1}$ or $\omega^{\pm1/2}$. In any case, $\alpha^m \neq 1$ because $3 \nmid m$. Same holds for type $\mathrm I_0^*$ with $j(E) = 0$ or $1728$.
			
			\medskip
			
			Now that we know $\tilde \psi(x,y) = (\alpha x, y)$ for the desired $\alpha$, \Cref{prop:varphi-action} says
			\[ \tilde \varphi (x,y) = \tilde \psi(x,y) + \tilde \varphi(0,0) = (\alpha x, y) + \tilde \varphi(0,0) .\]
			Since $\alpha \neq 1$, we can adjust the origin of the torsor $\tilde Z_u \cong E \times A_s$ to make $\tilde \varphi(0,0) = (0, a)$ and hence $\tilde \varphi(x,y) = (\alpha x, y+a)$. Since $\tilde \varphi^m (x,y) = (\omega^{\pm1} x, y)$ by \Cref{prop:semistable reduction theorem untangled ver}(1), $a \in A_s$ is $m$-torsion. It has order precisely $m$ because the induced $\mu_m = \langle \tilde \varphi \rangle$-action on $\tilde Y_t = C \times A_s$ is free.
		\end{proof}
		
		Recall that $\tilde Q$ is untangled by definition. The following computes the Kodaira type of $\tilde P$. Recall from \Cref{def:isogeny and conjugate of Kodaira type} the conjugate of a Kodaira type.
		
		\begin{lemma} \label{lem:multiple fiber v:Kodaira type of P}
			Keep the notation and assumption of \Cref{lem:multiple fiber v:varphi-action smooth}. In the first case of $\tilde \varphi$, the Kodaira type of $\tilde P$ is that of $\tilde Q$ or its conjugate. In the four cases (1--4), the following holds.
			\begin{enumerate}
				\item $\tilde Q$ has type $\mathrm I_0^*$ and $\tilde P$ has type $\mathrm{III}$ or $\mathrm{III}^*$.
				\item $\tilde Q$ has type $\mathrm I_0^*$, and $\tilde P$ has type $\mathrm{II}$ or $\mathrm{II}^*$.
				\item $\tilde Q$ has type $\mathrm{IV}$, and $\tilde P$ has type $\mathrm{II}$ or $\mathrm{II}^*$.
				\item $\tilde Q$ has type $\mathrm{IV}^*$, and $\tilde P$ has type $\mathrm{II}$ or $\mathrm{II}^*$.
			\end{enumerate}
		\end{lemma}
		\begin{proof}
			In \Cref{lem:multiple fiber v:varphi-action smooth}, we have deduced
			\[ \tilde \psi^m (x,y) = (\omega^{\pm1} x, y) ,\qquad \tilde \psi (x,y) = (\alpha x, y) \quad\mbox{for}\quad \alpha = \omega^{\pm1}, \omega^{\pm1/2}, \mbox{ or } \omega^{\pm1/3} .\]
			\Cref{cor:semistable reduction theorem} computes the type of $\tilde P$ from that of $\tilde R$.
		\end{proof}
		
		\begin{lemma} \label{lem:multiple fiber v:varphi-action strictly semistable}
			Assume that the sequence \eqref{eq:multiple fiber v:ses of H} splits and $h$ is strictly semistable (so $d = 2$ and $m$ is odd). Then
			\begin{enumerate}
				\item $\tilde P$ has type $\mathrm I_r^*$ for $r \ge 1$, $\tilde Q$ has type $\mathrm I_{mr}^*$, and $\tilde R$ has type $\mathrm I_{2mr}$.
				\item $\tilde Z_u = D \times A_s$ for a type $\mathrm I_{2mr}$ curve $D$ and the inertia action is
				\[ \tilde \varphi(i,z,y) = (-i, \tfrac{\zeta^i}{z}, y + a) \qquad\mbox{for}\quad (i, z) \in D, \ y \in A_s ,\]
				where $\zeta \in \GG_m$ is a primitive $2m$-th root of unity and $a \in A_s$ has order $m$.
			\end{enumerate}
		\end{lemma}
		
		Compare this statement with the type $\mathrm I^+_{kr}$ in \Cref{thm:classification of multiple fiber iv}: the only difference in their inertia actions is the order of $a \in A_s$.
		
		\begin{proof}
			(1) This is again \Cref{cor:semistable reduction theorem}, noticing that the central fiber $\tilde Z_u = D \times A_s$ is trivial as a base change of $\tilde Y$ by \Cref{prop:semistable reduction theorem untangled ver}(2). (2) By the same proposition, $\tilde \psi(i,z,y) = (-i, \zeta^i z^{-1}, y)$ for a primitive $2m$-th root of unity $\zeta$. Writing $\varphi (0,1,0) = (i_0, \varepsilon, a)$, we have as usual
			\[ \tilde \varphi(i,z,y) = \tilde \psi(i,z,y) \cdot \tilde \varphi (0,1,0) = (i_0 - i, \varepsilon \zeta^i z^{-1}, y + a) \quad\mbox{for}\quad (i,z,y) \in D \times A_s .\]
			Now use $\tilde \varphi^m (i,z,y) = (-i, (-1)^i z^{-1}, y)$ from the same proposition (and $m$ is odd) to conclude that $i_0 = 0$, $\varepsilon = 1$, and $ma = 0$. The freeness of $\tilde \varphi$ says $\ord a = m$.
		\end{proof}
		
		The following descends \Cref{lem:multiple fiber v:varphi-action smooth}--\ref{lem:multiple fiber v:varphi-action strictly semistable} to $\tilde Y$.
		
		\begin{corollary} \label{cor:multiple fiber v:classification of varphi}
			Assume \eqref{eq:multiple fiber v:ses of H} splits. Then the induced $\mu_m = \langle \tilde \varphi \rangle$-action on $\tilde Y_t = C \times A_s$ is diagonal, acting on $A_s$ as a translation by a torsion point $a \in A_s$ of order $m$, and acting on $C$ as follows.
			\begin{enumerate}
				\item When $h$ is smooth, $\tilde \varphi$ acts on $C$ either trivially or as the cases $\mathrm I_0^*$-a, $\mathrm I_0^*$-c, $\mathrm{IV}$-a, and $\mathrm{IV}^*$-a in \Cref{thm:classification of multiple fiber v} and \Cref{table:exceptional unstable-like fibers}.
				\item Suppose $h$ is strictly semistable and write the Kodaira type $\mathrm I_{mr}^*$ curve $C$ by $E_0 + E_1 + 2 \big( \sum_{j=0}^{mr} F_j \big) + E_2 + E_3$. Then $\tilde \varphi$ acts on $E_i$ trivially and on $F_j$ by $\tilde \varphi (z) = (-\zeta)^j z$ for $z \in F_j = \PP^1$.
			\end{enumerate}
		\end{corollary}
		\begin{proof}
			(1) Recall from \Cref{prop:P-action on bar X} that $\tilde Z / \tilde \varphi^m$ and $\tilde Y$ are $\tilde Q$-equivariantly birational over $T$. \Cref{lem:multiple fiber v:varphi-action smooth} classified the inertia action on $\tilde Z_u$. In the non-exceptional case, $\tilde \varphi (x,y) = (\omega^{\pm1} x, y+a)$ and $\tilde \varphi^m(x,y) = (\omega^{\pm 1}x ,y)$ on $\tilde Z_u = E \times A_s$. Hence the induced $\mu_m = \langle \tilde \varphi \rangle$ on $\tilde Z_u / \tilde \varphi^m = \bar C \times A_s$ is a translation by $a$ on $A_s$, so it acts on $\tilde Y_t = C \times A_s$ the same way.
			
			\medskip
			
			For the exceptional cases, let us only consider type $\mathrm I_0^*$, $j(E) = 1728$, and $m \equiv 2 \mod 4$ (others are similar). The four ramification points of $E \to F = E/\pm1 = \PP^1$ are $E[2]$. The action $\tilde \varphi = \pm \sqrt{-1}$ on $E$ fixes two of them and swaps the other two. Hence, the induced $\tilde \varphi$ on $F$ fixes two branch points and swaps the other two. The four branch points are $E_i \cap F$ in $\tilde Y_t = C \times A_s$ for $C = \sum_{i=1}^4 E_i + 2F$. We also note that $\tilde \varphi$ acts on the infinitesimal direction of $2F \subset C$ by $\pm \sqrt{-1}$, so this determines the automorphism $\tilde \varphi$ on $C$ uniquely up to conjugate.
			
			\medskip
			
			(2) \Cref{lem:multiple fiber v:varphi-action strictly semistable} describes the $\langle \tilde \varphi \rangle$-action on $\tilde Z_u = D \times A_s$ for $D = \sum_{i=1}^{2mr} D_i$. The action $\tilde \varphi^m(i,z) = (-i, (-1)^i z^{-1})$ induces an isomorphism $\tilde \varphi^m : D_i \to D_{-i}$. Since $\tilde \varphi(i,z) = (-i, \zeta^i z^{-1}) = (0, (-\zeta)^i) \cdot (-i, (-1)^i z^{-1})$, we conclude that the induced $\mu_m = \langle \tilde \varphi \rangle$-action on $F_i \subset D / \tilde \varphi^m$ is given by multiplying $(-\zeta)^i$ for $i \neq 0, mr$. The two components $D_{2mr}$ and $D_{mr}$ are preserved by $\tilde \varphi$, and their quotients are the two ends of the string $F_0$ and $F_{mr}$, which are $\PP^1$ where $\tilde \varphi$ acts trivially.
		\end{proof}
		
		Let us now consider the case when \eqref{eq:multiple fiber v:ses of H} is non-split. By \Cref{cor:multiple fiber v:d and m are coprime} and \ref{lem:multiple fiber v:untangle of X}, $\tilde Y$ has type $\mathrm I_0^*$, $j(E) = 1728$, $d = 2$, and $m \equiv 2 \mod 4$. Let us write
		\[ m = 2k \quad\mbox{for an odd integer } k .\]
		Recall that $H$ is an extension of $\mu_m = \mu_{2k}$ by $G$. Since $H$ is non-split, we have two choices $G = E^{\sqrt{-1}}$ or $E[2]$. In any case, $G$ contains a unique $\sqrt{-1}$-fixed element $\tau \neq 0$. Note that we can write $\tau = (1+\sqrt{-1})g$ for any $g \in E[2] \setminus E^{\sqrt{-1}}$. The Albanese stabilizer $G \hookrightarrow E \times A_s$ induces the another inclusion $\sigma : G \hookrightarrow A_s$.
		
		\begin{lemma} \label{lem:multiple fiber v:varphi-action nonsplit}
			Assume \eqref{eq:multiple fiber v:ses of H} does not split. Then there exists an isomorphism $Z_u \cong (E \times A_s)/G$ such that its inertia action is
			\[ \varphi (x,y) = (\pm \sqrt{-1} x + g, y + a) \qquad\mbox{for}\quad (x,y) \in (E \times A_s)/G .\]
			Here $g \in E[2] \setminus E^{\sqrt{-1}}$ and $a \in A_s$ is a torsion point of order $4k$ such that $2ka = \sigma((1+\sqrt{-1})g)$.
		\end{lemma}
		\begin{proof}
			Let us first compute the inertia automorphism $\psi$ on $R_u$. Recall that $\tilde R_u = E \times A_s$ and $R_u = (E \times A_s)/G$, where $G = E^{\sqrt{-1}}$ or $E[2]$ and the $G$-action is given by translations by combining $G \subset E$ and $\sigma : G \hookrightarrow A_s$ (\Cref{prop:semistable reduction theorem tangled ver}(2)). Since $\psi$ fixes $\{ 0 \} \times A_s \subset R_u$, there exists an induced group automorphism $\psi$ on the quotient $R_u / (\{ 0 \} \times A_s) = E/G$, whose $j$-invariant is again $1728$. Hence $\psi$ on $E/G$ is a multiplication by $\alpha = -1$ or $\pm \sqrt{-1}$. In any case, it admits a unique lift to $E \times \{ 0 \} \subset R_u$ by the same multiplication. This shows that $\psi$ is of the form
			\[ \psi (x,y) = (\alpha x, y + \xi(x)) \qquad\mbox{for}\quad (x,y) \in R_u = (E \times A_s)/G ,\]
			where $\xi : E \to A_s$ is a homomorphism. Since we already know $\psi^{2k} (x,y) = (-x,y)$ by \Cref{prop:semistable reduction theorem untangled ver}(1), $\alpha^{2k} = -1$ and $\xi((1+\alpha + \cdots + \alpha^{2k-1}) x) = 0$ for every $x \in E$. Recalling $k$ is odd, we obtain $\alpha = \pm \sqrt{-1}$ and $\xi \equiv 0$. This shows $\psi (x,y) = (\pm \sqrt{-1} x, y)$ for $(x,y) \in R_u = (E \times A_s)/G$.
			
			\medskip
			
			Writing $\varphi(0,0) = (g, a) \in Z_u = (E \times A_s)/G$, we have $\varphi (x,y) = (\pm \sqrt{-1} x + g, y + a)$. Since $\varphi$ is free and has order $4k$, $a \in A_s$ has order $4k$. Since $\varphi^{2k}(x,y) = (-x,y)$, we have $((1+\sqrt{-1})g, 2ka) \in G$. When $G = E^{\sqrt{-1}}$, this makes $g \in E[2] \setminus E^{\sqrt{-1}}$ as claimed. When $G = E[2]$, there are more freedom of $g$, but we can choose $g$ of the desired form.
		\end{proof}
		
		\begin{lemma} \label{lem:multiple fiber v:lift of varphi}
			Even when \eqref{eq:multiple fiber v:ses of H} does not split, there exists a preferred free $\mu_{4k} = \langle \tilde \varphi \rangle$-action on $\tilde Z$ lifting the $\langle \varphi \rangle$-action on $Z$. This $\tilde \varphi$ commutes with $\tau \in G$ and $\tilde Z \dashrightarrow \tilde Y$ is given by the $(\tau \circ \tilde \varphi^{2k})$-quotient.
		\end{lemma}
		\begin{proof}
			Simply take the automorphism $\tilde \varphi (x,y) = (\pm \sqrt{-1} x + g, y + a)$ of $\tilde Z_u = E \times A_s$, the same formula as \Cref{lem:multiple fiber v:varphi-action nonsplit}. Computation shows such $\tilde \varphi$ has an order $4k$ and commutes with $\tau = (1+\sqrt{-1})g$. The extension of $\tilde \varphi$ to an automorphism of $\tilde Z$ is formal as in \Cref{lem:multiple fiber ii:untangle of X} (use \Cref{lem:lifting morphism of abelian torsors}). Finally, $\tau \circ \tilde \varphi^{2k}(x,y) = \tau(-x + (1 + \sqrt{-1})g, y + 2ka) = (-x,y)$ is the automorphism describing the quotient $\tilde Z \to \tilde Z / \mu_2$ whose crepant resolution is $\tilde Y$.
		\end{proof}
		
		\begin{corollary} \label{cor:multiple fiber v:when H doesnt split}
			If \eqref{eq:multiple fiber v:ses of H} does not split, then $G = E^{\sqrt{-1}} = \ZZ/2$. Moreover,
			\begin{enumerate}
				\item $Q$ has type $\mathrm I_0^*/2$ and $P$ has type $\mathrm{III}/2$ or $\mathrm{III}^*/2$.
				\item $H \cong \ZZ/4k = \langle \tilde \varphi \rangle$ acts on $\tilde Y_t = C \times A_s$ as the case $\mathrm I_0^*$-b in \Cref{table:exceptional unstable-like fibers}.
			\end{enumerate}
		\end{corollary}
		\begin{proof}
			Take the automorphism $\tilde \varphi$ on $\tilde Z$ in \Cref{lem:multiple fiber v:lift of varphi}. Noticing $\tilde \varphi$ and $\tau$ commute, take the free quotients $\bar Z = \tilde Z / \tau$ and $\bar Y = \tilde Y / \tau$. Then $\tilde \varphi$ descends to the $\mu_{4k} = \langle \tilde \varphi \rangle$-action on $\bar Z$, and $\bar Z \dashrightarrow \bar Y$ is a rational quotient by $\tilde \varphi^{2k}$. Taking the further quotient $\bar X = \bar Y / \tilde \varphi$, we obtain the following commutative diagram:
			\[\begin{tikzcd}[column sep=normal]
				\tilde Z \arrow[r, dashed, "/\tau \tilde \varphi^{2k}"] \arrow[d, "/\tau"] & \tilde Y \arrow[d, "/\tau"] \\
				\bar Z \arrow[r, dashed, "/\tilde \varphi^{2k}"] \arrow[d] & \bar Y \arrow[r, "/\tilde \varphi"] \arrow[d] & \bar X \arrow[d] \\
				Z \arrow[r, dashed, "/\varphi^{2k}"] & Y \arrow[r, "/\varphi"] & X
			\end{tikzcd}.\]
			The lower half of the diagram plays the same role with \Cref{lem:multiple fiber ii:untangle of X}. Therefore, taking $G' = G / \langle \tau \rangle = 0$ or $\ZZ/2$, we have a left exact sequence $0 \to G' \to \bar P \to P$ with $G' \cap \bar P^\circ = 0$. This shows the inclusion $G' \subset \pi_0 (\bar P_s)$.
			
			\medskip
			
			(1) Let us now study the upper half of the diagram, or equivalently assume $G = E^{\sqrt{-1}} = \langle \tau \rangle$. By \Cref{lem:multiple fiber v:lift of varphi}, we have a $\langle \tilde \varphi \rangle \times \langle \tau \rangle$-action on $\tilde Z$. Such action descends to the $\mu_{4k} = \langle \tilde \varphi \rangle$-action on $\tilde Y$, which describes the $H$-action. \Cref{prop:semistable reduction theorem tangled ver} shows that $Q$ has type $\mathrm I_0^*/2$, and $P$ is isogenous to type $\mathrm{III}$ or $\mathrm{III}^*$. By \Cref{lem:multiple fiber v:varphi-action nonsplit}, $R_u^\psi = (E^{\sqrt{-1}} \times A_s)/\tau \cong A_s$ has a trivial $\pi_0$. By \Cref{thm:Edixhoven}, this shows $\pi_0(P_s) = 0$ as claimed. Now that we know $\pi_0 (\bar P_s) = 0$, this forces $G' = 0$. That is, the case $G = E[2]$ is impossible.
			
			\medskip
			
			(2) By definition of $\tilde \varphi(x) = \pm \sqrt{-1} x + g$ for $x \in E$, it cycles $\tilde \varphi : 0 \mapsto g \mapsto \tau = (1+\sqrt{-1})g \mapsto \sqrt{-1}g \mapsto 0$. These are the four ramification points of $E \to F = E / \pm1 = \PP^1$. Hence $\tilde \varphi$ descends to an automorphism on $F$ that cycles four branch points, which are $E_i \cap F$ in the type $\mathrm I_0^*$ curve $C = \sum_{i=1}^4 E_i + 2F$.
		\end{proof}
		
		\begin{proof} [Proof of \Cref{thm:classification of multiple fiber v}]
			Suppose that the sequence \eqref{eq:multiple fiber v:ses of H} splits. In this case, we have an isomorphism $\tfrac{1}{m} X_s = \tfrac{1}{m} \tilde X_s / G = \tilde Y_t / (\langle \tilde \varphi \rangle \times G) = (C \times A_s) / (\langle \tilde \varphi \rangle \times G)$ and the $\langle \tilde \varphi \rangle$-action on $\tilde Y_t$ is classified in \Cref{cor:multiple fiber v:classification of varphi}. Note that the Albanese stabilizer $G$ is a subgroup of $\pi_0(\tilde P_s)$ by \Cref{lem:multiple fiber v:classification of G}, which is again classified in \Cref{lem:multiple fiber v:Kodaira type of P} and \Cref{lem:multiple fiber v:varphi-action strictly semistable}(1).
			
			If $h$ is smooth and $\tilde \varphi$ is not exceptional, then $\tilde \varphi$ acts trivially on the $C$ factor, so $\frac{1}{m} \tilde X_s = (C \times A_s/a) / G$ is isomorphic to the unstable fibers in \Cref{thm:classification of unstable fibers}. If $h$ is smooth and $\tilde \varphi$ is exceptional, then we obtain the types in \Cref{table:exceptional unstable-like fibers} except $\mathrm I_0^*$-b. If $h$ is strictly semistable, this is the description of the type $\mathrm I_R^*$ for $R = mr$.
			
			\medskip
			
			Even when the sequence \eqref{eq:multiple fiber v:ses of H} does not split, we still have an isomorphism $\tfrac{1}{m} X_s = (C \times A_s) / H$. The $H$-action on $\tilde Y_t$ is described in \Cref{cor:multiple fiber v:when H doesnt split}. This is the remaining case $\mathrm I_0^*$-b.
		\end{proof}

\section{Examples} \label{sec:examples}
	Let us realize all central fibers $X_s$ classified in \Cref{thm:classification of semistable fibers}, \ref{thm:classification of unstable fibers}, and those in \S \ref{sec:classification of multiple fibers}. Throughout, we assume that $E$ is an elliptic curve, $A$ is an abelian variety of dimension $n-1$, and $\omega, \zeta \in \GG_m$ are roots of unity.
	
	\begin{example}
		Let $X \to S$ be a minimal elliptic fibration and $\mathfrak A \to S$ an abelian scheme of relative dimension $n-1$. Then $X \times_S \mathfrak A \to S$ is a minimal $\delta$-regular abelian fibration with a central fiber $X_s \times \mathfrak A_s = X_s \times A$.
	\end{example}
	
	\begin{example} [Semistable fibrations with torsion shears] \label{ex:semistable with torsion shear}
		Let $f : X \to S$ be any minimal $\delta$-regular abelian fibration with a t-automorphism scheme $P \to S$. Let $G \subset P_s$ be a finite subgroup that \emph{acts freely on $X_s$}, i.e., $G$ has a trivial intersection with every possible stabilizer of $X_s$. Shrinking $S$, we may extend it to torsion sections $G \subset P$. The quotient $\bar f : \bar X = X/G \to S$ is a minimal $\delta$-regular abelian fibration with a central fiber $\bar X_s = X_s / G$.
		
		If $f$ is as in the previous example, then $X_s = C \times A$ and $\bar X_s = (C \times A)/G$. Suppose that $C$ is a type $\mathrm I_{er}$ semistable curve for positive integers $e$ and $r$. Then $P_s \cong \ZZ/er \times \GG_m \times A$ and the only nontrivial stabilizer is $\{ 0 \} \times \GG_m \times \{ 0 \}$ by \Cref{lem:stabilizer of Ir}. Take any torsion element $(r,\zeta,a) \in P_s$ of order $e$ and $G \subset P_s$ a cyclic group generated by it. This yields a semistable $\bar X_s = (C \times A) / G$ of type $\mathrm I_r$ with a torsion shear $a \in A$ and Albanese torsor $\Alb_{\bar X_s} = A/a$.
	\end{example}
	
	\begin{example} [Unstable fibrations]
		In \Cref{ex:semistable with torsion shear}, suppose that $C$ is an unstable Kodaira curve. Then $P_s \cong \pi_0(P_s) \times \GG_a \times A_s$ and every stabilizer is contained in $\pi_0(P_s) \times \GG_a \times \{ 0 \}$. If $G$ is a finite abelian group with two inclusions $G \subset \pi_0(P_s)$ and $G \subset A_s$, then $G \subset P_s$ yields a minimal $\delta$-regular abelian fibration $\bar X = X / G \to S$ with Albanese stabilizer $G$.
	\end{example}
	
	\begin{example} [Semistable fibrations with non-torsion shears]
		Nakamura \cite{nak77} and Hwang--Oguiso \cite{hwang-ogu16} constructed explicit examples of minimal $\delta$-regular abelian fibrations $f : X \to S$ of type $\mathrm I_1$ with non-torsion shears. We refer to the original articles for their (analytic) constructions. If we take a degree $r$ totally ramified base change $T \to S$ and a crepant resolution $Y \to X_T \to T$, this constructs a type $\mathrm I_r$ fiber with non-torsion shears.
	\end{example}
	
	\begin{example} [Unstable isotrivial fibrations]
		Many of the examples in \cite[\S 6]{mat01}, \cite[\S 6]{hwang-ogu11}, \cite{kim-laza-mar26} are isotrivial. Let $f : X \to S$ be a minimal $\delta$-regular abelian fibration that is non-multiple and isotrivial (all smooth fibers are isomorphic to an abelian variety $B$). A standard level structure argument shows that there exists a finite base change $T \to S$ such that $f_T$ is birational to a trivial family $g : Y = B \times T \to T$. This means that $f$ is unstable (e.g., \Cref{thm:Edixhoven}) and moreover, it cannot be $\mathrm I_{\ge 1}^*$-type (\Cref{prop:semistable reduction theorem untangled ver}). This generalizes \cite[Proposition~4.20]{kim-laza-mar26}.
		
		Conversely, start with an elliptic curve $E$, abelian variety $A$ of dimension $n-1$, and trivial family $Y = (E \times A) \times \AA^1 \to \AA^1$. Endow it with a $\mu_d = \langle \varphi \rangle$-action by
		\[ \varphi (x,y,t) = (\omega^{\pm1} x, y, \omega t) \qquad\mbox{for}\quad (x,y,t) \in Y = E \times A \times \AA^1 \]
		described in \Cref{prop:semistable reduction theorem untangled ver}. Its quotient is a klt-minimal abelian fibration $\bar X = Y/\varphi \to \AA^1$ whose minimal model $X \to \AA^1$ can be any non-$\mathrm I_{\ge 1}^*$-type unstable fibration (with a trivial Albanese stabilizer). To produce examples with nontrivial Albanese stabilizers, use \Cref{ex:semistable with torsion shear}.
	\end{example}
	
	\begin{example} [Type $\mathrm I_0^*/4$]
		Given a variety $Y$, we denote by $Y^{[2]}$ and $Y^{(2)}$ their Hilbert and symmetric squares, respectively. Take a \emph{smooth} elliptic fibration $Y \to \AA^1$ and
		\[ Y^{[2]} \to Y^{(2)} \to (\AA^1)^{(2)} \cong \AA^2 .\]
		It is singular along the diagonal $\AA^1 \subset (\AA^1)^{(2)}$. Take a general trait $S$ on the base that intersects transversally with $\AA^1$. The restriction of $Y^{[2]}$ over $S$ is a minimal $\delta$-regular abelian fibration $f : X \to S$ whose central fiber $X_s$ is a union of two $\PP^1$-bundles over an elliptic curve $E$, one with multiplicity $1$ and the other one with multiplicity $2$. An interesting computation of $X_s$ is left to the reader, but even without it, $X_s$ must be type $\mathrm I_0^*/4$ according to our classification in \Cref{thm:classification of unstable fibers}.
	\end{example}
	
	The following examples construct multiple central fibers. Note that our classifications of multiple fibrations have explicit descriptions of the inertia action $\varphi$ (or $\tilde \varphi$), so constructing examples is easy.
	
	\begin{example} [Multiple fibers via nontrivial torsors]
		Let $S = \Spec R$ and $S_0 = \Spec K$ be the Spec of a strictly henselian dvr and its field of fractions. Let $P \to S$ be any $\delta$-regular N\'eron model of an abelian variety over $S_0$. Any $\alpha \in H^1_{\acute et} (S_0, P_0)$ defines a $P_0$-torsor $X_0 \to S_0$ and its unique minimal model $f : X \to S$. If $\alpha \neq 0$, then $f$ cannot have any sections, or the central fiber $X_s$ is a multiple fiber. The group $H^1_{\acute et} (S_0, P_0)$ is torsion, so suppose that $\alpha$ has order $m$. The multiple fiber $X_s$ then has a multiplicity $m$ since the degree $m$ totally ramified base change $T \to S$ kills $\alpha$. Though this interpretation was not used in our proofs, this is an important viewpoint to think about multiple fibrations.
	\end{example}
	
	\begin{example} [Type $m \cdot \mathrm I_0$] \label{ex:mI0}
		Let $B$ be an abelian variety of dimension $n$ and $B \times T \to T = \AA^1$ a trivial family. Take its free $\mu_m = \langle \varphi \rangle$-action given by
		\[ \varphi(x, t) = (x + b, \zeta t) \qquad\mbox{for}\quad (x,t) \in B \times T ,\]
		where both $b \in B$ and $\zeta \in \GG_m$ have order $m$. The quotient is a minimal fibration $f : X = (B \times T)/\varphi \to S$ with $X_s = m \cdot (B/b)$. See \Cref{thm:classification of multiple fiber i}.
	\end{example}
	
	\begin{example} [Type $m \cdot \mathrm I_0^+$]
		Take a free $\mu_m = \langle \varphi \rangle$-action on $E \times A \times T$ by
		\[ \varphi(x,y,t) = (\omega^{\pm1} x, y + a, \zeta t) \qquad\mbox{for}\quad (x,y,t) \in E \times A \times T .\]
		Here we choose $m = dk$ for $d = 2,3,4,$ or $6$, $\zeta \in \GG_m$ is a primitive $dk$-th root of unity, $\omega = \zeta^k \in \Aut_0(E)$ has order $d$, and $a \in A$ has order $dk$. The quotient has a central fiber $X_s = dk \cdot (E \times A)/\varphi$. See \Cref{thm:classification of multiple fiber ii}.
	\end{example}
	
	\begin{example} [Type $m \cdot \mathrm I^k_R$ for $m = kl$ and $R = kr$] \label{ex:mIr}
		Start with a minimal elliptic fibration $X \to S = \AA^1$ with a multiple central fiber $X_s$ of type $l \cdot \mathrm I_r$. Take a sequence of ramified coverings of $S$ and minimal base changes over them
		\[\begin{tikzcd}
			Y \arrow[r] \arrow[d] & X^+ \arrow[r] \arrow[d] & X \arrow[d] \\
			T \arrow[r, "\mu_k"] & S^+ \arrow[r, "\mu_l"] & S
		\end{tikzcd}.\]
		Then $X_s^+$ has type $\mathrm I_{lr}$ and $Y_t$ has type $\mathrm I_{klr}$. Moreover, $Y$ admits a $\mu_{kl} = \langle \varphi \rangle$-action. By \Cref{lem:classification of multiple fiber iii:psi-action}, the inertia action is of the form
		\[ \varphi (i,z) = (i + kr, \zeta^i z) \qquad\mbox{for}\quad (i,z) \in Y_t = \sum_{i=1}^{klr} C_i ,\]
		where $\zeta$ is a primitive $kl$-th root. Take an abelian variety $A$ of dimension $n-1$ and construct a $\mu_{kl}$-action on $Y \times A$ by
		\[ \varphi (x, y) = (\varphi(x), y + a) \qquad\mbox{for}\quad (x,y) \in Y \times A ,\]
		where $a \in A$ has order $kl$. The quotient $(Y \times A)/\mu_{kl} \to S$ has a central fiber of type $kl \cdot \mathrm I^k_{kr}$. The t-automorphism $P$ of the quotient is the $\psi$-invariant of the t-automorphism of $Y \times A$. See \Cref{thm:classification of multiple fiber iii}.
	\end{example}
	
	\begin{example} [Type $m \cdot \mathrm I^+_R$ for $m = 2k$ and $R = kr$] \label{ex:2kIkr+}
		Start with a minimal elliptic fibration $X \to S = \AA^1$ with a central fiber of type $\mathrm I_r^*$. Take a sequence of minimal elliptic fibrations
		\[\begin{tikzcd}
			Y \arrow[r] \arrow[d] & Z \arrow[r, dashed] \arrow[d] & X \arrow[d] \\
			T \arrow[r, "\mu_k"] & U \arrow[r, "\mu_2"] & S
		\end{tikzcd}.\]
		Then $Z_u$ has type $\mathrm I_{2r}$ and $Y_t$ has type $\mathrm I_{2kr}$. Moreover, $Y$ admits a $\mu_{2k} = \langle \varphi \rangle$-action, which is described on $Y_t$ by
		\[ \varphi (i,z) = (-i, \tfrac{\zeta^i}{z}) \qquad\mbox{for}\quad (i,z) \in Y_t ,\]
		where $\zeta$ is a primitive $2k$-th root. Now combine it with a translation automorphism by $a \in A$ of order $2k$ to form a diagonal $\mu_{2k}$-action on $Y \times A \to T$. Its quotient is $(Y \times A)/\varphi \to S$ with a central fiber of type $2k \cdot \mathrm I^+_{kr}$. See \Cref{thm:classification of multiple fiber iv}.
	\end{example}
	
	\begin{example} [Type $m \cdot \mathrm I^-_R$ for $m = 2k$ and $R = kr$]
		Take the same type $\mathrm I_{2kr}$ minimal elliptic fibration $Y \to T$ in \Cref{ex:2kIkr+}. It has a central t-automorphism group $\ZZ/2kr \times \GG_m$, so take any t-automorphism $\varepsilon_1$ of $Y$ that acts on $Y_t$ by $(1, \varepsilon) \in \ZZ/2kr \times \GG_m$. Then $\tilde \varphi = \varepsilon_1 \circ \varphi$ is an automorphism on $Y$ that acts on $Y_t$ by
		\[ \tilde \varphi (i,z) = (1-i, \tfrac{\varepsilon \zeta^i}{z}) \qquad\mbox{for}\quad (i,z) \in Y_t .\]
		One checks that it has order $4k$ on $Y$, such that $\tilde \varphi^{2k}$ acts trivially on $T$ and acts on $Y_t$ as a t-automorphism $\tilde \varphi^{2k}(i,z) = (i, -z)$. Use this to define a diagonal $\mu_{4k}$-action by $Y \times A$ (translation of order $4k$ on $A$), so that the quotient $(Y \times A)/\tilde \varphi \to S$ has a central fiber of type $2k \cdot \mathrm I^-_{kr}$. See \Cref{thm:classification of multiple fiber iv}.
	\end{example}
	
	\begin{example} [Type $m \cdot \mathrm I_R^*$ for odd $m = k$ and $R = kr$]
		Again start with $Y \to T$ in \Cref{ex:2kIkr+} of type $\mathrm I_{2kr}$, but assume furthermore that $k$ is odd. Endow $Y \times A$ with a diagonal $\mu_{2k}$-action
		\[ \varphi(x,y) = (\varphi(x), y+a) \]
		but with $a \in A$ of order $k$ (instead of the original $2k$). Then $\varphi^k$ acts as $\varphi^k (i,z,y) = (-i, \frac{(-1)^i}{z}, y)$ on the central fiber because $k$ is odd. This is the automorphism described in \Cref{prop:semistable reduction theorem untangled ver}(2), so $(Y \times A)/\varphi^k \to T/\varphi^k$ has a minimal model of type $\mathrm I_{kr}^*$. Its subsequent $\varphi$-quotient has type $k \cdot \mathrm I_{kr}^*$. See \Cref{lem:multiple fiber v:varphi-action strictly semistable}.
	\end{example}
	
	\begin{example} [Other unstable-like multiple fibers]
		Start with a trivial family $E \times A \times T \to T = \AA^1$. Take a group automorphism $\omega$ on $E$ of order $d = 2,3,4,$ or $6$, an integer $m$ coprime with $d$, and a primitive $dm$-th root $\zeta$ such that $\omega = \zeta^m$. Endow it with a $\mu_{dm}$-action by
		\[ \varphi (x,y,t) = (\omega^{\pm1} x, y + a, \zeta t) \qquad\mbox{for}\quad (x,y,t) \in E \times A \times T ,\]
		where $a \in A$ has order $m$. The quotient $(E \times A \times T)/\varphi \to S$ has an unstable-like multiple central fiber as described in \Cref{lem:multiple fiber v:varphi-action smooth}. Changing the first entry of $\varphi$ to $\omega^{\pm1/2} x$ or $\omega^{\pm1/3}x$ will yield four exceptional types (and their quotients). See \Cref{lem:multiple fiber v:varphi-action smooth}.
		
		To produce the most exotic type $m \cdot \mathrm I_0^*$-b, assume $j(E) = 1728$, $d = 2$, $m = 2k$ for odd $k$, and endow it with a $\mu_{4k} \times \ZZ/2$-action on $E \times A \times U$ by
		\[ \tilde \varphi(x,y,t) = (\pm \sqrt{-1} x + g, y + a, \zeta t) ,\qquad \tau(x,y,t) = (x + (1 + \sqrt{-1}) g, y + 2ka, t) ,\]
		where $g \in E[2] \setminus E^{\sqrt{-1}}$, $a \in A$ has order $4k$, and $\zeta$ is a primitive $4k$-th root. The quotient $(E \times A \times U)/\langle \tilde \varphi, \tau \rangle \to S = U / \mu_{4k}$ has a central fiber of type $2k \cdot \mathrm I_0^*$-b. See \Cref{lem:multiple fiber v:varphi-action nonsplit}.
	\end{example}

\appendix
\section{Auxiliary results} \label{sec:appendix}
	All discussions are over a fixed field $k$ of characteristic $0$. In many cases, we need to apply these results to non-reduced and reducible schemes $X$.
	
	\medskip
	
	We thank Takumi Murayama for explaining his results on the Koll\'ar injectivity and torsion-freeness theorem.
	
	\begin{lemma} \label{lem:Kollar-Kovacs}
		Let $S$ be a Dedekind scheme over $k$, and $f : X \to S$ a flat projective morphism of relative dimension $n$ where $X$ has rational singularities. Then $R^if_* \mathcal O_X$ is locally free for $0 \le i \le n$ and its formation commutes with base change.
	\end{lemma}
	\begin{proof}
		This is a consequence of \cite[Theorem~1]{kol-kov25}, where $S$ assumed to be a variety over $k$. We simply point out their arguments applies the same when $S$ is a Dedekind scheme, which we summarize as follows.
		
		\medskip
		
		Our claim is Zariski local on the base, so we may assume $S = \Spec R$ where $R$ is a Dedekind domain with residue field $k$. First, $R^if_* \omega_X$ is locally free because it is torsion-free by \cite[Theorem~8.6.1]{mur24} (or by \cite[Theorem~A(ii)]{mur25} with the injectivity $\Longrightarrow$ torsion-freeness argument). By (16) and (19) of \cite{kol-kov25}, both $R^if_* \omega_f$ and $R^if_* \mathcal O_X$ are locally free as well. To show the formation of higher direct images of $\mathcal O_X$ commutes with arbitrary base change, it is enough to prove a quasi-isomorphism $Rf_* \mathcal O_X = \bigoplus_{i=0}^n R^if_* \mathcal O_X [-i]$ by the cohomology base change for quasi-coherent sheaves \cite[Tag~08IB]{stacks-project}.
		
		Again by (16) in op. cit., we may prove $Rf_* \omega_f = \bigoplus_{i=0}^n R^if_* \omega_f [-i]$ instead. For it, take a sufficiently general and positive divisor $H \subset X$, the composition $g : H \subset X \to S$, and apply the induction argument of Theorem~1 in op. cit and \cite[Theorem~3.1]{kol86II}. The only part that needs to be checked is the splitting of the surjective homomorphism $\alpha : g_* \omega_g \to R^1f_* \omega_f$. Since $S$ is affine and $R^1f_* \omega_f$ is locally free, we have $\Ext^1_S (R^1f_* \omega_f, \ker \alpha) = 0$. This completes the proof.
	\end{proof}

	\subsection{Algebraic groups and actions}
		An \emph{isogeny} between two commutative connected algebraic groups $G$ and $H$ is a quasi-finite surjective homomorphism $f : G \to H$. An \emph{isogeny} between two abelian torsors is a finite \'etale morphism between them, which is necessarily Galois.
		
		\begin{lemma} \label{lem:isogeny}
			Let $G$ and $H$ be connected commutative algebraic groups. Then an isogeny $f : G \to H$ induces isogenies of their linear and abelian variety parts. In particular, the linear parts of $G$ and $H$ are isomorphic.
		\end{lemma}
		\begin{proof}
			The homomorphism $f$ induces a homomorphism between Chevalley exact sequences
			\[\begin{tikzcd}
				& \ker g \arrow[d] & \ker f \arrow[d] & \ker h \arrow[d] \\
				0 \arrow[r] & L_G \arrow[r] \arrow[d, "g"] & G \arrow[r] \arrow[d, "f"] & A_G \arrow[r] \arrow[d, "h"] & 0 \\
				0 \arrow[r] & L_H \arrow[r] \arrow[d] & H \arrow[r] \arrow[d] & A_H \arrow[r] \arrow[d] & 0 \\
				& 0 & 0 & 0
			\end{tikzcd}.\]
			The homomorphism $h$ is surjective by the Snake lemma. The neutral component of $\ker h$ is an abelian variety, so it does not admit a surjection onto a connected linear algebraic group. This forces $\coker g$ to be finite, or equivalently $g$ is surjective. Hence both $g$ and $h$ are isogenies. For the second statement, notice that every connected and commutative linear algebraic group over $k$ is of the form $\GG_m^{\times a} \times \GG_a^{\times b}$. If the two are isogenous, then they are isomorphic (not by the isogeny).
		\end{proof}
		
		A \emph{locally algebraic group} is a locally finite type group scheme $G$ over $k$. An \emph{algebraic group} is a finite type group scheme $G$ over $k$, or equivalently a locally algebraic group with finite $\pi_0(G)$.
		
		\begin{lemma} \label{lem:pi_0 sequence splits}
			Let $G$ be a commutative locally algebraic group.
			\begin{enumerate}
				\item If $\pi_0(G)$ is finitely generated, then $0 \to G^{\circ} \to G \to \pi_0(G) \to 0$ splits.
				\item If $\pi_0(G)$ is finite and $G^{\circ}$ is unipotent, then the splitting is unique.
			\end{enumerate}
		\end{lemma}
		\begin{proof}
			Fix an isomorphism $\pi_0(G) = \bigoplus_{i=1}^s \ZZ/d_i \oplus \bigoplus_{j=1}^r \ZZ$. For each $i$, take a connected component $G^i \subset G$ corresponding to a generator $1 \in \ZZ/d_i \subset \pi_0(G)$. The $d_i$-th power morphism $[d_i] : G^i \to G^{\circ}$ is surjective by \cite[Lemma~7.3.1]{neron}, so take any point $g_i \in G^i$ such that $d_i \cdot g_i = 0$. For each $j$, take any point $g_j \in G$ in the component corresponding to a generator of the $j$-th $\ZZ$. Then the subgroup of $G$ generated by all $g_i$ and $g_j$ is isomorphic to $\pi_0(G)$ and contains precisely one element in each connected component of $G$. This proves the $\pi_0$-sequence splits. The splitting is unique iff $\Hom(\pi_0(G), G^{\circ}) = 0$.
		\end{proof}
		
		\begin{lemma} \label{lem:lifting morphism of abelian torsors}
			Let $A$, $B$, and $\tilde B$ be abelian torsors over a scheme $S$, $f : A \to B$ a morphism of abelian torsors, and $\tilde B \to B$ an isogeny. Assume that $f_s : A_s \to B_s$ over a closed point $s \in S$ lifts to a morphism $\tilde f_s : A_s \to \tilde B_s$. Then shrinking $(S, s)$ \'etale locally, there exists a unique lift of $f$ to a morphism of abelian torsors $\tilde f : A \to \tilde B$ whose restriction over $s$ is $\tilde f_s$.
		\end{lemma}
		\begin{proof}
			We may assume $S = \Spec \mathcal O_{S, s}^{sh}$ is strictly henselian. Then all abelian torsors are trivial, so we may assume they are abelian schemes and $f$ is a homomorphism. Take a short exact sequence $0 \to G \to \tilde B \to B \to 0$ of fppf abelian sheaves over $S$ (here $G$ is a constant finite group scheme) and its long exact sequence
			\[ \Hom (A, G) = 0 \to \Hom (A, \tilde B) \to \Hom (A, B) \to \Ext^1 (A, G) .\]
			The last term is the global section of the sheaf-Ext $\SheafExt^1 (A, G)$ since $S$ is strictly henselian \cite[Tag~03QO]{stacks-project}. Noticing $\SheafExt^1 (A, G) = \SheafHom (\hat G, \SheafExt^1 (A, \GG_m)) = \SheafHom (\hat G, \check A)$, the last term is $\Hom (\hat G, \check A)$. A torsion section of $\check A$ is trivial if and only if its restriction to $s$ is trivial.
		\end{proof}
		
		\begin{lemma} \label{lem:quotient action}
			Let $G$ be an algebraic group acting on a quasi-projective scheme $X$ over $k$. If $K \subset Z_G$ is a finite subgroup of the center of $G$, then there exists an induced $G/K$-action on $X/K$.
		\end{lemma}
		\begin{proof}
			Endow the action morphism $G \times X \to X$ with an equivariant $K^{\times 2}$-action by $(k,k') \cdot (g,x) = (kg, k' \cdot x)$ and $(k,k') \cdot x = (kk') \cdot x$. Its scheme-theoretic quotient yields a morphism $G/K \times X/K \to X/K$. The group action axioms are easily checked.
		\end{proof}
		
		The following is a slight generalization of \cite[Lemma~5.16]{ari-fed16}.
		
		\begin{lemma} \label{lem:linear stabilizer}
			Let $G$ be an algebraic group acting on a proper connected scheme $X$ over $k$. If the kernel of the action is affine, then then the stabilizer group $\St_x$ is affine for every point $x \in X$.
		\end{lemma}
		\begin{proof}
			We may assume $G$ is connected. Let us first prove the claim when $X$ is reduced. Writing $X = \bigcup_i X_i$ for its irreducible decomposition, $G$ acts on each $X_i$. Since $X_i$ is integral, the proof of \cite[Lemma~5.16]{ari-fed16} shows that every $x \in X_i$ has an affine stabilizer if and only if $\ker (G \to \Aut_{X_i})$ is affine. Assume on the contrary that there exists an irreducible component $X_i$ with non-affine kernel, and let $H \subset G$ be the (nontrivial) \emph{anti-affine part} of the kernel. Then $H$ acts trivially on the entire $X_i$. By connectedness of $X$, there exists another component $X_j$ such that $H$ fixes a point $x \in X_i \cap X_j$. By \cite[Lemma~1.1]{brion09}, $H$ acts on the entire $X_j$ trivially. Induction (and connectedness of $X$) concludes that $H$ acts on $X$ trivially, or $H \subset \ker (G \to \Aut_X)$. But the kernel was assumed to be affine, so it cannot contain a nontrivial anti-affine subgroup.
			
			\medskip
			
			In a general case, we finish the proof by showing the induced $G$-action on $X_{\red}$ has an affine kernel. This will imply the claim because the problem of considering stabilizer groups $\St_x$ for $x \in X$ is set-theoretic and thus can be checked on $X_{\red}$.
			
			Assume on the contrary that $\ker (G \to \Aut_{X_{\red}})$ is a not affine, and let $H$ be its anti-affine part. Since every point $x \in X$ is fixed by $H$, the local ring $\mathcal O_{X, x}$ admits an $H$-action. Writing $\mathfrak m$ for its maximal ideal, $H$ acts trivially on $\mathcal O_{X,x} / \mathfrak m^i$ for $i \ge 0$ again by \cite[Lemma~1.1]{brion09}. Since $\bigcap_i \mathfrak m^i = 0$, this means that $H$ acts trivially on $\mathcal O_{X, x}$ for every point $x \in X$. Hence $H$ acts on $X$ trivially, violating the fact that the kernel of the $G$-action is affine.
		\end{proof}
		
		\begin{lemma} \label{lem:automorphism of chain of rational curves}
			Let $C$ be a proper curve whose irreducible components, with reduced induced structures, are rational curves. Then $\Aut_C$ is affine.
		\end{lemma}
		\begin{proof}
			We claim that every connected algebraic group $G$ acting faithfully on $C$ is affine. Take an induced $G$-action on the normalization $C^\nu = \bigsqcup_{i=1}^r \PP^1$. This constructs a morphism $G \to \PGL_2^{\times r}$, so it suffices to show that its kernel is affine. That is, we may assume $G$ acts on $C^\nu$ trivially, or on $C_{\red}$ trivially.
			
			Now we have a faithful $G$-action on $C$ acting trivially on $C_{\red}$. We can use Brion's method as in \Cref{lem:linear stabilizer}. Assume on the contrary that $G$ is not affine, and take its anti-affine part $H \subset G$. Then every point $x \in C$ is fixed by $H$ so we have an $H$-action on $\mathcal O_{C, x}$. By \cite[Lemma~1.1]{brion09}, $H$ acts trivially on the quotient $\mathcal O_{C, x}/\mathfrak m^i$ for every $i$, whence acts trivially on the entire $\mathcal O_{C,x}$. This shows $H$ acts on $C$ trivially, violating the assumption that $G$ acts faithfully on $C$.
		\end{proof}

		\subsubsection{Picard schemes}
			\begin{lemma} \label{lem:Picard scheme of free quotient}
				Let $G$ be a finite abelian group acting freely on a projective scheme $X$ over $k$ with $h^0 (X, \mathcal O_X) = 1$, and $q : X \to Z = X/G$ its quotient. Then $q^* : \Pic_Z \to \Pic_X$ has a kernel $\hat G = \Hom (G, \GG_m)$.
			\end{lemma}
			\begin{proof}
				The stacky quotient of the structure morphism $X \to \Spec k$ is a morphism $f : Z \to BG = [\Spec k/G]$. The Leary spectral sequence applied to $f$ is:
				\[ E_2^{p,q} = H^p (G, H^q (X, \GG_m)) \quad \Longrightarrow \quad H^{p+q} (Z, \GG_m) .\]
				This yields the desired left exact sequence $0 \to H^1 (G, \GG_m) \to H^1 (Z, \GG_m) \to H^1 (X, \GG_m)$. Note that $H^0 (X, \mathcal O_X) = k$ implies $H^0 (X, \GG_m) = \GG_m$ (an fppf sheaf on $BG$), and $H^1 (G, \GG_m) = \Hom (G, \GG_m) = \hat G$.
			\end{proof}
			
			The following is a generalization of \cite[Tag~0C6S]{stacks-project} for proper schemes of characteristic $0$.
			
			\begin{lemma} \label{lem:Picard scheme of thickening}
				Let $Y$ be a proper scheme over $k$ and $i : Z \hookrightarrow Y$ a closed subscheme defined by a nilpotent ideal. Then
				\begin{enumerate}
					\item $i^* : \Pic^\circ_Y \to \Pic^\circ_Z$ has a unipotent kernel and cokernel.
					\item $i^* : \NS(Y_{\bar k}) \to \NS(Z_{\bar k})$ is injective and has a torsion-free cokernel.
				\end{enumerate}
			\end{lemma}
			\begin{proof}
				We may assume $k$ is algebraically closed. All claims are stable under compositions $W \hookrightarrow Z \hookrightarrow Y$ of infinitesimal thickenings, so we may further assume $i$ is a first-order thickening, i.e., the ideal sheaf $I = I_Z$ satisfies $I^2 = 0$. Take a short exact sequence of fppf abelian sheaves $0 \to I \to \GG_{m, Y} \to i_* \GG_{m, Z} \to 0$ on $Y$, where the first homomorphism is defined by $f \mapsto 1 + f$ using the assumption $I^2 = 0$. Its long exact sequence
				\begin{equation} \label{eq:lem:Picard scheme of thickening}
					\begin{tikzcd}
						H^1 (Y, I) \arrow[r] & \Pic_Y \arrow[r, "i^*"] & \Pic_Z \arrow[r] & H^2 (Y, I)
					\end{tikzcd}
				\end{equation}
				immediately shows $\coker i^*$ is torsion-free. Taking $d$-torsions of \eqref{eq:lem:Picard scheme of thickening}, $i^* : \Pic_Y[d] \to \Pic_Z[d]$ is an isomorphism of $d$-torsion points for every $d \in \ZZ$ \cite[Tag~0C6S]{stacks-project}. Therefore, $\ker i^*$ is torsion-free as well.
				
				\medskip
				
				(Cokernels of (1) and (2)) Take a homomorphism between short exact sequences
				\[\begin{tikzcd}
					& \ker f \arrow[d] & \ker i^* \arrow[d] & \ker g \arrow[d] \\
					0 \arrow[r] & \Pic^\circ_Y \arrow[r] \arrow[d, "f"] & \Pic_Y \arrow[r] \arrow[d, "i^*"] & \NS(Y) \arrow[r] \arrow[d, "g"] & 0 \\
					0 \arrow[r] & \Pic^\circ_Z \arrow[r] \arrow[d] & \Pic_Z \arrow[r] \arrow[d] & \NS(Z) \arrow[r] \arrow[d] & 0 \\
					& \coker f & \coker i^* & \coker g
				\end{tikzcd}.\]
				Note that $\coker f$ is connected and $\coker g$ is discrete. Therefore, $\coker (\ker g \to \coker f)$ is the neutral component of $\coker i^*$. Since $\coker i^*$ was torsion-free, its neutral component $\coker (\ker g \to \coker f)$ is unipotent. But $\ker g$ is discrete, so this mans $\coker f$ is unipotent and $\ker g \to \coker f$ is the zero map. Moreover, the torsion-freeness of $\coker i^*$ implies that of $\coker g$ by \Cref{lem:pi_0 sequence splits}.
				
				\medskip
				
				(Kernels of (1) and (2)) Since $\ker i^*$ is torsion-free, so is $\ker f$. This proves the neutral component of $\ker f$ is unipotent. But $\ker f$ is of finite type and hence $\pi_0 (\ker f)$ is a finite group, so $\pi_0(\ker f)$ must be trivial and hence $\ker f$ unipotent. Now the short exact sequence $0 \to \ker f \to \ker i^* \to \ker g \to 0$ (uniquely) splits again by \Cref{lem:pi_0 sequence splits}. Since $\ker i^* = \ker f \times \ker g$ is torsion-free, $\ker g \cong \ZZ^r$ is torsion-free and $\ker i^* \cong \GG_a^{\times s} \times \ZZ^r$. Finally, recall from \eqref{eq:lem:Picard scheme of thickening} that the $k$-point group $(\ker i^*)(k) = k^s \times \ZZ^r$ is a quotient of an abelian group $H^1(Y, I) \cong k^t$. Since $k^t$ is a divisible group, its quotient $k^s \times \ZZ^r$ is divisible as well. This proves $r = 0$, or $\ker g = 0$ as claimed.
			\end{proof}

	\subsection{Albanese morphisms}
		Let us make precise the version of the Albanese morphism we use in this article.
		
		\begin{definition} \label{def:Albanese}
			Let $X$ be a finite type scheme over $k$. The \emph{Albanese morphism} of $X$ is a morphism $\alb : X \to \Alb_X$ universal among morphisms from $X$ to an abelian torsor over $k$. The abelian torsor $\Alb_X$ is called the \emph{Albanese torsor} of $X$.
		\end{definition}
		
		Notice that this definition fixes no $k$-point of $X$ (as a result, $\Alb_X$ should be an abelian torsor instead of abelian variety). This subtle distinction is necessary when $k$ is not algebraically closed or when $X$ is equipped with a group action. It is well-known that every smooth proper connected variety $X$ has an Albanese morphism. The following is a vast generalization of this fact.
		
		\begin{theorem}[{\cite{con:albanese}, \cite{lau-sch24}}] \label{thm:Albanese}
			Let $X$ be a proper scheme over $k$ with $h^0 (X, \mathcal O_X) = 1$. Then
			\begin{enumerate}
				\item The Albanese morphism $\alb : X \to \Alb_X$ exists.
				\item The pullback under $\alb$ is an injective homomorphism
				\[ \alb^* : \Pic^{\circ}_{\Alb_X} = \Alb_X^\vee \hookrightarrow \Pic^{\circ}_X ,\]
				whose image is the maximal abelian subvariety of $\Pic^{\circ}_X$. In particular, $\Alb_X$ is a torsor under the dual of the maximal abelian subvariety of $\Pic^{\circ}_X$.
			\end{enumerate}
		\end{theorem}
		
		\begin{corollary} \label{cor:Albanese}
			Let $Y$ be a proper scheme over $k$ and $i : Z \hookrightarrow Y$ a closed subscheme defined by a nilpotent ideal. Assume $h^0 (Y, \mathcal O_Y) = h^0 (Z, \mathcal O_Z) = 1$. Then there exists a surjective morphism $\Alb_Z \twoheadrightarrow \Alb_Y$ and a commutative diagram
			\[\begin{tikzcd}
				Z \arrow[r, hook, "i"] \arrow[d, "\alb"] & Y \arrow[d, "\alb"] \\
				\Alb_Z \arrow[r, twoheadrightarrow] & \Alb_Y
			\end{tikzcd}.\]
		\end{corollary}
		\begin{proof}
			\Cref{thm:Albanese} applies to both $Y$ and $Z$. Hence it is enough to show $i^* : \Pic^\circ_Y \to \Pic^\circ_Z$ induces an injective homomorphism between their maximal abelian subvarieties. Let $A = \Alb_Y^\vee$ be the maximal abelian subvariety of $\Pic^\circ_Y$. Since $\ker i^*$ is unipotent by \Cref{lem:Picard scheme of thickening}, we have $A \cap \ker i^* = 0$, or $i^*$ injects $A$ into $\Pic^\circ_Z$.
		\end{proof}

	\subsection{Non-normal schemes via conductor subschemes} \label{sec:conductor}
		This subsection follows \cite[Tag~0ECH]{stacks-project} and \cite{sch12}. It is a generalization of the construction for singular curves in \cite[\S IV.1]{serre:class_field}.
		
		\begin{proposition}[{\cite[Tag~0E25]{stacks-project}}] \label{prop:pushout}
			Let $\tilde Y$ be a quasi-projective scheme over $k$, $\tilde Z \subset \tilde Y$ a closed subscheme, and $\tilde Z \to Z$ a finite morphism. Then there exists a pushout $Y = \tilde Y \amalg_{\tilde Z} Z$ in the category of schemes:
			\begin{equation} \label{diag:pushout}
			\begin{tikzcd}
				\tilde Z \arrow[r, hook] \arrow[d] & \tilde Y \arrow[d] \\
				Z \arrow[r, hook] & Y
			\end{tikzcd} .
			\end{equation}
			Moreover, $Z \subset Y$ is a closed subscheme, $\tilde Y \to Y$ is finite, and the diagram is also cartesian.
		\end{proposition}
		
		The diagram \eqref{diag:pushout} is a pushout in the category of topological spaces as well, so it really looks like contracting $\tilde Z$ to $Z$.
		
		\medskip
		
		Let us now start with a reduced quasi-projective scheme $Y$ over $k$ and take its normalization
		\[ \nu : \tilde Y \to Y .\]
		The homomorphism $\mathcal O_Y \to \nu_* \mathcal O_{\tilde Y}$ is injective because $Y$ is reduced. The scheme-theoretic support of $F = \nu_* \mathcal O_{\tilde Y}/\mathcal O_Y$ is a closed subscheme $Z \subset Y$, the \emph{conductor subscheme} of $Y$. Its ideal sheaf $\mathfrak c = I_Z \subset \mathcal O_Y$ is called the \emph{conductor ideal}, which is equivalently defined as
		\[ \mathfrak c = \{ f \in \mathcal O_{\tilde Y} : f \, \mathcal O_{\tilde Y} \subset \mathcal O_Y \} \subset \mathcal O_Y .\]
		One checks that $\mathfrak c \subset \mathcal O_{\tilde Y}$ is an ideal as well, so it defines a closed subscheme $\tilde Z = \nu^{-1}(Z) \subset \tilde Y$. We thus arrive at a cartesian diagram \eqref{diag:pushout}.
		
		\begin{proposition} [Reduced schemes from normal schemes \cite{sch14}] \label{prop:conductor pushout}
			Let $Y$ be a reduced quasi-projective scheme, $\nu : \tilde Y \to Y$ be its normalization, and $\tilde Z \subset \tilde Y$ and $Z \subset Y$ be the conductor subschemes as above. Then the cartesian diagram \eqref{diag:pushout} is a pushout diagram.
		\end{proposition}
		\begin{proof}
			The following is a copy of Schwede's proof. We may assume $Y = \Spec A$ for a reduced $k$-algebra $A$ of finite type. Writing $\bar A$ for the normalization of $A$ and $\mathfrak c \subset A \subset \bar A$ the conductor ideal, we need to show that the diagram
			\begin{equation} \label{diag:schwede}
			\begin{tikzcd}
				\bar A / \mathfrak c & \bar A \arrow[l] \\
				A / \mathfrak c \arrow[u, hook] & A \arrow[l] \arrow[u, hook]
			\end{tikzcd}
			\end{equation}
			is a pullback diagram of commutative rings. In other words, we need to show the ring homomorphism $A \to \{ (f, \bar h) \in \bar A \times A/\mathfrak c : \bar f = \bar h \in \bar A / \mathfrak c \}$ defined by $f \mapsto (f, \bar f)$ is an isomorphism. It is injective because the normalization $A \subset \bar A$ is injective. It is surjective because if $f-h \in \mathfrak c \subset A$ for $f \in \bar A$ and $h \in A$, then $f \in A$.
		\end{proof}
		
		In other words, every reduced scheme $Y$ is obtained by ``contracting'' its normalization $\tilde Y$ in a preferred way, using the notion of the conductor subscheme/ideal.

		\subsubsection{Equivariance}
			The above discussion about conductors is equivariant.
			
			\begin{lemma} \label{lem:quotient and normalization}
				Let $G$ be a finite group acting on a quasi-projective scheme $Y$ over $k$. Then taking normalization and quotient commute: $Y^\nu / G = (Y/G)^\nu$.
			\end{lemma}
			\begin{proof}
				The quasi-projectiveness is simply to guarantee the existence of a quotient as a scheme. Since the statement is local on $Y$ and $Y_{\red}/G = (Y/G)_{\red}$, we may assume $Y = \Spec A$ for a reduced $k$-algebra $A$ of finite type.
				
				Denote the total ring of fractions by $K = S^{-1} A$, where $S \subset A$ is the set of all nonzero-divisors, and the normalization of $A$ by $\bar A$. We have $A \subset \bar A \subset K$. The $G$-invariant of $K$ is the total ring of fractions of $A^G$ because $K^G = (S^{-1}A)^G = (S \cap A^G)^{-1} A^G$ by \cite[Exercise~5.12]{ati-mac:commutative}. We need to show $\overline{A^G} = \bar A^G$. The inclusion $(\subset)$ is clear. To show $(\supset)$, assume that $f \in \bar A$ is $G$-invariant. Take an identity $f^n + a_{n-1} f^{n-1} + \cdots + a_0 = 0$ for $a_i \in A$. Applying each $g \in G$ and adding them together, we have $|G| \cdot f^n + (\sum_g g \cdot a_{n-1}) f^{n-1} + \cdots + (\sum_g g \cdot a_0) = 0$. Since $k$ has characteristic $0$, we conclude $f \in \overline{A^G}$.
			\end{proof}
			
			\begin{lemma} \label{lem:quotient of conductor}
				Let $G$ be a finite group acting on a reduced quasi-projective scheme $Y$ over $k$. Consider the conductor subscheme diagram \eqref{diag:pushout} in \Cref{prop:conductor pushout}. Then the diagram is $G$-equivariant and its quotient is a pushout diagram as in \Cref{prop:pushout}.
			\end{lemma}
			\begin{proof}
				Recall the coherent sheaf $F = \nu_* \mathcal O_{\tilde Y}/\mathcal O_Y$. The $G$-action uniquely lifts to $\tilde Y$ so $F$ is a $G$-equivariant sheaf on $Y$. Hence its support $Z$ and its ideal sheaf $\mathfrak c$ are $G$-equivariant. This proves that the conductor subscheme diagram is equivariant.
				
				Let us now prove a quotient of an equivariant pushout diagram \eqref{diag:pushout} is again a pushout diagram. Start with a $G$-equivariant pullback diagram of commutative rings
				\[\begin{tikzcd}
					B/IB & B \arrow[l] \\
					A/I \arrow[u] & A \arrow[l] \arrow[u]
				\end{tikzcd},\]
				or an isomorphism $A \to \{ (f, \bar h) \in B \times A/I : \bar f = \bar h \in B/IB \}$. Taking its $G$-invariant, the composition
				\[ A^G \to \{ (f, \bar h) \in B^G \times A^G/I^G : \bar f = \bar h\} \subset \{ (f, \bar h) \in (B \times A/I)^G : \bar f = \bar h\} \]
				is an isomorphism. This proves that both the first and second maps are isomorphisms. That is, the $G$-invariant part of the diagram is cartesian. We do not know whether the $G$-quotient of the conductor diagram of $Y$ is again the conductor diagram of $Y/G$ (but it is still a pushout diagram).
			\end{proof}

\printbibliography
\end{document}